\documentclass[10pt, a4paper,oneside]{amsart}

\usepackage[a4paper,inner=2.5cm,outer=2cm,top=3cm,bottom=3cm,pdftex]{geometry}

\usepackage{amssymb,amsmath,bm}
\usepackage{amsopn}
\usepackage{enumerate}
\usepackage{mathrsfs}
\usepackage{color}

\newtheorem{theorem}{Theorem}[section]
\newtheorem{lemma}[theorem]{Lemma}

\newtheorem{proposition}[theorem]{Proposition}
\newtheorem{remark}[theorem]{Remark}

\newcommand\NN{{\mathbb N}}
\newcommand\RR{{{\mathbb R}}}

\newcommand{\eps}{\varepsilon}

\newcommand{\norm}[1]{\big\Vert#1\big\Vert}

\newcommand{\abs}[1]{\left\vert#1\right\vert}
\newcommand{\set}[1]{\left\{#1\right\}}

\newcommand{\inner}[1]{\left(#1\right)}

\newcommand{\comii}[1]{\left<#1\right>}
\newcommand{\com}[1]{\bigl[#1\bigr]}
\newcommand{\reff}[1]{(\ref{#1})}

\newcommand{\p}{\partial_x}

\numberwithin{equation}{section}
\newcommand{\bef}{\begin{proof}}
\newcommand{\enf}{\end{proof}}

\begin{document}

\title[Gevrey class smoothing effect for   Prandtl  equation]
{Gevrey class smoothing effect for  the Prandtl  equation}

\author[W. Li, D. Wu \& C.-J. Xu]
{Wei-Xi LI, Di WU and Chao-Jiang XU}
\date{}
\address{\noindent \textsc{Wei-Xi Li, School of Mathematics, Wuhan university 430072, Wuhan, P.R. China}}
\email{wei-xi.li@whu.edu.cn}
\address{\noindent \textsc{Di Wu, School of Mathematics, Wuhan university 430072, Wuhan, P.R. China}}
\email{wudi2530@whu.edu.cn}

\address{\noindent \textsc{Chao-Jiang Xu, Universit\'e de Rouen, CNRS UMR 6085, Laboratoire de Math\'ematiques, 76801 Saint-Etienne du Rouvray, France\\
and\\
School of Mathematics, Wuhan university 430072, Wuhan, P.R. China}}
\email{Chao-Jiang.Xu@univ-rouen.fr}


\begin{abstract}
  It is well known that the Prandtl boundary layer equation is instable, and the well-posedness in Sobolev space for the Cauchy problem is an open problem. Recently, under the Oleinik's monotonicity assumption for the initial datum, \cite{awxy} have proved the local well-posedness of Cauchy problem in Sobolev space (see also \cite{masmoudi-1}). In this work, we study the Gevrey smoothing effects of the local solution obtained in \cite{awxy}. We prove that the Sobolev's class solution   belongs to some Gevrey class with respect to tangential variables at any positive time. 
\end{abstract}

\maketitle

\section{Introduction}\label{s1}

In this work, we study the regularity of solutions to the Prandtl equation which is the foundation
of the boundary layer theory introduced by Prandtl in 1904,  \cite{prandtl}.  The inviscid limit   of an incompressible viscous flow with the non-slip boundary condition is still a challenging problem of mathematical analysis due to the appearance of a boundary layer,  where the tangential velocity adjusts rapidly from nonzero away from the boundary to zero on the boundary.  Prandtl equation describes the behavior of the flow near the boundary  in the small viscosity limit, and it reads
\begin{equation*}
\left\{\begin{array}{l} u_t + u u_x + vu_y + p_x
= u_{yy},\quad t>0,\quad x\in\RR,\quad y>0, \\
u_x +v_y =0, \\
u|_{y=0} = v|_{y=0} =0 , \ \lim\limits_{y\to+\infty} u =U(t,x), \\
u|_{t=0} =u_0 (x,y)\, ,
\end{array}\right.
\end{equation*}
where     $u(t,x,y)$ and $v(t,x,y)$ represent the tangential and normal
velocities of the boundary layer, with $y$ being the scaled normal variable
to the boundary, while $U(t,x)$ and $p(t,x)$ are the values on the
boundary of the tangential velocity and pressure of the outflow satisfying the Bernoulli law
\[
\partial_t U + U\partial_x U +\partial_x q=0.
\]
Because of  the degeneracy in tangential variable,  the well-posedness theories and the justification of the Prandtl's boundary layer theory remain as the challenging problems in the mathematical theory of fluid mechanics.   Up to now, there are only  a few  rigorous mathematical results (see \cite{e-1, GV-N, guo, hong-hunter, metivier} and referencesin). Under a monotonic assumption on the tangential
velocity of the outflow, Oleinik was the first to obtain
the local existence of classical solutions for the initial-boundary value problems, and this result together with some of her
works with collaborators were well presented in
the monograph \cite{oleinik-3}.  In addition
to Oleinik's monotonicity assumption on the velocity field, by imposing a so called favorable condition on the pressure, Xin-Zhang  \cite{xin-zhang} obtained the existence of global weak solutions to the Prandtl equation. All these well-posedness results were
 based on the Crocco transformation to overcome  the main difficulty caused by  degeneracy and mixed type of the equation.  Very recently the   well-posedness in the Sobolev space   was explored by virtue of energy method  instead of the Crocco transformation;  see   Alexandre et. all \cite{awxy}  and Masmoudi-Wong \cite{masmoudi-1}.  There is very few work concerned with the Prandtl equation without  the monotonicity assumption; we refer \cite{caf, cannone-2, cannone, DJ, Samm, zhangzhang} for the works in the analytic frame,  and \cite{GV-Ma,kmv} for the recent works  concerned with the existence in Gevrey class.  Recall Gevrey class, denoted by $G^s, s\geq 1$,  is an intermediate space between analytic functions  and $C^\infty$ space.  Given a domain $\Omega,$  the (global) Gevrey space $G^s(\Omega)$ is consist of such functions that $f\in C^\infty(\Omega)$ and that
 \begin{eqnarray*}
 	\norm{\partial^\alpha f}_{L^2(\Omega)}\leq L^{\abs\alpha+1}(\alpha !)^s
 \end{eqnarray*}
 for some constant $L$ independent of  $\alpha.$   The  significant difference between Gevery  ($s>1$) and analytic ($s=1$) classes is that there exist nontrivial  Gevrey functions admitting compact support.   

   We mention that due to the degeneracy in $x$,  it is natural  to expect Gevrey regularity  rather than analyticity for a subelliptic equation.   We refer \cite{clx3, clx1, clx2, DZ} for the link between subellipticity and Gevrey reguality.      In this paper we first study the intrinsic subelliptic structure due to the monotonicity condition, and then deduce,  basing on the subelliptic estimate,  the Gevrey smoothing effect;  that is,  given a monotonic initial data belonging to some Sobolev space,  the solution will lie in some Gevrey class at positive time,  just as like heat equation. It is different from  
   the Gevrey  propagation property obtained in the  aforementioned works,  where  the initial data is supposed to be of some Gevrey class,  for instance $G^{7/4}$ in \cite{GV-Ma}, and  the well-posedness is obtained in the same Gevrey space.  

Now we state our main result.  Without loss of generality, we  only consider  here the case of an uniform outflow $U=1$,  and the conclusion will still hold for Gevrey class  outflow $U$.   We mention that the Gevrey regularity for outflow $U$ is well developed (see \cite{kv} for instance).    For the uniform outflow,  we get  the constant pressure
$p$ due to the Bernoulli law.   Then the  Prandtl equation can be rewritten as
\begin{equation}\label{prandtl1}
\left\{\begin{array}{l} u_t + u u_x + vu_y- u_{yy}=0,\quad (t, x, y)\in ]0, T[\times\RR_+^2, \\
u_x +v_y =0, \\
u|_{y=0} = v|_{y=0} =0 , \ \lim\limits_{y\to+\infty} u =1, \\
u|_{t=0} =u_0 (x,y)\, ,
\end{array}\right.
\end{equation}

The main result concerned with the Gevrey class regularity can be stated as follows. 

\begin{theorem}\label{mainthm}
Let  $u(t,x,y)$ be a classical local in time  solution to Prandtl equation  \reff{prandtl1} on $[0, T]$  with  the properties subsequently listed below: 
\begin{enumerate}[\quad (i)]
  \item There exist two constants  $C_{*}>1, \sigma>1/2$ such that for any $(t, x, y)\in [0, T]\times\RR_+^2$, 
  \begin{equation}\label{1.3} 
   \begin{array}{l}
  	  C_*^{-1} \comii y^{-\sigma}\leq \partial_y u (t, x, y) \leq C_* \comii y^{-\sigma}, \\
  	    \abs{\partial_y^2 u (t, x, y)}+\abs{\partial_y^3 u (t, x, y)} \leq C_* \comii y^{-\sigma-1}, 
  \end{array}  
  \end{equation}
 where $\comii y =(1+|y|^2)^{1/2}$.
 \item   There exists $c>0, C_0>0$ and integer $N_0\geq 7$ such that   
\begin{eqnarray}\label{1.5}
	\|e^{2cy} \partial_x u \|_{L^{\infty}\inner{[0,T];~H^{N_0}( \RR_+^2)}}+
	\|e^{2cy} \partial_x \partial_y u \|_{L^{2}\inner{[0,T];~H^{N_0}( \RR_+^2)}}\le C_0\,.
\end{eqnarray}
\end{enumerate}
Then  for any $0<T_1<T$, there exists a constant $L$,  such that  for any $0<t\le T_1$,
\begin{equation}\label{maiest}
\forall~m>1+N_0  ,\quad   \norm{ e^{\tilde cy}\partial_x^m
  u(t)}_{L^{2}( \RR_+^2)}\leq t^{-3 (m-N_0-1)}\,L^{m}\inner{m!}^{3(1+\sigma)},
\end{equation}
where $0<\tilde  c<c$. The constants $L$ depends only on $C_0,  T_1, C_*,  c, \tilde c$ and $\sigma$.  Therefore, the solution $u$  belongs to the Gevery class of index $3(1+\sigma)$ with respect to $x\in \mathbb{R}$ for any $0<t\le T_1$.
\end{theorem}

\begin{remark}{\em 
\item [1).]  Such a solution in the above theorem  exists,  for instance, suppose that the  initial data $u_0$ can be written as 
\begin{eqnarray*}
	u_0(x,y)=u_0^s(y)+\tilde u_0(x,y),
\end{eqnarray*}
where  $u^s_0$  is a function of $y$ but independent of $x$ such that $ C^{-1}\comii y^{-\sigma} \leq \partial_yu_0^s(y) \leq C  \comii y^{-\sigma} $ for some constant $C\geq 1$,  and $\tilde u_0 $ is a small perturbation such that  its weighted  Sobolev norm $\norm{e^{2cy}\tilde u_0}_{H^{2N_0+7}(\mathbb R_+^2)}$ is suitably small.   Then using the arguments in \cite{awxy},  we can obtain the desired solution with the properties listed in Theorem \ref{mainthm} fulfilled.   Precisely,  the solution $u(t, x,y)$  is a perturbation of a shear flow $u^s(t, y)$ such that property (i) in the above theorem holds for $u,$  and  moreover $e^{2cy} \inner{u-u^s} \in L^{\infty}\inner{[0,T];~H^{N_0+1}( \RR_+^2)}.$  

\smallskip
\item[2).] The well-posedness problem of Prandtl's equation depends crucially on the choice of the underlying function spaces, especially
on the regularity  in the tangential variable $x$.  If the initial datum is analytic in $x$,  then the local in time solution exists(c.f. \cite{cannone, Samm, zhangzhang}),  but the Cauchy problem is ill-posedness in Sobolev space for linear and non linear Prandtl equation (cf. see \cite{e-2, GV-D}).   Indeed, the main mathematical difficulty is the lack of control on the $x$ derivatives.  For example,  $v$  in \eqref{prandtl1}  may be written as $-\int^y_0 u_x (y')dy'$ by the divergence condition,  and here we lose  one derivatives in $x$-regularity.  The degeneracy  can't be balanced directly by any horizontal diffusion term,  so that the standard energy estimates do not apply to establish the existence of local solution. But the results in our main Theorem \ref{mainthm} shows that {\em the loss of derivative in tangential variable $x$ can be  partially compensated  via the monotonic condition}.     

\smallskip
\item [3).] Under the hypothesis \eqref{1.3}, the   equation \eqref{prandtl1} 
is a non linear hypoelliptical equation of  H\"ormander type with a gain of regularity of order $\frac 13$ in $x$ variable (see Proposition \ref{subtang}), so that any $C^2$ solution is locally $C^\infty$, see \cite{Xu1,Xu2, Xu3};  for the corresponding linear operator,  \cite{DZ} obtained the regularity in the local Gevrey space $G^3$.  However, in this paper we study the equation \eqref{prandtl1} as a boundary layer equation, so that the local property of solution is not of interest to the physics application,  and our goal is then to study the global estimates in Gevrey class.  In view of  \eqref{1.3} we see  $u_y$ decays polynomially  at infinite,   so we only have a weighted subelliptic estimate  (see Proposition \ref{subtang}).  This explains why the Gevrey index,  which  is $3(1+\sigma)$,    depends also on the decay index $\sigma$  in \eqref{1.3}. 
  
\smallskip
\item[4).] Finally, the  estimate \eqref{maiest} gives an explicit Gevrey norm of solutions for the Cauchy problem with respect to $t>0$ when the initial datum is only in some finite order Sobolev space. Since the Gevrey class is an intermediate space between analytic space and Sobolev space,  the qualitative study of  solutions   in Gevery class can help us to understand the Prandtl boundary layer theory which has been justified in analytic frame.
}
\end{remark}

The paper is organized as follows.  In Section \ref{section2} we  prove Theorem \ref{mainthm},  and state some preliminaries  lemmas used in the proof.  The other sections are occupied by the proof of the preliminaries lemmas. Precisely,  we prove in Section \ref{section3}  a subelliptic estimate  for the linearized Prandtl operator.   Section \ref{section4} and Section \ref{section5} are devoted to presenting a crucial estimate for an auxilliary function and non linear terms.    The last section is an appendix, where the equation fulfilled by the auxilliary function is deduced.

\section{Proof for the main Theorem}\label{section2}
We will prove in this section the Gevery estimate \eqref{maiest} by induction on $m$.   As in \cite{masmoudi-1}, we consider the following auxilliary function   
\begin{eqnarray}
\label{fmfun+}
f_m=\partial_x^m \omega-\frac{\partial_y \omega}{\omega}\partial_x^m u=\omega\partial_y\inner{\frac{\partial_x^mu}{\omega}}, \quad m\geq 1,
\end{eqnarray}
where $\omega=\partial_y u>0$ and $u$  is a solution of equation \eqref{prandtl1} which satisfy the hypotheis \eqref{1.3}.  We also introduce the following inductive weight , 
\begin{eqnarray}\label{wmi+++++}
W^\ell_m=e^{2cy}\inner{1+\frac{2cy}{(3m+\ell)\sigma} }^{-\frac{(3m+\ell)\sigma}{2}}
\inner{1+cy}^{-1} \Lambda^{{\ell\over 3}}, \quad 0\le \ell\le 3,\,\,\,m\in\NN,\,\, y>0,
\end{eqnarray}
where  $\Lambda^d=\Lambda^d_x$ is the Fourier multiplier of symbol $\comii{\xi}^d $ with respect to $x\in\RR$.  Notting  
\begin{equation}\label{2.3}
W_m^0\geq e^{cy}\inner{1+cy}^{-1}\geq c_0e^{\tilde cy},
\end{equation}
for $0<\tilde c<c.$

Since 
$$
\left|\frac{\partial_y \omega}{\omega}\right|\le C^2_* \comii y^{-1},
$$
we have that , if $u$ is smooth, 
\begin{equation*}
\|W^0_m f_m\|_{L^2(\RR^2_+)}\le 	\|W^0_m \partial^m_x \omega\|_{L^2(\RR^2_+)}
+C^2_*\|W^0_m\comii y^{-1} \partial^m_x u\|_{L^2(\RR^2_+)}.
\end{equation*}
On the other hand, we have the following Poincar\'e type inequality.

\begin{lemma} \label{lemma2.1}
There exist $C_1, \tilde {C_1}>0$ independents of $m\ge 1, 0\le \ell\le 3$, such that 
\begin{eqnarray}\label{equ1+a}
 \norm{  \comii y^{-1}W^\ell_m \p^{m} u}_{L^2(\RR^2_+)}+\norm{  \comii y^{-1}W^\ell_m \partial^{m}_x\omega}_{L^2(\RR^2_+)} \leq C_1\big\|  W^\ell_m  f_m \big\|_{L^2(\RR^2_+)}.
 \end{eqnarray}
As a result,
 \begin{eqnarray}\label{++equ1+a}
\big\|  \Lambda^{-1} W^0_{m}  f_{m+1} \big\|_{L^2(\RR^2_+)}\leq \tilde {C_1} \norm{   W^0_{m} f_{m}}_{L^2(\RR^2_+)},
 \end{eqnarray}
 and 
  \begin{eqnarray*}
\big\|  \Lambda^{-1} \partial_y  W^0_{m}  f_{m+1} \big\|_{L^2(\RR^2_+)}\leq \tilde {C_1}\inner{ \norm{ \partial_y  W^0_{m} f_{m}}_{L^2(\RR^2_+)}+\big\|     W^0_{m}  f_{m} \big\|_{L^2(\RR^2_+)}}. 
 \end{eqnarray*}
\end{lemma}
We will prove the above lemma in the section \ref{section4} as Lemma \ref{lemma4.2}.

\bigskip
Since the initial datum of the equation \eqref{prandtl1} is only in  Sobolev space $H^{N_0+1}$,   we have to introduce the following cut-off function, with respect to $0\leq  t\le T\leq1$,    to study the Gevrey smoothing effect by using the hyopelliticity, 
\begin{eqnarray}\label{pmi++++++}
\phi^\ell_m=\phi^{3(m-(N_0+1))+\ell}=(t(T-t))^{3(m-(N_0+1))+\ell},\quad  m\ge N_0+1,\,\, 0\le \ell\le 3.
\end{eqnarray} 
We will prove  by induction an energy estimate for the function $\phi^0_mW_m^0 f_m.$   For this purpose    
we need the following lemma concerned with the link between  $\phi^0_{m+1}W_{m+1}^0 f_{m+1}$ and   $\phi^3_mW_m^3 f_m$,   whose proof is postponed to the section \ref{section4} as Lemma \ref{equ+} and Lemma \ref{E1}.

\begin{lemma}\label{le+2.2+}
  There exists a constant $C_2$,  depending only on the numbers $\sigma,$ $c$ and  the constant $C_*$ in Theorem \ref{mainthm}, in particular, independents on $m$ , such that for any $m\geq N_0+1$,
\begin{eqnarray*}
	&&\norm{  \phi_{m+1}^0
W^0_{m+1} f_{m+1}}_{L^\infty( [0,T];~L^2( \mathbb
  R_+^2))}+\sum_{j=1}^2\norm{\partial_y^j\Lambda^{-\frac{2(j-1)}{3}} \phi_{m+1}^0
W^0_{m+1} f_{m+1}}_{L^2( [0,T]\times \mathbb
  R_+^2)}  \\
 & \leq& ~C_2\norm{    \phi_{m}^{3}
W^{3}_{m}  f_{m}}_{L^\infty( [0,T];~ L^2(\mathbb
  R_+^2))}+C_2\sum_{j=1}^2\norm{\partial_y^j\Lambda^{-\frac{2(j-1)}{3}}   \phi_{m}^{3}
W^{3}_{m}  f_{m}}_{L^2( [0,T]\times \mathbb
  R_+^2)},
\end{eqnarray*}
and
\begin{eqnarray*} 
	&& \norm{\partial_y^3 \Lambda^{-1} \phi_{m+1}^0
W^0_{m+1} f_{m+1}}_{L^2( [0,T]\times \mathbb
  R_+^2)} \\
 & \leq& ~C_2\norm{    \phi_{m}^{3}
W^{3}_{m}  f_{m}}_{L^\infty( [0,T];~ L^2(\mathbb
  R_+^2))} +C_2\sum_{j=1}^2 \norm{\partial_y^j \Lambda^{-\frac{2(j-1)}{3}} \phi_{m}^3
W^{3}_{m} f_{m}}_{L^2( [0,T]\times \mathbb
  R_+^2)}\\
  &&+C_2 \norm{\partial_y^3\Lambda^{-1} \phi_{m}^3
W^{3}_{m} f_{m}}_{L^2( [0,T]\times \mathbb
  R_+^2)},
\end{eqnarray*}
and
\begin{eqnarray*}
		 \norm{\comii y^{-\sigma/2} \partial_y\Lambda^{1/3}\Lambda^{-2}_\delta W_m^{\ell-1}f_m}_{L^2(\mathbb R_+^2)}\leq  C_2\norm{  \partial_y \Lambda^{-2}_\delta W_m^{\ell}f_m}_{L^2(\mathbb R_+^2)}
+C_2 \norm{  \Lambda^{-2}_\delta W_m^{\ell}f_m}_{L^2(\mathbb R_+^2)}.
	\end{eqnarray*}
\end{lemma}

\bigskip
Now we prove Theorem \ref{mainthm} by induction on the estimate of $\phi^0_mW^0_mf_m$.  The procedure of induction is as follows.

\smallskip
\noindent{\bf Initial hypothesis of the induction.} From the hypothesis \eqref{1.3} and \reff{1.5} of Theorem 
\ref{mainthm}, we have firstly, in view of \reff{fmfun+}, 
 \begin{equation} \label{15042225}
0\leq m \le N_0+1,\quad  \norm{e^{2cy}  f_m}_{L^\infty\inner{[0,T];~L^2(\RR^2_+)}}+ \sum^3_{i=1}\norm{e^{2cy}\partial^i_y  f_m}_{L^2\inner{[0,T]\times \RR^2_+}}<C_0.
\end{equation}

\smallskip
\noindent{\bf Hypothesis of the induction.}
Suppose that there exists $A>C_0+1$ such that, for some  $m\ge N_0+1$ and  for any $N_0+1\leq k\leq m$, we have  
\begin{equation}\label{2.6-}
\partial^3_y \Lambda^{-1} \phi^0_k W^0_k  f_k\in 	L^{2}([0,T]\times \RR^2_+),
\end{equation}
\begin{equation}\label{2.6}
\| \phi^0_k W^0_k  f_k\|_{L^\infty\inner{[0,T];~L^{2}(\RR^2_+)}}
+\sum^2_{j=1}\|\partial^j_y\Lambda^{-\frac{2(j-1)}{3}} \phi^0_k  W^0_kf_k\|_{L^{2}([0,T]\times \RR^2_+)} \leq A^{k-5 }\inner{(k-5)! }^{3(1+\sigma)}.
\end{equation}

\smallskip
\noindent {\bf Claim $\bm{I_{m+1}}$:}  we claim that \eqref{2.6-} and \eqref{2.6} are also true for $m+1$.    As a result,  \eqref{2.6-} and \eqref{2.6} hold for all $k\geq N_0+1$ by induction.

\smallskip
\begin{proof}[{\bf Completeness of the proof for Theorem \ref{mainthm}}].
  
Before proving the above {\bf Claim $\bm{I_{m+1}}$},  we remark    that  Theorem  \ref{mainthm}  is just its immediate consequence.   Indeed,  induction processus imply that for any $m>1+N_0$, we have for any $0<t<T$,
\begin{eqnarray*}
   \big\|  \phi_{m}^0 W_{m}^0  f_m (t) \big\|_{ L^2(\RR^2_+)}\leq  A^{m-5 }\inner{(m-5)! }^{3(1+\sigma)}\leq  A^{m}\inner{m! }^{3(1+\sigma)},
 \end{eqnarray*}
then with \reff{wmi+++++},  \eqref{2.3}, \eqref{equ1+a} and  \reff{pmi++++++}, we get 
 \begin{eqnarray*}
  \forall~0<t\leq T_1<T\leq 1,\quad	t^{3(m-N_0-1)} \norm{ e^{\tilde cy}   \p^{m} u}_{ L^2(\RR^2_+)} \leq \inner{T-T_1}^{-3(m-N_0-1)} \big\|  \phi_{m}^0 W_{m}^0  f_m \big\|_{ L^2(\RR^2_+)},
\end{eqnarray*}
yields,  for any   $m >N_0+1$, 
  \begin{eqnarray*}
  \forall~0<t\leq T_1<T\leq 1,\qquad	t^{3(m-N_0-1)} \norm{ e^{\tilde cy}   \p^{m} u}_{ L^2(\RR^2_+)} &\leq& \inner{T-T_1}^{-3(m-N_0-1)} A^{m}\inner{m! }^{3(1+\sigma)}\\
  &\leq& \inner{T-T_1}^{-3m} A^{m}\inner{m! }^{3(1+\sigma)}. 
  \end{eqnarray*}
As a result,    Theorem  \ref{mainthm} follows if we take $L=(T-T_1)^{-3}A$.     
\end{proof}

\medskip
Now we begin to prove  {\bf Claim $\bm{I_{m+1}}$},  and to do so it is sufficient to prove that the following:
 
\indent 
{\bf Claim $\bm{E_{m,\ell}}, 0\le \ell\le 3$: }  The following property hold for $0\le \ell\le 3$,
\begin{eqnarray}
&\partial^3_y \Lambda^{-1} \phi^\ell_m W^\ell_m  f_m\in 	L^{2}([0,T]\times \RR^2_+),\nonumber\\
&\| \phi^\ell_m W^\ell_m  f_m\|_{L^\infty\inner{[0,T];~L^{2}(\RR^2_+)}}
+\sum^2_{j=1}\|\partial^j_y\Lambda^{-\frac{2(j-1)}{3}} \phi^\ell_m  W^\ell_m   f_m\|_{L^{2}([0,T]\times \RR^2_+)}\label{2.6+}\\
 &\leq A^{m-5 +\frac \ell 6}\inner{(m-5)! }^{3(1+\sigma)}(m-4)^{\ell(1+\sigma)}.\nonumber
\end{eqnarray}
In fact,   {\bf Claim $\bm{E_{m,3}}$}  yields   $\partial^3_y \Lambda^{-1} \phi^3_m W^3_m  f_m\in 	L^{2}([0,T]\times \RR^2_+)$ and 
\begin{eqnarray*}
	&&\| \phi^3_m W^3_m  f_m\|_{L^\infty\inner{[0,T];~L^{2}(\RR^2_+)}}
+\sum^2_{j=1}\|\partial^j_y\Lambda^{-\frac{2(j-1)}{3}} \phi^3_m  W^3_m   f_m\|_{L^{2}([0,T]\times \RR^2_+)}\\
& \leq& A^{m-5 +\frac 1 2}\inner{(m-5)! }^{3(1+\sigma)}(m-4)^{3(1+\sigma)}\\
 &= & A^{m-5 +\frac 1 2}\com{\big((m+1)-5\big)! }^{3(1+\sigma)},
\end{eqnarray*}
which, along with Lemma \ref{le+2.2+},  yields   $\partial^3_y \Lambda^{-1} \phi^0_{m+1} W^0_{m+1}  f_{m+1}\in 	L^{2}([0,T]\times \RR^2_+)$ and 
\begin{eqnarray*}
	&&\| \phi^0_{m+1} W^0_{m+1}  f_{m+1}\|_{L^\infty\inner{[0,T];~L^{2}(\RR^2_+)}}
+\sum^2_{j=1}\|\partial^j_y\Lambda^{-\frac{2(j-1)}{3}} \phi^0_{m+1}  W^0_{m+1}   f_{m+1}\|_{L^{2}([0,T]\times \RR^2_+)}\\
 &\leq&  C_2A^{m-5 +\frac 1 2}\com{\big((m+1)-5\big)! }^{3(1+\sigma)},
\end{eqnarray*}
 recalling  $C_2$ is a constant depending only on the numbers $\sigma,$ $c$ and the constants  $C_0, C_*$ in Theorem \ref{mainthm}.   As a result,  if we choose $A$ in such a way that
 \begin{eqnarray*}
 	A^{1/2}\geq C_2,
 \end{eqnarray*} 
 then we see \reff{2.6} is also valid  for $k=m+1$.   Thus the desired   {\bf Claim $\bm{I_{m+1}}$} follows.  

 \bigskip

\noindent
{\bf Proof of the Claim $\bm{E_{m,\ell}}$ }. 

The rest of this section is devoted to proving Claim $\bm{E_{m,\ell}}$ holds for all $ {0\leq \ell\leq 3}$,  supposing   the inductive hypothesis \reff{2.6-} and \reff{2.6-} hold.   
 
 We will prove {\bf Claim $\bm{E_{m,\ell}}$} by iteration on $0\le \ell\le 3$.   Obviously  {\bf Claim $\bm{E_{m,0}}$} holds, due to the hypothesis of  induction  \reff{2.6-} and    \reff{2.6} with $k=m$.  Now supposing {\bf Claim $\bm{E_{m,i}}$} holds for all $0\leq i\leq \ell-1$, i.e.,  for  all $0\le i\le \ell-1$ we have 
\begin{eqnarray}
&\partial^3_y \Lambda^{-1} \phi^i_m W^i_m  f_m\in 	L^{2}([0,T]\times \RR^2_+),\nonumber\\
&\| \phi^i_m W^i_m  f_m\|_{L^\infty\inner{[0,T];~L^{2}(\RR^2_+)}}
+\sum^2_{j=1}\|\partial^j_y\Lambda^{-\frac{2(j-1)}{3}} \phi^i_m  W^i_m   f_m\|_{L^{2}([0,T]\times \RR^2_+)}\label{+ini}\\
& \leq A^{m-5 +\frac i 6}\inner{(m-5)! }^{3(1+\sigma)}(m-4)^{i(1+\sigma)},\nonumber
\end{eqnarray}
we will prove in the remaining part {\bf Claim $\bm{E_{m,\ell}}$} also holds.  To do so, we first introduce the mollifier $\Lambda_\delta^{-2}=\Lambda_{\delta, x}^{-2}$ which is the Fourier multiplier with the symbol $\comii{\delta \xi}^{-2}$,  $0<\delta<1$, and then consider the function  $F=\Lambda_\delta^{-2} \phi^{\ell}_{m}W_{m}^\ell f_{m}$. Under the inductive assumption \reff{+ini}, we see  $F$ is a classical solution to the following problem ( See the detail computation in Section \ref{apen} and the equation \reff{eqnfms} fulfilled by $f_m$ ): 
  \begin{eqnarray}\label{2.16}
\left\{
\begin{array}{l}
\inner{\partial_t +u\p +v\partial_y-\partial^2_y}F=\mathcal{Z}_{m,\ell, \delta},\\
	\partial_y F\big|_{y=0}=0,\\
	 F\big|_{t=0}=0,
\end{array}
\right.  
\end{eqnarray}
where
\begin{eqnarray}\label{2.18a}
\mathcal{Z}_{m,\ell,\delta}= \Lambda_\delta^{-2} \phi^{\ell}_{m}W^\ell_m
\mathcal{Z}_{m}+\Lambda_\delta^{-2} \left(\partial_t\phi^{\ell}_{m}\right)W^\ell_{m}f_{m}+\com{u \partial_x+v\partial_y-\partial_y^2, ~\Lambda_\delta^{-2}\phi^\ell_m W_m^\ell}  f_m ,
\end{eqnarray}
with   $\mathcal{Z}_{m}$  given in the appendix (seeing Section \ref{apen}),  that is, 
\begin{eqnarray*}
\mathcal{Z}_{m}&=&-\sum_{j=1}^m{m\choose j}(\p^j u) f_{m+1-j}
-\sum_{j=1}^{m-1}{m\choose j}(\p^j v)(\partial_y f_{m-1})\\
&&-\left[\partial_y\inner{\frac{\partial_y\omega}{\omega}}\right]\sum_{j=1}^{m-1}{m\choose j}
(\p^j v)(\p^{m-j}u)
-2\left[\partial_y\inner{\frac{\partial_y\omega}{\omega} }\right]f_m.
\end{eqnarray*}
 The initial value and boundary value in \eqref{2.16} is take in the sense of trace in Sobolev space,  due to the  induction hypothesis \eqref{2.6}  and the facts that $\partial_y \Lambda_\delta^{-2} \phi^{\ell}_{m}  f_{m}|_{y=0}=0$ (seeing \reff{15042201} in the appendix) and 
\begin{eqnarray*}
	\partial_y\bigg(e^{2cy}\inner{1+\frac{2cy}{(3m+i)\sigma} }^{-\inner{3m+i}\sigma/2}
(1+cy)^{-1}\bigg)\bigg|_{y=0}=0.
\end{eqnarray*}

We will prove    an energy estimate for the equation \eqref{2.16}.   For this purpose,   let $t\in[0,T]$, and take $L^2\inner{ [0,t]\times\mathbb R_+^2}$ inner product with  $F$ on both sides of the first equation in \reff{2.16};  this gives 
\begin{eqnarray*}
{\rm  Re}~\inner{ \inner{\partial_t +u\p +v\partial_y-\partial^2_y}F,~F}_{L^2\inner{[0,t]\times\mathbb R_+^2}}
  = {\rm  Re}~\inner{\mathcal{Z}_{m,\ell, \delta},~F}_{L^2\inner{[0,t]\times\mathbb R_+^2}}.
\end{eqnarray*}
Moreover observing  the initial-boundary conditions in \reff{2.16} and the facts that $u|_{y=0}=v|_{y=0}=0$ and  $\partial_x u+\partial_y v=0,$  we integrate by parts to obtain,  
\begin{eqnarray*}
{\rm  Re}~\inner{ \inner{\partial_t +u\p +v\partial_y-\partial^2_y}F,~F}_{L^2\inner{[0,t]\times\mathbb R_+^2}}= \frac{1}{2} \|F(t)\|_{L^2(\RR^2_+)}^2
+\int_0^t\| \partial_y F(t)\|^2_{L^2(\RR^2_+)}dt. 
\end{eqnarray*}
Thus we infer
\begin{eqnarray*}
	\|F\|_{L^\infty\inner{[0,T];~L^2(\RR^2_+)}}^2+\|\partial_y F \|_{L^2\inner{[0,T]\times \RR^2_+}}^2
	\leq 2 \abs{ \inner{\mathcal{Z}_{m,\ell, \delta},~F}_{L^2\inner{[0,T]\times\mathbb R_+^2}}},
\end{eqnarray*}
and thus 
\begin{eqnarray}\label{enest+}
\begin{split}
	&\|F\|_{L^\infty\inner{[0,T];~L^2(\RR^2_+)}}^2+\|\partial_y F \|_{L^2\inner{[0,T]\times \RR^2_+}}^2 +\|\partial_y^2\Lambda^{-2/3}F \|_{L^2\inner{[0,T]\times \RR^2_+}}^2 
\\	\leq & ~2 \abs{ \inner{\mathcal{Z}_{m,\ell, \delta},~F}_{L^2\inner{[0,T]\times\mathbb R_+^2}}}+\|\partial_y^2\Lambda^{-2/3}F \|_{L^2\inner{[0,T]\times \RR^2_+}}^2\\	\leq & ~2 \norm{\phi^{1/2}  \mathcal{Z}_{m,\ell, \delta}}_{L^2\inner{[0,T]\times\mathbb R_+^2}}\norm{\phi^{-1/2} F}_{L^2\inner{[0,T]\times\mathbb R_+^2}}+\|\partial_y^2\Lambda^{-2/3}F \|_{L^2\inner{[0,T]\times \RR^2_+}}^2. 
\end{split}
\end{eqnarray}
In order to treat the first term on the right hand side,  we need the following proposition, whose proof is postponed to Section \ref{section5}.
    
 \begin{proposition}\label{nonli}
Under the induction hypothesis \reff{15042225} -\eqref{2.6} and \eqref{+ini},  there exists a constant $C_3$,   such that,  using the notation   $F=\Lambda_\delta^{-2} \phi^{\ell}_{m} W^\ell_{m} f_{m}$ and $\tilde f=\phi^{1/2} \Lambda_\delta^{-2} \phi^{\ell-1}_{m} W^{\ell-1}_{m} f_{m}$ with $\phi$  defined in \reff{pmi++++++},
\begin{eqnarray}
&& \norm{  \phi^{1/2} \mathcal Z_{m,\ell, \delta}}_{L^2\inner{[0,T]\times\mathbb R_+^2}}\nonumber\\
 &\leq & m C _3 \norm{  \phi^{-1/2}  F} _{L^2\inner{[0,T]\times\mathbb R_+^2}}+C_3 \norm{  \partial_y F}_{L^2\inner{[0,T]\times\mathbb R_+^2}}+ C_3  A^{m-6}  \inner{(m-5)!}^{3(1+\sigma)}  \label{2.20},
  \end{eqnarray}
  and
  \begin{eqnarray}
&& \norm{ \Lambda^{-1/3}  \phi^{1/2}\mathcal Z_{m,\ell-1, \delta}}_{L^2\inner{[0,T]\times\mathbb R_+^2}}\nonumber\\
 &\leq & m C _3  \norm{ \Lambda^{-1/3} \phi^{-1/2}\Lambda_\delta^{-2} \phi^{\ell-1}_{m} W^{\ell-1}_{m} f_{m} } _{L^2\inner{[0,T]\times\mathbb R_+^2}}\label{+2.20++}\\
 &&+  C_3\norm{\partial_y \Lambda^{-1/3}   \Lambda_\delta^{-2}  \phi^{\ell-1}_{m} W^{\ell-1}_{m} f_{m} }_{L^2\inner{[0,T]\times\mathbb R_+^2}} + C_3  A^{m-6}  \inner{(m-5)!}^{3(1+\sigma)} \nonumber,
  \end{eqnarray}
 and
 \begin{eqnarray}
&& \norm{ \Lambda^{-\frac{2}{3}}\partial_y \phi^{1/2}\mathcal Z_{m,\ell-1, \delta}}_{L^2\inner{[0,T]\times\mathbb R_+^2}}\nonumber\\
 && \leq C_3 \| \comii y^{-\sigma}\Lambda^{1/3}\tilde f \|_{L^2([0,T]\times\RR^2_+)}+C_3\norm{ \partial_y^2 \Lambda^{-2/3}\tilde f}_{L^2\inner{[0,T]\times\mathbb R_+^2}}\label{2.21}\\
 &&\quad +m\,C_3\inner{ \norm{ \Lambda^{-2/3}  \phi^{\ell-1}_{m} W^{\ell-1}_{m} f_{m}}_{L^2\inner{[0,T]\times\mathbb R_+^2}}+\norm{ \Lambda^{-2/3}\phi^{-1/2}  \partial_y    \phi^{\ell-1}_{m} W^{\ell-1}_{m} f_{m} }_{L^2\inner{[0,T]\times\mathbb R_+^2}}}\nonumber\\
 &&\quad+ C_3 A^{m-6}  \inner{(m-5)!}^{3(1+\sigma)}.
 \nonumber
  \end{eqnarray} 
The constant $C_3$ depends only on $\sigma,$ $c$, and the constant $C_*$,  but is  independent of  $m$ and $\delta.$ 
\end{proposition}

 \smallskip
Now combining \reff{2.20} in the above proposition and \reff{enest+}, we have 
\begin{eqnarray*}
	&&\|F\|_{L^\infty\inner{[0,T];~L^2(\RR^2_+)}}^2+\sum_{j=1}^2\norm{\partial_y^j\Lambda^{-\frac{2(j-1)}{3}} F }_{L^2\inner{[0,T]\times \RR^2_+}}^2  \\
	&\leq & 2m C _3 \norm{  \phi^{-1/2} F} _{L^2\inner{[0,T]\times\mathbb R_+^2}}^2+(2C_3)^2 \norm{   \phi^{-1/2} F} _{L^2\inner{[0,T]\times\mathbb R_+^2}}^2+\frac{1}{2}\norm{   \partial_y F} _{L^2\inner{[0,T]\times\mathbb R_+^2}}^2\\
	&&  \inner{ A^{m-6 }  \inner{(m-5)!}^{3(1+\sigma)}  }^2+  \|\partial_y^2\Lambda^{-2/3} F\|_{L^2\inner{[0,T]\times \RR^2_+}}^2,
\end{eqnarray*}
which yields, 
 denoting by $C_4= 4C_3+10 C_3^2+2, $
 \begin{eqnarray*}
	&&\|F\|_{L^\infty\inner{[0,T];~L^2(\RR^2_+)}}^2+\sum_{j=1}^2\norm{\partial_y^j\Lambda^{-\frac{2(j-1)}{3}} F }_{L^2\inner{[0,T]\times \RR^2_+}}^2  \\
	&\leq &  m C _4 \norm{  \phi^{-1/2} F} _{L^2\inner{[0,T]\times\mathbb R_+^2}}^2  +  2 \inner{ A^{m-6 }  \inner{(m-5)!}^{3(1+\sigma)}  }^2+ 2 \|\partial_y^2\Lambda^{-2/3} F\|_{L^2\inner{[0,T]\times \RR^2_+}}^2,
\end{eqnarray*}
or equivalently, 
\begin{eqnarray}
	&&
\|F\|_{L^\infty\inner{[0,T];~L^2(\RR^2_+)}}+\sum_{j=1}^2\norm{\partial_y^j\Lambda^{-\frac{2(j-1)}{3}} F }_{L^2\inner{[0,T]\times \RR^2_+}}\nonumber  \\
	&&\leq C_4  \inner{m^{1/2}  \norm{  \phi^{-1/2} F} _{L^2\inner{[0,T]\times\mathbb R_+^2}}+ \|\partial_y^2\Lambda^{-2/3} F \|_{L^2\inner{[0,T]\times \RR^2_+}}}  +  2  A^{m-6 }  \inner{(m-5)!}^{3(1+\sigma)}  .\label{bigf}
\end{eqnarray}
It remains to treat the right terms on the right hand side. To do so we need to study  the subellipticity of the linearized Prandtl equation :
\begin{eqnarray}\label{eqf1}
\mathcal P f= \partial_t  f+u\partial_x f +v\partial_y f  -\partial_y^2
    f=h, &\quad (t,x,y)\in   ]0, T[\times \mathbb R_+^2,
\end{eqnarray}
where $u, v$ is solution of Prandtl's equation \eqref{prandtl1} satisfying the condition 
\eqref{1.3} and \eqref{1.5}. Then we have

\begin{proposition}\label{subtang}
Let $h, g\in L^2(  [0, T]\times \mathbb R_+^2)$ be given such that  $\partial_yh, \partial_y g \in L^2(  [0, T]\times \mathbb R_+^2)$.
Suppose that  $f\in L^2\inner{[0,T];~
  H^2(\mathbb R_+^2)}$  with $\partial_y^3f\in L^2\inner{[0,T]\times\mathbb R_+^2}$,  is  a classical solution to the equation \eqref{eqf1}
with the following initial and boundary conditions :
\begin{eqnarray}\label{ib}
f(0,x,y)=f(T,x,y)=0, \quad (x,y)\in\mathbb R_+^2,
\end{eqnarray}
and 
\begin{eqnarray}\label{bdc+}
\partial_y  f(t,x,0)=0,  \quad \partial_t f(t,x,0)=\inner{\partial_y^2  f}(t,x,0)+g(t,x,0),\quad  (t,x)\in ]0, T[\times \mathbb R.
\end{eqnarray}
Then for any $\eps>0$   there exists a constant $C_\eps$,  depending  only on   $\eps, \sigma$ and the  constants $C_*$ ,   such that 
\begin{eqnarray}
&& \norm{\comii y^{-\sigma/2}   \Lambda^{1/3}f}_{L^2(   [0, T]\times \RR_+^2)}+\norm{ \partial_y^2 \Lambda^{-1/3}f}_{L^2(   [0, T]\times \RR_+^2)}\nonumber \\
& \leq &  \eps   \norm{\Lambda^{-2/3}\partial_y h}_{L^2(\mathbb R^2)}   +
     C_\eps~   \Big(\norm{
       \Lambda^{-1/3} h }_{L^{2}(  [0, T]\times \RR_+^2)}+\norm{\partial_y f}_{L^{2}(  [0, T]\times \RR_+^2)}+\norm{ f}_{L^{2}(  [0, T]\times \mathbb R_+^2)}  \Big)\label{2.12}\\
&& +
     C_\eps ~  \Big( \norm{\comii y^{-{\sigma\over2}}  \partial_y \Lambda^{1/6}  f}_{L^2(   [0, T]\times \mathbb R_+^2)}+\norm{\comii y^\sigma\Lambda^{-1/3}  \partial_y g}_{L^2(   [0, T]\times \mathbb R_+^2)}\Big)\, .\nonumber
\end{eqnarray}
Moreover
\begin{eqnarray*}
&& \norm{ \partial_y^3 \Lambda^{-2/3}f}_{L^2(   [0, T]\times \RR_+^2)} \\
& \leq &      \tilde  C   \Big(\norm{\comii y^{-\sigma/2}   \Lambda^{1/3}f}_{L^2(   [0, T]\times \RR_+^2)}+ \norm{\Lambda^{-2/3}\partial_y h}_{L^2(\mathbb R^2)}  +\norm{\partial_y f}_{L^{2}(  [0, T]\times \RR_+^2)}+\norm{ f}_{L^{2}(  [0, T]\times \mathbb R_+^2)}  \Big),
\end{eqnarray*}
where $\tilde C$ is a constant depending only on $\sigma, c,$ and $C_*, C_0$ in Theorem \ref{mainthm}.
\end{proposition}
We will prove this proposition in next section \ref{section3}.  This subellipitic estimate gives a gain of regularity of order $\frac 13$ with respect to $x$ variable, so it is sufficient to repeat  the same procedure for $3$ 
times to get $1$ order of regularity.

 \bigskip

\noindent
{\bf Continuation of the proof of the Claim $\bm{E_{m,\ell}}$ }. 

We now use the above subellipticity for the   function $f=\tilde f$, with $\tilde f$ defined in Proposition \ref{nonli}, i.e.,
 $$f =\phi^{1/2}\Lambda_\delta^{-2}\phi^{\ell-1}_mW_{m}^{\ell-1} f_{m}=\Lambda_\delta^{-2}\phi^{3(m-N_0-1)+\ell-\frac{1}{2}}W_{m}^{\ell-1} f_{m}.$$  
Similar to \reff{2.16},  we see $f$ is a classical solution to the following problem:
  \begin{eqnarray*} 
\left\{
\begin{array}{l}
\inner{\partial_t +u\p +v\partial_y-\partial^2_y}f=\phi^{1/2}\mathcal{Z}_{m,\ell-1, \delta}+\inner{\partial_t\phi^{1/2}}\Lambda_\delta^{-2}\phi^{\ell-1}_mW_{m}^{\ell-1} f_{m}, \\
	\partial_y f\big|_{y=0}=0,\\
	 f\big|_{t=0}=f\big|_{t=T}=0,
\end{array}
\right.  
\end{eqnarray*}
where $\mathcal{Z}_{m,\ell-1,\delta}$
is defined in \reff{2.18a}. 
 The initial value and boundary value in \eqref{2.16} is take in the sense of trace in Sobolev space.   The validity of {\bf Claim $\bm{E_{m,\ell-1}}$} due to the inductive assumption \reff{+ini} yields that 
 $\partial_y^3 f\in L^2\inner{[0,T]\times\mathbb R_+^2}.$ 
 Next  we calculate $(\partial_t f-\partial_y^2 f)\big |_{y=0}$.   Firstly we have,  seeing \reff{15042202}  in the appendix,  
 \begin{eqnarray*}
\left(\partial_t f_m-\partial_y^2 f_m\right)|_{y=0}= -2 \left[\partial_y\inner{\frac{\partial_y\omega}{\omega} }\right]f_m \Big|_{y=0}. 
\end{eqnarray*}
Then
 \begin{eqnarray*}
 \partial_t f \big|_{y=0}  &=&\Lambda^{-2}_\delta\inner{\partial_t \phi^{3(m-N_0-1)+\ell-{1\over 2}}}W^{\ell-1}_m f_m\big|_{y=0} +\Lambda^{-2}_\delta\phi^{3(m-N_0-1)+\ell-{1\over 2}}W^{\ell-1}_m \partial_tf_m\big|_{y=0}\\
&=&\Lambda^{-2}_\delta\inner{\partial_t \phi^{3(m-N_0-1)+\ell-{1\over 2}}}W^{\ell-1}_mf_m\big|_{y=0} +\Lambda^{-2}_\delta\phi^{3(m-N_0-1)+\ell-{1\over 2}}W^{\ell-1}_m  \partial_y^2 f_m\big|_{y=0}\\
&&-2\Lambda^{-2}_\delta\phi^{3(m-N_0-1)+\ell-{1\over 2}}W^{\ell-1}_m \left[\partial_y\inner{\frac{\partial_y\omega}{\omega}}\right]f_m\Big|_{y=0}\\
&=&\Lambda^{-2}_\delta\inner{\partial_t \phi^{3(m-N_0-1)+\ell-{1\over 2}}}\Lambda^{(\ell-1)/3} f_m\big|_{y=0} +\Lambda^{-2}_\delta\phi^{3(m-N_0-1)+\ell-{1\over 2}}W^{\ell-1}_m  \partial_y^2 f_m\big|_{y=0}\\
&&-2\Lambda^{-2}_\delta\phi^{3(m-N_0-1)+\ell-{1\over 2}}\Lambda^{(\ell-1)/3} \left[\partial_y\inner{\frac{\partial_y\omega}{\omega}}\right]f_m\Big|_{y=0}.
\end{eqnarray*}
This, along with the fact that  
\begin{eqnarray*}
	 && \Lambda^{-2}_\delta\phi^{3(m-N_0-1)+\ell-{1\over 2}}W^{\ell-1}_m   \partial_y^2 f_m\big|_{y=0}\\
	 &=&\partial_y^2 \Lambda^{-2}_\delta\phi^{3(m-N_0-1)+\ell-{1\over 2}}W^{\ell-1}_m   f_m\big|_{y=0}-\com{\partial_y^2,~W^{\ell-1}_m}\Lambda^{-2}_\delta\phi^{3(m-N_0-1)+\ell-{1\over 2}}   f_m\big|_{y=0}\\
	  &=&\partial_y^2 f\big|_{y=0}-\inner{\frac{2c^2}{(3m+\ell-1)\sigma}+3c^2}\Lambda^{-2}_\delta\phi^{3(m-N_0-1)+\ell-{1\over 2}}\Lambda^{(\ell-1)/3} f_m\big|_{y=0}
\end{eqnarray*}
due to  the fact that $\partial_y \Lambda^{-2}_\delta f_m|_{y=0}=0$ (seeing \reff{15042201} in  the appendix),  gives
\begin{eqnarray*}
&&\inner{\partial_t f -\partial_y^2f }\big|_{y=0}\\
&=&\Lambda^{-2}_\delta\inner{\partial_t \phi^{3(m-N_0-1)+\ell-{1\over 2}}}\Lambda^{(\ell-1)/3} f_m\big|_{y=0}\\
&&  -\inner{\frac{2c^2}{(3m+\ell-1)\sigma}+3c^2}\Lambda^{-2}_\delta\phi^{3(m-N_0-1)+\ell-{1\over 2}}\Lambda^{(\ell-1)/3} f_m\big|_{y=0}\\
&&-2\Lambda^{-2}_\delta\phi^{3(m-N_0-1)+\ell-{1\over 2}}\Lambda^{(\ell-1)/3} \left[\partial_y\inner{\frac{\partial_y\omega}{\omega}}\right]f_m\Big|_{y=0}\\
&\stackrel{\rm def}{=}& g \big|_{y=0}
\end{eqnarray*}
with 
\begin{eqnarray}
	g &=&\Lambda^{-2}_\delta\inner{\partial_t \phi^{3(m-N_0-1)+\ell-{1\over 2}}}\Lambda^{(\ell-1)/3} f_m\nonumber \\
	&&  -\inner{\frac{2c^2}{(3m+\ell-1)\sigma}+3c^2}\Lambda^{-2}_\delta\phi^{3(m-N_0-1)+\ell-{1\over 2}}\Lambda^{(\ell-1)/3} f_m\label{gam} \\
&&-2\Lambda^{-2}_\delta\phi^{3(m-N_0-1)+\ell-{1\over 2}}\Lambda^{(\ell-1)/3} \left[\partial_y\inner{\frac{\partial_y\omega}{\omega}}\right]f_m. \nonumber
\end{eqnarray}
Then using Proposition  \ref{subtang} for $h=\phi^{1/2}\mathcal{Z}_{m,\ell-1, \delta}+\inner{\partial_t\phi^{1/2}}\Lambda_\delta^{-2}\phi^{\ell-1}_mW_{m}^{\ell-1} f_{m}$ and the above $g$, we have
\begin{eqnarray*}
&& \norm{\comii y^{-\sigma/2}   \Lambda^{1/3}f}_{L^2(   [0, T]\times \RR_+^2)}+\norm{ \partial_y^2 \Lambda^{-1/3}f}_{L^2(   [0, T]\times \RR_+^2)}  \\
& \leq &  \eps   \norm{\Lambda^{-2/3}\partial_y h}_{L^2(\mathbb R^2)}   +
     C_\eps~   \Big(\norm{
       \Lambda^{-1/3} h }_{L^{2}(  [0, T]\times \RR_+^2)}+\norm{\partial_y f}_{L^{2}(  [0, T]\times \RR_+^2)}+\norm{ f}_{L^{2}(  [0, T]\times \mathbb R_+^2)}  \Big) \\
&& +
     C_\eps ~  \Big( \norm{\comii y^{-{\sigma\over2}}  \partial_y \Lambda^{1/6}  f}_{L^2(   [0, T]\times \mathbb R_+^2)}+\norm{\comii y^\sigma\Lambda^{-1/3}  \partial_y g}_{L^2(   [0, T]\times \mathbb R_+^2)}\Big)\, . 
\end{eqnarray*}
{\bf We claim,} for any $\tilde \eps>0,$  
\begin{eqnarray}\label{claim++1}
	&& \eps   \norm{\Lambda^{-2/3}\partial_y h}_{L^2(\mathbb R^2)}   +
     C_\eps~   \Big(\norm{
       \Lambda^{-1/3} h }_{L^{2}(  [0, T]\times \RR_+^2)}+\norm{\partial_y f}_{L^{2}(  [0, T]\times \RR_+^2)}+\norm{ f}_{L^{2}(  [0, T]\times \mathbb R_+^2)}  \Big)\nonumber \\
& &\qquad\qquad+
     C_\eps ~  \Big( \norm{\comii y^{-{\sigma\over2}}  \partial_y \Lambda^{1/6}  f}_{L^2(   [0, T]\times \mathbb R_+^2)}+\norm{\comii y^\sigma\Lambda^{-1/3}  \partial_y g}_{L^2(   [0, T]\times \mathbb R_+^2)}\Big)\nonumber \\
   &  \leq &  \eps\, C_5 \inner{ \norm{\comii y^{-\sigma/2}    \Lambda^{1/3}f}_{L^2(   [0, T]\times \RR_+^2)}+\norm{ \partial_y^2 \Lambda^{-1/3}f}_{L^2(   [0, T]\times \RR_+^2)} }\\
     && +\tilde \eps\, m^{-(1+\sigma)/2}\inner{\norm{F}_{L^2(   [0, T]\times \RR_+^2)}+\norm{ \partial_yF}_{L^2(   [0, T]\times \RR_+^2)}}\nonumber \\
     &&+C_{\eps, \tilde\eps} m^{(1+\sigma)/2}  A^{m-5+\frac{\ell-1}{6}} \inner{(m-5)!}^sm^{(\ell-1)(1+\sigma)},\nonumber
\end{eqnarray}
where $C_5$ is a constant depending only on $\sigma,$ $c$, and the constant $C_*$,  but independent of  $m$ and $\delta,$ and $C_{\eps, \tilde\eps}$ is a constant depending only on $\eps,\tilde\eps, \sigma,$ $c$, and the constant $C_*$,  but independent of  $m$ and $\delta.$ Recall $F=\Lambda_\delta^{-2} \phi^{\ell}_{m}W_{m}^\ell f_{m}. $  The proof of \reff{claim++1} is postponed to the end of this section.   
Now combining the above inequalities and letting $\eps$ be small enough, we infer for any $\tilde \eps>0,$
\begin{eqnarray}
	&& \norm{\comii y^{-\sigma/2}   \Lambda^{1/3}f}_{L^2(   [0, T]\times \RR_+^2)}+\norm{ \partial_y^2 \Lambda^{-1/3}f}_{L^2(   [0, T]\times \RR_+^2)}\nonumber\\
	&\leq & \tilde \eps\, m^{-(1+\sigma)/2}\inner{\norm{F}_{L^2(   [0, T]\times \RR_+^2)}+\norm{ \partial_yF}_{L^2(   [0, T]\times \RR_+^2)}} \label{for225}\\
     &&+C_{\tilde\eps} m^{(1+\sigma)/2}  A^{m-5+\frac{\ell-1}{6}} \inner{(m-5)!}^sm^{(\ell-1)(1+\sigma)}.\nonumber
\end{eqnarray}

Now we come back to estimate the terms on the right side of \reff{bigf}. To do so  we need  the following technic Lemma, whose proof     is presented at the end of  Section \ref{section4}.

\begin{lemma}\label{lemma261}
	Recall $F=\Lambda_\delta^{-2} \phi^{\ell}_{m}W_{m}^\ell f_{m}$ and $f=\phi^{1/2}\Lambda_\delta^{-2} \phi^{\ell-1}_{m}W_{m}^{\ell-1} f_{m}$. 
	There exists a constant $C_6$,  depending only on $\sigma,$ $c$, and the constant $C_*$,  but independent of  $m$ and $\delta,$  such that 
	\begin{eqnarray*}
&&\| \phi^{-1/2} F\|_{L^2( [0,T]\times \mathbb
  R_+^2)}+\norm{\partial_y^2\Lambda^{-2/3}F}_{L^2( [0,T]\times \mathbb
  R_+^2)} \\
 & \leq & ~ C_6\inner{ m^{\sigma/2} \norm{\comii y ^{-\frac{\sigma}{2}}
\Lambda^{1/3}f }_{L^2( [0,T]\times \mathbb
  R_+^2)} +  \norm{\partial_y^2\Lambda^{-1/3} f}_{L^2( [0,T]\times \mathbb
  R_+^2)}}\\
  &&~+C_6\inner{ \|  \phi^{\ell-1}_{m}W_{m}^{\ell-1} f_{m}\|_{L^2( [0,T]\times \mathbb
  R_+^2)}+\norm{\partial_y  \phi^{\ell-1}_{m}W_{m}^{\ell-1} f_{m}}_{L^2( [0,T]\times \mathbb
  R_+^2)}},
    \end{eqnarray*}
  and
  \begin{eqnarray}\label{+upb+++}
\begin{split}
& \norm{\partial_y^3\Lambda^{-1}F}_{L^2( [0,T]\times \mathbb
  R_+^2)}\\
  \leq & C_6 \norm{\partial_y^2\Lambda^{-1/3} f}_{L^2( [0,T]\times \mathbb
  R_+^2)} +C_6\norm{\partial_y^3\Lambda^{-2/3}f}_{L^2( [0,T]\times \mathbb
  R_+^2)}\\
  &~+C_6 \inner{ \|  \phi^{\ell-1}_{m}W_{m}^{\ell-1} f_{m}\|_{L^2( [0,T]\times \mathbb
  R_+^2)}+\norm{\partial_y  \phi^{\ell-1}_{m}W_{m}^{\ell-1} f_{m}}_{L^2( [0,T]\times \mathbb
  R_+^2)}}.
\end{split}
  \end{eqnarray}
\end{lemma}


\noindent
{\bf End of the proof of the Claim $\bm{E_{m,\ell}}$ }. 

We combine  \reff{for225} and the first estimate in Lemma \ref{lemma261}, to conclude
\begin{eqnarray*}
	&&\| \phi^{-1/2} F\|_{L^2( [0,T]\times \mathbb
  R_+^2)}+\norm{\partial_y^2\Lambda^{-2/3}F}_{L^2( [0,T]\times \mathbb
  R_+^2)} \\
 & \leq & ~ \tilde\eps C_6  m^{-1/2} \inner{\norm{F}_{L^2( [0,T]\times \mathbb
  R_+^2)}+\norm{\partial_y F}_{L^2( [0,T]\times \mathbb
  R_+^2)}}  \\
  &&~+C_6 C_{\tilde\eps}  m^{\frac{1}{2}+\sigma}    A^{m-5+\frac{\ell-1}{6}} \inner{(m-5)!}^sm^{(\ell-1)(1+\sigma)}\\
  &&+ C_6\inner{ \|  \phi^{\ell-1}_{m}W_{m}^{\ell-1} f_{m}\|_{L^2( [0,T]\times \mathbb
  R_+^2)}+\norm{\partial_y  \phi^{\ell-1}_{m}W_{m}^{\ell-1} f_{m}}_{L^2( [0,T]\times \mathbb
  R_+^2)}}\\
 & \leq & ~ \tilde\eps C_6  m^{-1/2} \inner{\norm{F}_{L^2( [0,T]\times \mathbb
  R_+^2)}+\norm{\partial_y F}_{L^2( [0,T]\times \mathbb
  R_+^2)}}  \\
  &&~+\inner{C_6 C_{\tilde\eps} +C_6}  m^{\frac{1}{2}+\sigma}   A^{m-5+\frac{\ell-1}{6}}  \inner{(m-5)!}^{3(1+\sigma)} (m-4)^{(\ell-1)(1+\sigma)}\end{eqnarray*}
the last inequality using \reff{+ini}.  This along with \reff{bigf} yields 
\begin{eqnarray*}
	&&
\|F\|_{L^\infty\inner{[0,T];~L^2(\RR^2_+)}}+\sum_{j=1}^2\norm{\partial_y^j\Lambda^{-\frac{2(j-1)}{3}} F }_{L^2\inner{[0,T]\times \RR^2_+}}\nonumber  \\
	&&\leq \tilde\eps  C_4 C_6 \inner{   \norm{    F} _{L^2\inner{[0,T]\times\mathbb R_+^2}}+ \|\partial_y F \|_{L^2\inner{[0,T]\times \RR^2_+}}}\\
	&&\quad +C_4\inner{C_6 C_{\tilde\eps} +C_6}  m^{1+\sigma}   A^{m-5+\frac{\ell-1}{6}}  \inner{(m-5)!}^{3(1+\sigma)} (m-4)^{(\ell-1)(1+\sigma)}+  2 A^{m-6}  \inner{(m-5)!}^{3(1+\sigma)}. 
\end{eqnarray*}
Consequently, letting $\tilde \eps>0$ be small sufficiently,  
\begin{eqnarray*}
	&&
\|F\|_{L^\infty\inner{[0,T];~L^2(\RR^2_+)}}+\sum_{j=1}^2\norm{\partial_y^j\Lambda^{-\frac{2(j-1)}{3}} F }_{L^2\inner{[0,T]\times \RR^2_+}}\nonumber  \\
	&\leq & C_7 m^{1+\sigma} A^{m-5+\frac{\ell-1}{6}}  \inner{(m-5)!}^{3(1+\sigma)} (m-4)^{(\ell-1)(1+\sigma)}+  C_7  A^{m-6}  \inner{(m-5)!}^{3(1+\sigma)}\\
	&\leq & C_8 (m-4)^{1+\sigma} A^{m-5+\frac{\ell-1}{6}}  \inner{(m-5)!}^{3(1+\sigma)} (m-4)^{(\ell-1)(1+\sigma)}+  C_7  A^{m-6}  \inner{(m-5)!}^{3(1+\sigma)},
\end{eqnarray*}
where $C_7, C_8$ are two constants depending only on  $\sigma,$ $c$, and the constants $ C_0, C_*$ in Theorem \ref{mainthm},  but is  independent of  $m$ and $\delta.$   Now we choose $A$ such that 
\begin{eqnarray*}
	A\geq ( 2C_8 +2C_7+1)^6.
\end{eqnarray*}
It then follows that, observing $\ell\geq 1,$ 
\begin{eqnarray*}
	&&
\|F\|_{L^\infty\inner{[0,T];~L^2(\RR^2_+)}}+\sum_{j=1}^2\norm{\partial_y^j\Lambda^{-\frac{2(j-1)}{3}} F }_{L^2\inner{[0,T]\times \RR^2_+}}  \leq   A^{m-5+\frac{\ell }{6}}  \inner{(m-5)!}^{s} (m-4) ^{\ell(1+\sigma)}. 
\end{eqnarray*}
Observe the above constant $A$ is independent of $\delta$, and thus  letting $\delta\rightarrow 0,$ we see  \reff{+ini} holds for $i=\ell$.
 It remains to prove that $\partial_y^3\Lambda^{-1} \phi^{\ell}_{m}W_{m}^\ell f_{m}$.   The above estimate together with \reff{for225} gives
\begin{eqnarray*}
	\norm{\comii y^{-\sigma/2}   \Lambda^{1/3}f}_{L^2(   [0, T]\times \RR_+^2)}+\norm{ \partial_y^2 \Lambda^{-1/3}f}_{L^2(   [0, T]\times \RR_+^2)}<C_{m,1}
\end{eqnarray*} 
with $C_{m,1}
$ a constant depending on $m$ but independent of $\delta$, and thus, using the last estimate in Proposition  \ref{subtang} and \reff{claim++1},
\begin{eqnarray*}
	\norm{ \partial_y^3 \Lambda^{-2/3}f}_{L^2(   [0, T]\times \RR_+^2)} \leq C_{m,2},
\end{eqnarray*}
with $C_{m,2}
$ a constant depending on $m$ but independent of $\delta$.  As a result,  combining \reff{+upb+++}, we conclude
\begin{eqnarray*}
	\norm{\partial_y^3\Lambda^{-1}F}_{L^2( [0,T]\times \mathbb
  R_+^2)}<C_{m,3}
\end{eqnarray*}
with $C_{m,3}
$ a constant depending on $m$ but independent of $\delta$.   Thus letting $\delta\rightarrow 0, $ we see  $\partial^3_y \Lambda^{-1} \phi^\ell_m W^\ell_m  f_m\in 	L^{2}([0,T]\times \RR^2_+)$.   Thus {\bf Claim $\bm{E_{m,\ell}}$} holds.  This completes the proof of {\bf Claim $\bm{I_{m+1}}$}, and thus the proof of  Theorem \ref{mainthm}.

\bigskip

We end up this section by the following 

\begin{proof}[{\bf Proof of the estimate \reff{claim++1}}]  In the proof we use $C$ to denote different constants depending only on  $\sigma,$ $c$, and the constants $ C_0, C_*$ in Theorem \ref{mainthm},  but is  independent of  $m$ and $\delta.$ 

\medskip
{\it  (a)}  We first estimate $\norm{\Lambda^{-1/3}h}_{L^2\inner{[0,T]\times \mathbb R_+^2}},$ recalling $$h=\phi^{1/2}\mathcal{Z}_{m,\ell-1, \delta}+\inner{\partial_t\phi^{1/2}}\Lambda_\delta^{-2}\phi^{\ell-1}_mW_{m}^{\ell-1} f_{m}.$$
Using interpolation inequality gives, observing $\abs{\partial_t\phi^{1/2}}\leq \phi^{-1/2},$
\begin{eqnarray*}
	&&\norm{\Lambda^{-1/3} \inner{\partial_t\phi^{1/2}}\Lambda_\delta^{-2}\phi^{\ell-1}_mW_{m}^{\ell-1} f_{m}}_{L^2 (\mathbb R_x)} \\ 
	&\leq& m^{-1/2} \phi^{1/2}  \norm{  \inner{\partial_t\phi^{1/2}}\Lambda_\delta^{-2}\phi^{\ell-1}_mW_{m}^{\ell-1} f_{m}}_{L^2 (\mathbb R_x)} \\
	&& +m^{(\ell+1)/2}\phi^{-(\ell+1)/2}  \norm{\Lambda^{-1-\frac{\ell-1}{3}} \inner{\partial_t\phi^{1/2}}\Lambda_\delta^{-2}\phi^{\ell-1}_mW_{m}^{\ell-1} f_{m}}_{L^2 (\mathbb R_x)}\\ 
	&\leq&    m^{-1/2}  \norm{   \Lambda_\delta^{-2}\phi^{\ell-1}_mW_{m}^{\ell-1} f_{m}}_{L^2 (\mathbb R_x)} + m^{(\ell+1)/2} \norm{\Lambda^{-1-\frac{\ell-1}{3}} \phi^{-(\ell+2)/2}\Lambda_\delta^{-2}\phi^{\ell-1}_mW_{m}^{\ell-1} f_{m}}_{L^2 (\mathbb R_x)}\\ 
	&\leq&     m^{-1/2}  \norm{   \Lambda_\delta^{-2}\phi^{\ell-1}_mW_{m}^{\ell-1} f_{m}}_{L^2 (\mathbb R_x)} + m^{(\ell+1)/2} \norm{\Lambda^{-1}  \Lambda_\delta^{-2}\phi^{0}_{m-1}W_{m-1}^{0} f_{m}}_{L^2 (\mathbb R_x)},
\end{eqnarray*} 
the last inequality following from \reff{wmi+++++} which shows $W^0_i, i\geq 1, $ is a decreasing sequence of functions as $i$ varies in $\mathbb N$,  and the fact that 
\begin{eqnarray*}
	\phi^{-(\ell+2)/2}\phi^{\ell-1}_m\leq \phi^{0}_{m-1}. 
\end{eqnarray*}
Moreover,  using  \reff{++equ1+a} and the inductive assumptions \reff{+ini} and \reff{2.6}, we compute,   observing $\ell/2+1 \leq 3(1+\sigma),$ 
\begin{eqnarray*}
  &&m^{-1/2} \norm{   \Lambda_\delta^{-2}\phi^{\ell-1}_mW_{m}^{\ell-1} f_{m}}_{L^2 (\mathbb R_x)} + m^{(\ell+1)/2} \norm{\Lambda^{-1}  \Lambda_\delta^{-2}\phi^{0}_{m-1}W_{m-1}^{0} f_{m}}_{L^2\inner{[0,T]\times \mathbb R_+^2}} \\ 
	 	&\leq&   m^{-1/2}   \norm{   \phi^{\ell-1}_mW_{m}^{\ell-1} f_{m}}_{L^2 (\mathbb R_x)} +   \tilde {C_1} m^{(\ell+1)/2} \norm{   \phi^{0}_{m-1}W_{m-1}^{0} f_{m-1}}_{L^2\inner{[0,T]\times \mathbb R_+^2}} \\ 
	 	&\leq&   m^{-1/2}  A^{m-5 +\frac {\ell-1} {6}}\inner{(m-5)! }^{3(1+\sigma)}(m-4)^{(\ell-1)(1+\sigma)} +   \tilde {C_1} m^{(\ell+1)/2} A^{m-6}\inner{(m-6)! }^{3(1+\sigma)}\\ 
	 	&\leq&  C m^{-1/2}  A^{m-5 +\frac {\ell-1} {6}}\inner{(m-5)! }^{3(1+\sigma)}(m-4)^{(\ell-1)(1+\sigma)} .
\end{eqnarray*}
Thus we have, combining the above inequalities, 
\begin{eqnarray}
	&&\norm{\Lambda^{-1/3} \inner{\partial_t\phi^{1/2}}\Lambda_\delta^{-2}\phi^{\ell-1}_mW_{m}^{\ell-1} f_{m}}_{L^2\inner{[0,T]\times \mathbb R_+^2}}\nonumber\\
	& \leq& C m^{-1/2} A^{m-5 +\frac {\ell-1} {6}}\inner{(m-5)! }^{3(1+\sigma)}(m-4)^{(\ell-1)(1+\sigma)}\label{jklambda}.
\end{eqnarray}
Similarly, we can show that
\begin{eqnarray}
	&&\norm{\partial_y \Lambda^{-1/3}  \inner{\partial_t\phi^{1/2}}\Lambda_\delta^{-2}\phi^{\ell-1}_mW_{m}^{\ell-1} f_{m}}_{L^2\inner{[0,T]\times \mathbb R_+^2}}\nonumber\\
	& \leq& C m^{-1/2} A^{m-5 +\frac {\ell-1} {6}}\inner{(m-5)! }^{3(1+\sigma)}(m-4)^{(\ell-1)(1+\sigma)}\label{jklambda++}.
\end{eqnarray}
Using \reff{+2.20++} in Proposition \ref{nonli}, we have
\begin{eqnarray*}
	&& \norm{ \Lambda^{-1/3}  \phi^{1/2}\mathcal Z_{m,\ell-1, \delta}}_{L^2\inner{[0,T]\times\mathbb R_+^2}} \\
 &\leq & m C _3  \norm{ \Lambda^{-1/3} \phi^{-1/2}\Lambda_\delta^{-2} \phi^{\ell-1}_{m} W^{\ell-1}_{m} f_{m} } _{L^2\inner{[0,T]\times\mathbb R_+^2}}+ C_3 \norm{\partial_y \Lambda^{-1/3} \Lambda_\delta^{-2}  \phi^{\ell-1}_{m} W^{\ell-1}_{m} f_{m} }_{L^2\inner{[0,T]\times\mathbb R_+^2}} \\
 &&+ C_3  A^{m-6 }  \inner{(m-5)!}^{s},
\end{eqnarray*}
and moreover repeating the arguments as in \reff{jklambda} and  \reff{jklambda++},    with $\partial_t\phi^{1/2}$ there replaced by $\phi^{-1/2},$ 
\begin{eqnarray*}
	&& m C _3  \norm{ \Lambda^{-1/3} \phi^{-1/2}\Lambda_\delta^{-2} \phi^{\ell-1}_{m} W^{\ell-1}_{m} f_{m} } _{L^2\inner{[0,T]\times\mathbb R_+^2}}+ C_3 \norm{\partial_y \Lambda^{-1/3}\Lambda_\delta^{-2}  \phi^{\ell-1}_{m} W^{\ell-1}_{m} f_{m} }_{L^2\inner{[0,T]\times\mathbb R_+^2}} \\
	&\leq & C  m^{ 1/2} A^{m-5 +\frac {\ell-1} {6}}\inner{(m-5)! }^{3(1+\sigma)}(m-4)^{(\ell-1)(1+\sigma)},
\end{eqnarray*}
and thus
\begin{eqnarray*}
	 \norm{ \Lambda^{-1/3}  \phi^{1/2}\mathcal Z_{m,\ell-1, \delta}}_{L^2\inner{[0,T]\times\mathbb R_+^2}} 
 \leq C  m^{ 1/2} A^{m-5 +\frac {\ell-1} {6}}\inner{(m-5)! }^{3(1+\sigma)}(m-4)^{(\ell-1)(1+\sigma)}.
\end{eqnarray*}
This along with \reff{jklambda}   yields
\begin{eqnarray*}
	\norm{\Lambda^{-1/3}h}_{L^2\inner{[0,T]\times \mathbb R_+^2}}\leq C  m^{ 1/2} A^{m-5 +\frac {\ell-1} {6}}\inner{(m-5)! }^{3(1+\sigma)}(m-4)^{(\ell-1)(1+\sigma)}. 
\end{eqnarray*}

\medskip
 (b)  In this step we treat $\norm{\Lambda^{-2/3}\partial_y h}_{L^2\inner{[0,T]\times \mathbb R_+^2}}$.  It follows from    \reff{jklambda++} that 
 \begin{eqnarray*}
 	&&\norm{\Lambda^{-2/3} \partial_y   \inner{\partial_t\phi^{1/2}}\Lambda_\delta^{-2}\phi^{\ell-1}_mW_{m}^{\ell-1} f_{m}}_{L^2\inner{[0,T]\times \mathbb R_+^2}} \\
	& \leq& C m^{-1/2} A^{m-5 +\frac {\ell-1} {6}}\inner{(m-5)! }^{3(1+\sigma)}(m-4)^{(\ell-1)(1+\sigma)}.
 \end{eqnarray*}
 On the other hand, 
by  \reff{2.21} we have,  recalling $\tilde f=f=\phi^{1/2}\Lambda_\delta^{-2} \phi^{\ell-1}_{m}W_{m}^{\ell-1} f_{m},$
 \begin{eqnarray*}
 	&&\norm{ \Lambda^{-\frac{2}{3}}\partial_y \phi^{1/2}\mathcal Z_{m,\ell-1, \delta}}_{L^2\inner{[0,T]\times\mathbb R_+^2}} \\
 &\leq &C_3 \| \comii y^{-\sigma}\Lambda^{1/3} f \|_{L^2([0,T]\times\RR^2_+)}+C_3\norm{ \partial_y^2 \Lambda^{-2/3}f }_{L^2\inner{[0,T]\times\mathbb R_+^2}} \\
 &&+m\,C_3\inner{ \norm{ \Lambda^{-2/3}\phi^{\ell-1}_{m} W^{\ell-1}_{m} f_{m}}_{L^2\inner{[0,T]\times\mathbb R_+^2}}+\norm{ \Lambda^{-2/3} \phi^{-1/2}\partial_y  \phi^{\ell-1}_{m} W^{\ell-1}_{m} f_{m} }_{L^2\inner{[0,T]\times\mathbb R_+^2}}} \\
 &&+ C_3 A^{m-6 }  \inner{(m-5)!}^{s}   
 \end{eqnarray*}
 and moreover similar to \reff{jklambda} and \reff{jklambda++},   we have
 \begin{eqnarray*}
 	&& m\,C_3\inner{ \norm{ \Lambda^{-2/3}\phi^{\ell-1}_{m} W^{\ell-1}_{m} f_{m}}_{L^2\inner{[0,T]\times\mathbb R_+^2}}+\norm{ \Lambda^{-2/3} \phi^{-1/2}\partial_y  \phi^{\ell-1}_{m} W^{\ell-1}_{m} f_{m} }_{L^2\inner{[0,T]\times\mathbb R_+^2}}}  \\
	& \leq& C m^{ 1/2} A^{m-5 +\frac {\ell-1} {6}}\inner{(m-5)! }^{3(1+\sigma)}(m-4)^{(\ell-1)(1+\sigma)},
 \end{eqnarray*}
 since $\abs{\partial_t\phi^{1/2}}\geq 1.$  Combining the above three inequalities gives
 \begin{eqnarray*}
	\norm{\Lambda^{-2/3}\partial_y h}_{L^2\inner{[0,T]\times \mathbb R_+^2}}&\leq& C \inner{ \| \comii y^{-\sigma}\Lambda^{1/3} f \|_{L^2([0,T]\times\RR^2_+)}+ \norm{ \partial_y^2 \Lambda^{-2/3}f }_{L^2\inner{[0,T]\times\mathbb R_+^2}} }\\
	&&+C m^{ 1/2} A^{m-5 +\frac {\ell-1} {6}}\inner{(m-5)! }^{3(1+\sigma)}(m-4)^{(\ell-1)(1+\sigma)}. 
\end{eqnarray*}
 
 \medskip
 (c)  It follows from the inductive assumption \reff{+ini} that, observing $\phi^{1/2}\leq 1,$
 \begin{eqnarray*}
 	\sum_{j=0}^1\norm{\partial_y^j f}_{L^2\inner{[0,T]\times\mathbb R_+^2}}&\leq&   \norm{ \phi^{\ell-1}_{m} W^{\ell-1}_{m} f_{m}}_{L^2\inner{[0,T]\times\mathbb R_+^2}}+\norm{   \partial_y  \phi^{\ell-1}_{m} W^{\ell-1}_{m} f_{m} }_{L^2\inner{[0,T]\times\mathbb R_+^2}}\\
 	&\leq&   A^{m-5 +\frac {\ell-1} {6}}\inner{(m-5)! }^{3(1+\sigma)}(m-4)^{(\ell-1)(1+\sigma)}. 
 \end{eqnarray*}
 Now we estimate  $\norm{\comii y^\sigma\Lambda^{-1/3}  \partial_y g}_{L^2(   [0, T]\times \mathbb R_+^2)}$, with $g$ is defined in \reff{gam}.  It is quite similar as in step (a).  For instance,
 \begin{eqnarray*}
 	&&\norm{\comii y^\sigma\Lambda^{-1/3}  \partial_y \Lambda^{-2}_\delta\inner{\partial_t \phi^{3(m-N_0-1)+\ell-{1\over 2}}}\Lambda^{(\ell-1)/3} f_m}_{L^2(   [0, T]\times \mathbb R_+^2)} \\
 	&\leq & \norm{ \Lambda^{-1/3}  \phi^{-1/2} \Lambda_\delta^{-2}\phi^{\ell-1}_m W_{m}^{\ell-1} \partial_y f_{m}}_{L^2\inner{[0,T]\times \mathbb R_+^2}} \\
 	&\leq & C \norm{ \Lambda^{-1/3}  \phi^{-1/2} \Lambda_\delta^{-2}\phi^{\ell-1}_m W_{m}^{\ell-1}  f_{m}}_{L^2\inner{[0,T]\times \mathbb R_+^2}} +C \norm{ \partial_y \Lambda^{-1/3}  \phi^{-1/2} \Lambda_\delta^{-2}\phi^{\ell-1}_m W_{m}^{\ell-1}  f_{m}}_{L^2\inner{[0,T]\times \mathbb R_+^2}}.
 \end{eqnarray*}
 Then  similar to \reff{jklambda} and \reff{jklambda++}, we conclude 
 \begin{eqnarray*}
 	&& \norm{\comii y^\sigma\Lambda^{-1/3}  \partial_y \Lambda^{-2}_\delta\inner{\partial_t \phi^{3(m-N_0-1)+\ell-{1\over 2}}}\Lambda^{(\ell-1)/3} f_m}_{L^2(   [0, T]\times \mathbb R_+^2)} \\
 	&\leq & C m^{-1/2} A^{m-5 +\frac {\ell-1} {6}}\inner{(m-5)! }^{3(1+\sigma)}(m-4)^{(\ell-1)(1+\sigma)}.
 \end{eqnarray*}
 The other terms in \reff{gam}  can be estimated similarly,  and a classical commutator estimate (see Lemma \ref{lemma3.1} in the following section)  will be used for treatment of the third term in \reff{gam}.  Thus we conclude   
 \begin{eqnarray*}
 	\norm{\comii y^\sigma\Lambda^{-1/3}  \partial_y g}_{L^2(   [0, T]\times \mathbb R_+^2)}\leq  C   A^{m-5 +\frac {\ell-1} {6}}\inner{(m-5)! }^{3(1+\sigma)}(m-4)^{(\ell-1)(1+\sigma)}. 
 \end{eqnarray*}

\medskip 
 (d) It remains to estimate $\norm{\comii y^{-\sigma/2} \partial_y  \Lambda^{1/6}f}_{L^2(   [0, T]\times \RR_+^2)}$, and we have  
\begin{eqnarray*}
&&\norm{\comii y^{-\sigma/2}  \partial_y \Lambda^{1/6}f}_{L^2(   [0, T]\times \RR_+^2)}^2 =\norm{\comii y^{-\sigma/2}\partial_y\Lambda^{1/6} \Lambda^{-2}_{\delta} \phi^{3(m-N_0-1)+\ell-{1\over 2}} W^{\ell-1}_m
f_{m} }_{L^2\inner{[0,T]\times\mathbb R_+^2}} ^2\\
&=& \inner{\comii y^{-\sigma}\partial_y  \Lambda^{1/3}\Lambda^{-2}_{\delta} \phi^{3(m-N_0-1)+\ell} W^{\ell-1}_m
f_{m} ,~\partial_y  \Lambda^{-2}_{\delta}\phi^{3(m-N_0-1)+\ell-1} W^{\ell-1}_m
f_{m} }_{L^2\inner{[0,T]\times\mathbb R_+^2}} \\
&\leq & \norm{\comii y^{-\sigma}\partial_y \Lambda^{1/3} \Lambda^{-2}_{\delta}  \phi^{\ell}_m W^{\ell-1}_m
f_{m} }_{L^2\inner{[0,T]\times\mathbb R_+^2}} \norm{ \partial_y  \Lambda^{-2}_{\delta}\phi^{ \ell-1}_m W^{\ell-1}_m
f_{m} }_{L^2\inner{[0,T]\times\mathbb R_+^2}} \\
&\leq & C \inner{\norm{ \partial_y \Lambda^{-2}_{\delta}  \phi^{\ell}_mW^{\ell}_m
f_{m} }_{L^2\inner{[0,T]\times\mathbb R_+^2}}+\norm{   \Lambda^{-2}_{\delta}  \phi^{\ell}_m W^{\ell}_m
f_{m} }_{L^2\inner{[0,T]\times\mathbb R_+^2}} }\norm{ \partial_y \phi^{ \ell -1}_m W^{\ell-1}_m
f_{m} }_{L^2\inner{[0,T]\times\mathbb R_+^2}},
\end{eqnarray*}
the last inequality following from the third estimate in  Lemma \ref{le+2.2+}.   This,  along with the inductive assumption  \eqref{+ini} implies,  for any $\tilde\eps>0, $
\begin{eqnarray*} 
&&\norm{\comii y^{-\sigma/2}  \partial_y \Lambda^{1/3}f}_{L^2(   [0, T]\times \RR_+^2)}  \\
&\leq & ~\tilde \eps m^{-(1+\sigma)/2}  \inner{\norm{F}_{L^2\inner{[0,T]\times\mathbb R_+^2}}+\norm{ \partial_y F }_{L^2\inner{[0,T]\times\mathbb R_+^2}} } + C_{\tilde \eps} m^{(1+\sigma)/2}  A^{m-5+\frac{\ell-1}{6}} \inner{(m-5)!}^sm^{(\ell-1)(1+\sigma)},
\end{eqnarray*}
recalling $F=\Lambda^{-2}_{\delta}  \phi^{\ell}_mW^{\ell}_m
f_{m} $. 	

Now combining the estimates in the above steps (a)-(d), we obtain the desired \reff{claim++1}.  
\end{proof}

\section{Subelliptic estimate}\label{section3}
In this section we prove the Proposition \ref{subtang}. We need the following 
commutators estimates.  Throughout the paper we use $\com{Q_1,~Q_2}$ to denote the commutator between two operators $Q_1$ and $Q_2$,  which is defined by 
\begin{eqnarray*}
	\com{Q_1,~Q_2}= Q_1Q_2 -Q_2Q_1=-\com{Q_2,~Q_1}. 
\end{eqnarray*}
We have
\begin{eqnarray}\label{15042210}
	 \com{Q_1,~Q_2Q_3}=Q_2\com{Q_1,~Q_3}+\com{Q_1,~Q_2}Q_3.
\end{eqnarray}
\begin{lemma}\label{lemma3.1}
Denote by $\com{\alpha}$ the largest integer less than or equal to $\alpha\geq 0.$   For any $\tau\in\RR$ and $a\in C^{[|\tau|]+1}_b(\RR^2_+)$, the space of functions such that all their derivatives up to the order of  $[|\tau|]+1$ are continuous and bounded,  there exists $C>0$ such that 
for suitable function $f$ and any $0<\delta<1$,
$$
\|[a, \Lambda^\tau \Lambda^{-2}_\delta] f\|_{L^2(\RR^2_+)}\le C 
\|\Lambda^{\tau-1} \Lambda^{-2}_\delta f\|_{L^2(\RR^2_+)},
$$
and
$$
\|[a\partial_x, \Lambda^\tau \Lambda^{-2}_\delta] f\|_{L^2(\RR^2_+)}\le C 
\|\Lambda^{\tau} \Lambda^{-2}_\delta f\|_{L^2(\RR^2_+)}.
$$
The constant $C$ depends on only on $\tau$ and $\|a\|_{C^{[|\tau|]+1}_b(\RR^2_+)}$.
\end{lemma}
Since $\Lambda^\tau \Lambda^{-2}_\delta$ is only a Fourier multiplier of $x$ variable, so we can prove the above Lemma by direct calculus or pseudo-differential computation, cf. \cite{Hormander85, MR2599384}.  In this section, we use above Lemma with $a=u$ or $a=v$ and $\tau=-1/3, -2/3$. So that with hypothesis \eqref{1.5}, the constant in   Lemma \ref{lemma3.1} depends only on the constant  $C_0$ in Theorem \ref{mainthm}. 
 
\begin{proof} [ Proof of the Proposition \ref{subtang}] 
Taking the operator $ \Lambda^{-2/3}$ on both sides of \reff{eqf1},  we see the function $ \Lambda^{-2/3} f$ satisfies 
 the following equation  in $   ]0, T[~\times \mathbb R_+^2$:
 \begin{eqnarray}\label{dtl}
 \begin{split}
&\partial_t  \Lambda^{-2/3}  f+u\partial_x  \Lambda^{-2/3} f +v\partial_y  \Lambda^{-2/3} f -\partial_y^2
     \Lambda^{-2/3} f\\
    =&~ \Lambda^{-2/3} h+\com{u\partial_x+v\partial_y, ~  \Lambda^{-2/3} }f,
\end{split}
\end{eqnarray}
and that 
\begin{eqnarray}\label{bdc}
 \Lambda^{-2/3}f\big|_{t=0}= \Lambda^{-2/3}f\big|_{t=T}=0,\quad  \partial_y \Lambda^{-2/3}f\big|_{y=0}=0
\end{eqnarray}
due to \reff{ib} and \reff{bdc+}, since  $ \Lambda^{-2/3}$ is an operator acing only on $x$ variable. 
Recall  $\com{u\partial_x+v\partial_y, ~  \Lambda^{-2/3} }$ stands for the commutator between $u\partial_x+v\partial_y$ and $ \Lambda^{-2/3}.$  

{\it Step 1)}.      We will show in this step  that 
\begin{eqnarray}\label{stepb}
\begin{split}
 &\norm{\inner{\partial_y u}^{1/2} \partial_x   \Lambda^{-2/3}f}_{L^2(   [0, T]\times \mathbb R_+^2)}^2\\
  &\leq 2\abs{{\rm Re}~\inner{  \partial_t       \Lambda^{-2/3}f,~
   \partial_y \partial_x    \Lambda^{-2/3}  f }_{L^2(   [0, T]\times \mathbb R_+^2)}} + \norm{\partial_y^2   \Lambda^{-1/3} f}_{L^2(  [0, T]\times \mathbb R_+^2)} ^2 \\
&\quad+C\inner{ \norm{\Lambda^{-1/3} h}_{L^2(  [0, T]\times \mathbb R_+^2)}^2+  \norm{\partial_y f}_{L^2(  [0, T]\times \mathbb R_+^2)}^2 +\norm{ f}_{L^2(  [0, T]\times \mathbb R_+^2)}^2}.
\end{split}
\end{eqnarray}
To do so,  we take $L^2(  [0, T]\times \mathbb R_+^2)$ inner product  with the function $  \partial_y \partial_x  \Lambda^{-2/3}  f \in L^2(  [0, T]\times \mathbb R_+^2)$ on the both sides of
equation \reff{dtl},  and then consider  the real parts; this gives  
\begin{eqnarray}\label{leftside}
\begin{split}
&-{\rm Re}~\inner{u\partial_x    \Lambda^{-2/3} f ,~\partial_y \partial_x    \Lambda^{-2/3}  f }_{L^2(  [0, T]\times \mathbb R_+^2)} \\
=&
~{\rm Re}~\inner{  \partial_t      \Lambda^{-2/3}f,~
   \partial_y \partial_x    \Lambda^{-2/3}  f }_{L^2(   [0, T]\times \mathbb R_+^2)}
- {\rm Re}~\inner{  \partial_y^2   \Lambda^{-2/3}f,~
   \partial_y \partial_x    \Lambda^{-2/3}  f }_{L^2(   [0, T]\times \mathbb R_+^2)}\\
   &~+{\rm Re}~\inner{  v\partial_y     \Lambda^{-2/3} f,~
   \partial_y \partial_x    \Lambda^{-2/3}  f }_{L^2(   [0, T]\times \mathbb R_+^2)}
    -{\rm Re}~\inner{    \Lambda^{-2/3} h,~
    \partial_y \partial_x    \Lambda^{-2/3}  f }_{L^2(   [0, T]\times \mathbb R_+^2)}\\
    &~-{\rm Re}~\inner{\com{u\partial_x+v\partial_y, ~     \Lambda^{-2/3} }f,~
    \partial_y \partial_x    \Lambda^{-2/3}  f }_{L^2(   [0, T]\times \mathbb R_+^2)}.
\end{split}
\end{eqnarray}
We will treat the terms on both sides.  For the  term on left
hand side  we integrate by parts to obtain, here we use  $u\big|_{y=0}=0$ , 
\begin{eqnarray*}
&&-{\rm Re}~\inner{u\partial_x     \Lambda^{-2/3}f ,~ \partial_y \partial_x     \Lambda^{-2/3}f}_{L^2(   [0, T]\times \mathbb R_+^2)}\\
&=&- \frac{1}{2}\set{\inner{u\partial_x     \Lambda^{-2/3}f ,~  \partial_y \partial_x     \Lambda^{-2/3}f}_{L^2(   [0, T]\times \mathbb R_+^2)}+\inner{ \partial_y \partial_x      \Lambda^{-2/3} f,~ u\partial_x     \Lambda^{-2/3}f }_{L^2(    [0, T]\times \mathbb R_+^2)} }\\
&=& \frac{1}{2}\norm{\inner{\partial_y u}^{1/2} \partial_x   \Lambda^{-2/3}f}_{L^2(   [0, T]\times \mathbb R_+^2)}^2.
\end{eqnarray*}
Next we estimate the terms on the right hand side and have,  by Cauchy-Schwarz's inequality , 
\begin{eqnarray*}
\abs{- {\rm Re}~\inner{  \partial_y^2   \Lambda^{-2/3}f,~
   \partial_y \partial_x    \Lambda^{-2/3}  f }_{L^2(   [0, T]\times \mathbb R_+^2)}}\leq \frac{1}{2}\norm{\partial_y^2   \Lambda^{-1/3} f}_{L^2(  [0, T]\times \mathbb R_+^2)} ^2+\frac{1}{2}\norm{\partial_y f}_{L^2(  [0, T]\times \mathbb R_+^2)} ^2,\\
    \abs{{\rm Re}~\inner{    \Lambda^{-2/3} h,~
    \partial_y \partial_x    \Lambda^{-2/3}  f }_{L^2(   [0, T]\times \mathbb R_+^2)}} \leq  \norm{\Lambda^{-1/3} h}_{L^2(  [0, T]\times \mathbb R_+^2)}^2+ \norm{\partial_y f}_{L^2(  [0, T]\times \mathbb R_+^2)}^2
    \end{eqnarray*}
    and
    \begin{eqnarray*}
    	&& \abs{-{\rm Re}~\inner{  v\partial_y     \Lambda^{-2/3} f,~
   \partial_y \partial_x    \Lambda^{-2/3}  f }_{L^2(   [0, T]\times \mathbb R_+^2)}}\\
   &\leq&  \abs{ \inner{       \partial_y   f,~
\com{  \Lambda^{-2/3},~v}   \partial_y\partial_x    \Lambda^{-2/3}  f }_{L^2(   [0, T]\times \mathbb R_+^2)}}+\abs{ \inner{       v \partial_y f,~
  \Lambda^{-2/3}  \partial_y \partial_x    \Lambda^{-2/3}  f }_{L^2(   [0, T]\times \mathbb R_+^2)}}\\
   &\leq& C \norm{\partial_y f}_{L^2(  [0, T]\times \mathbb R_+^2)} ^2,
    \end{eqnarray*}
    the last inequality using Lemma \ref{lemma3.1}. Finally
    \begin{eqnarray*}
   && \abs{-{\rm Re}~\inner{\com{u\partial_x+v\partial_y, ~     \Lambda^{-2/3} }f,~
    \partial_y \partial_x    \Lambda^{-2/3}  f }_{L^2(   [0, T]\times \mathbb R_+^2)}}\\
   & \leq&  \norm{\Lambda^{1/3}\com{u\partial_x+v\partial_y, ~     \Lambda^{-2/3} }f}_{L^2(   [0, T]\times \mathbb R_+^2)}^2+\norm{
   \partial_y f }_{L^2(   [0, T]\times \mathbb R_+^2)}^2\\
   & \leq& 2 \inner{\norm{\com{u\partial_x+v\partial_y, ~    \Lambda^{1/3}   \Lambda^{-2/3} }f}_{L^2(   [0, T]\times \mathbb R_+^2)}^2+\norm{\com{u\partial_x+v\partial_y, ~     \Lambda^{1/3} } \Lambda^{-2/3} f}_{L^2(   [0, T]\times \mathbb R_+^2)}^2}\\
   &&+\norm{
   \partial_y  f }_{L^2(   [0, T]\times \mathbb R_+^2)}^2\\
    &\leq &C \inner{\norm{f}_{L^2(  [0, T]\times \mathbb R_+^2)}^2 +\norm{\partial_y f}_{L^2(  [0, T]\times \mathbb R_+^2)}^2  },
\end{eqnarray*}
These inequalities, together with \reff{leftside}, yields the desired \reff{stepb}.
 
\medskip
{\it Step 2).} In this step we will estimate the second term on the right hand side of \reff{stepb} and 
show that for any $\eps>0,$
\begin{eqnarray}\label{020602}
\begin{split}
&\norm{
 \partial_y^2  \Lambda^{-1/3} f }_{L^2(  [0, T]\times \mathbb R_+^2)}^2\\
    \leq&~ \eps ~\|\inner{\partial_{y}u}^{1/2}\partial_x   \Lambda^{-2/3}f\|_{L^{2}\inner{  [0, T]\times \mathbb R_+^2}}^{2}\\
    &+C_{\eps} \inner{ \norm{\partial_y f}_{L^{2}(  [0, T]\times \mathbb R_+^2)}^{2}+\norm{ f}_{L^{2}(  [0, T]\times \mathbb R_+^2)}^{2} +\norm{
       \Lambda^{-1/3} h }_{L^2([0,T]\times \mathbb
      R_+^2)}^2},
\end{split} 
\end{eqnarray}
with $C_\eps$ a constant depending on $\eps.$
We see that  the function $ \Lambda^{-1/3} f$  satisfies 
 the equation  in $   ]0, T[~\times \mathbb R_+^2,$
 \begin{eqnarray}\label{dtl+}
 \begin{split}
&\partial_t  \Lambda^{-1/3}  f+\inner{u\partial_x +v\partial_y } \Lambda^{-1/3} f -\partial_y^2
     \Lambda^{-1/3} f\\
    =&~ \Lambda^{-1/3} h+\com{u\partial_x+v\partial_y, ~  \Lambda^{-1/3} }f,
\end{split}
\end{eqnarray}
with the boundary condition 
\begin{eqnarray}\label{bdc++}
\Lambda^{-1/3}f\big|_{t=0}= \Lambda^{-1/3}f\big|_{t=T}=0,\quad  \partial_y\Lambda^{-1/3} f\big|_{y=0}=0.
\end{eqnarray}
Now we take $L^2(  [0, T]\times \mathbb R_+^2)$ inner product  with the function   $ 
-\partial_y^2 \Lambda^{-1/3}f\in L^2(  [0, T]\times \mathbb R_+^2)$ on both sides of \reff{dtl+},  and then consider the real parts;  this gives 
\begin{equation}\label{141224}
     \norm{\partial_y^2 
 \Lambda^{-1/3}f}_{L^2(\mathbb
R_+^3)}^2 \leq \sum_{p=1}^4 J_p,
\end{equation}
where
\begin{eqnarray*}
J_1&=&\abs{{\rm Re}~\inner{\partial_t   \Lambda^{-1/3}  f,~  \partial_y^2  \Lambda^{-1/3}  f}_{L^2([0,T]\times\mathbb
     R_+^2)}},\\
J_2&=& \abs{{\rm Re}~\inner{
     \inner{u\partial_x +v\partial_y
     } \Lambda^{-1/3} f,~
     \partial_y^2  \Lambda^{-1/3}f }_{[0,T]\times L^2([0,T]\times\mathbb R_+^2)}}, \\
J_3&=&   \abs{{\rm Re}~\inner{ \Lambda^{-1/3}
     h,~   \partial_y^2\Lambda^{-1/3} f }_{L^2(  [0, T]\times \mathbb R_+^2)}},\\
J_4&=&  \abs{{\rm Re}~\inner{
     \com{u\partial_x+v\partial_y ,~  \Lambda^{-1/3}
} f,~   \partial_y^2\Lambda^{-1/3} f }_{L^2([0,T]\times\mathbb R_+^2)}}. 
\end{eqnarray*}
Integrating by parts and observing the condition \reff{bdc++}, we see
\begin{eqnarray*}
 \inner{ \partial_t   \Lambda^{-1/3}  f,~  \partial_y^2  \Lambda^{-1/3}  f}_{L^2([0,T]\times\mathbb
     R_+^2)} &=& -\inner{\partial_t \partial_y   \Lambda^{-1/3}  f,~  \partial_y  \Lambda^{-1/3}  f}_{L^2([0,T]\times\mathbb
     R_+^2)},
\end{eqnarray*}
which along with the fact 
\begin{eqnarray*}
{\rm Re}~ \inner{\partial_t\partial_y  \Lambda^{-1/3}  f,~  \partial_y    \Lambda^{-1/3}  f}_{L^2([0,T]\times\mathbb
     R_+^2)}=0
\end{eqnarray*}
due to \reff{bdc++}, 
 implies
\begin{eqnarray*}
J_1= \abs{{\rm Re}~\inner{\partial_t   \Lambda^{-1/3}  f,~  \partial_y^2  \Lambda^{-1/3}  f}_{L^2([0,T]\times\mathbb
     R_+^2)}}=0.
\end{eqnarray*}
About $J_2$ we integrate by parts again and observe the boundary conditition \reff{bdc++}, to compute 
\begin{eqnarray*}
&&{\rm Re}\inner{u\p \Lambda^{-1/3}f,
~\partial_y^2 \Lambda^{-1/3} f}_{L^{2}(  [0, T]\times \mathbb R_+^2)}\\
&=&- {\rm Re}\inner{u \p   \Lambda^{-1/3}\partial_y f,~
 \Lambda^{-1/3}\partial_{y}f}_{L^{2}(  [0, T]\times \mathbb R_+^2)}-{\rm Re}~\inner{(\partial_yu)\partial_x \Lambda^{-1/3}f,~
 \Lambda^{-1/3}\partial_{y}f}_{ L^{2}\inner{[0,T]\times\RR^2_+}}\\
&=& \frac{1}{2}\inner{\inner{\p u}  \Lambda^{-1/3}\partial_y f,~
 \Lambda^{-1/3}\partial_{y}f}_{L^{2}(  [0, T]\times \mathbb R_+^2)}-{\rm Re}~\inner{(\partial_yu)\partial_x \Lambda^{-1/3}f,~
 \Lambda^{-1/3}\partial_{y}f}_{ L^{2}\inner{[0,T]\times\RR^2_+}}.
\end{eqnarray*}
This gives
\begin{eqnarray*}
&&\abs{{\rm Re}\inner{ u\p \Lambda^{-1/3}f,
~\partial_y^2 \Lambda^{-1/3} f}_{L^{2}(  [0, T]\times \mathbb R_+^2)}}\\
 &\leq& \| \Lambda^{-1/3}\inner{\partial_{y}u}\partial_x  \Lambda^{-1/3}f\|_{L^{2}\inner{  [0, T]\times \mathbb R_+^2}}  \|   \partial_y f\|_{L^{2}(  [0, T]\times \mathbb R_+^2)}+ C \|  \partial_y  f\|_{L^{2}(  [0, T]\times \mathbb R_+^2)}^2\\
 & \leq &
 \inner{\|\inner{\partial_{y}u}   \Lambda^{-1/3}\partial_x \Lambda^{-1/3}f\|_{L^{2}\inner{  [0, T]\times \mathbb R_+^2}} + \|\com{  \Lambda^{-1/3},~\partial_{y}u} \partial_x\Lambda^{-1/3}f\|_{L^{2}\inner{  [0, T]\times \mathbb R_+^2}}}  \|   \partial_y f\|_{L^{2}(  [0, T]\times \mathbb R_+^2)}\\
 &&+C\|   \partial_y f\|_{L^{2}(  [0, T]\times \mathbb R_+^2)}^2\\
& \leq &
C \inner{ \|\inner{\partial_{y}u}^{1/2}  \partial_x  \Lambda^{-2/3}f\|_{L^{2}\inner{  [0, T]\times \mathbb R_+^2}}+ \|f\|_{L^{2}\inner{  [0, T]\times \mathbb R_+^2}}}  \|   \partial_y f\|_{L^{2}(  [0, T]\times \mathbb R_+^2)} +C\norm{\partial_y f}_{L^{2}\inner{  [0, T]\times \mathbb R_+^2}}^2\\
&\leq&~\eps ~\|\inner{\partial_{y}u}^{1/2}\partial_x   \Lambda^{-2/3}f\|_{L^{2}\inner{  [0, T]\times \mathbb R_+^2}}^{2} +C_\eps \inner{ \norm{\partial_y f}_{L^{2}(  [0, T]\times \mathbb R_+^2)}^{2}+\norm{ f}_{L^{2}(  [0, T]\times \mathbb R_+^2)}^{2}}.
\end{eqnarray*}
 Moreover integrating by part,  we obtain   
\begin{eqnarray*}  \abs{{\rm Re}\inner{ v\partial_y \Lambda^{-1/3}f,
~\partial_y^2 \Lambda^{-1/3} f}_{L^{2}(  [0, T]\times \mathbb R_+^2)}}
&=&\frac{1}{2}\abs{\inner{\inner{\partial_{y}v}\partial_{y} \Lambda^{-1/3}  f,~
\partial_{y} \Lambda^{-1/3}f}_{L^{2}([0,T]\times\mathbb R_+^2)}}\\
&\leq&~ C \norm{\partial_y f}_{L^{2}(  [0, T]\times \mathbb R_+^2)}^{2}.
\end{eqnarray*}
Thus
\begin{eqnarray}\label{J2}
J_2
\leq~\eps ~\|\inner{\partial_{y}u}^{1/2}\partial_x   \Lambda^{-2/3}f\|_{L^{2}\inner{  [0, T]\times \mathbb R_+^2}}^{2} +C_\eps \inner{ \norm{\partial_y f}_{L^{2}(  [0, T]\times \mathbb R_+^2)}^{2}+\norm{ f}_{L^{2}(  [0, T]\times \mathbb R_+^2)}^{2}}.
\end{eqnarray}
It remains to estimate $J_3$ and $J_4$. 
Let $\tilde \eps>0$ be an arbitrarily small number.    Cauchy-Schwarz's inequality gives
\begin{eqnarray*}
J_3 &=& \abs{{\rm Re}~\inner{ \Lambda^{-1/3}
     h,~   \partial_y^2\Lambda^{-1/3} f }_{L^2(  [0, T]\times \mathbb R_+^2)}} \\\\
&\leq& \tilde\eps ~\norm{
     \partial_y^2  \Lambda^{-1/3} f }_{L^2(  [0, T]\times \mathbb R_+^2)}^2+ C_{\tilde\epsilon} \norm{
       \Lambda^{-1/3} h }_{L^2(\mathbb
      R_+^3)}^2,
\end{eqnarray*}
and for $J_4$,  Lemma \ref{lemma3.1} implies 
\begin{eqnarray*} 
J_4&=&  \abs{{\rm Re}~\inner{
     \com{u\partial_x+v\partial_y ,~  \Lambda^{-1/3}
} f,~   \partial_y^2\Lambda^{-1/3} f }_{L^2(\mathbb R_+^2)}}\\
&\leq&   \tilde\eps~ \norm{
     \partial_y^2  \Lambda^{-1/3} f }_{L^2(  [0, T]\times \mathbb R_+^2)}^2+C_{\tilde\epsilon}\inner{ \norm{
       f }_{L^2(\mathbb
      R_+^3)}^2+\norm{
       \partial_y f }_{L^2(\mathbb
      R_+^3)}^2},
      \end{eqnarray*}
where $\tilde C_\eps$  is constant depending on $\tilde\eps.$ 
Now the above two estimates for $J_3$ and $J_4$, along with \reff{141224} - \reff{J2},  gives
\begin{eqnarray*}
 \norm{
     \partial_y^2  \Lambda^{-1/3} f }_{L^2(  [0, T]\times \mathbb R_+^2)}^2&\leq& \tilde\eps~ \norm{
     \partial_y^2  \Lambda^{-1/3} f }_{L^2(  [0, T]\times \mathbb R_+^2)}^2+\eps ~\|\inner{\partial_{y}u}^{1/2}\partial_x   \Lambda^{-2/3}f\|_{L^{2}\inner{  [0, T]\times \mathbb R_+^2}}^{2}\\
     && +C_\eps \inner{ \norm{\partial_y f}_{L^{2}(  [0, T]\times \mathbb R_+^2)}^{2}+\norm{ f}_{L^{2}(  [0, T]\times \mathbb R_+^2)}^{2}}\\
     &&+C_{\tilde\epsilon} \inner{\norm{
       \Lambda^{-1/3} h }_{L^2(\mathbb
      R_+^3)}^2+  \norm{
       f }_{L^2(\mathbb
      R_+^3)}^2+\norm{\partial_y
       f }_{L^2(\mathbb
      R_+^3)}^2},
\end{eqnarray*}
and thus, letting $\tilde\eps$ small sufficiently, 
\begin{eqnarray*}
     &&\norm{
 \partial_y^2  \Lambda^{-1/3} f }_{L^2(  [0, T]\times \mathbb R_+^2)}^2\\
    &\leq& \eps ~\|\inner{\partial_{y}u}^{1/2}\partial_x   \Lambda^{-2/3}f\|_{L^{2}\inner{  [0, T]\times \mathbb R_+^2}}^{2}+C_\eps \inner{ \norm{\partial_y f}_{L^{2}(  [0, T]\times \mathbb R_+^2)}^{2}+\norm{ f}_{L^{2}(  [0, T]\times \mathbb R_+^2)}^{2} +\norm{
       \Lambda^{-1/3} h }_{L^2(\mathbb
      R_+^3)}^2}.
\end{eqnarray*}
This is just the desired estimate \reff{020602}.

Combining the estimates \reff{stepb} and \reff{020602},  we obtain, choosing $\eps$ sufficiently small,
\begin{eqnarray}\label{step3s}
\begin{split}
 &\norm{\inner{\partial_y u}^{1/2} \partial_x   \Lambda^{-2/3}f}_{L^2(   [0, T]\times \mathbb R_+^2)}^2+\norm{\partial_y^2   \Lambda^{-1/3} f}_{L^2(  [0, T]\times \mathbb R_+^2)} ^2\\
  &\leq C \abs{{\rm Re}~\inner{  \partial_t       \Lambda^{-2/3}f,~
   \partial_y \partial_x    \Lambda^{-2/3}  f }_{L^2(   [0, T]\times \mathbb R_+^2)}}  \\
&\quad+C\inner{ \norm{\Lambda^{-1/3} h}_{L^2(  [0, T]\times \mathbb R_+^2)}^2+  \norm{\partial_y f}_{L^2(  [0, T]\times \mathbb R_+^2)}^2 +\norm{ f}_{L^2(  [0, T]\times \mathbb R_+^2)}^2}.
\end{split}
\end{eqnarray}

\medskip
{\it  Step 3)} It remains to treat the first term on the right hand side of \reff{step3s}. 
In this step we will prove that,  for any  $\eps_1>0$,
\begin{eqnarray}\label{stepc}
\begin{split}
&\abs{{\rm Re}~\inner{ \partial_t     \Lambda^{-2/3}f,~
   \partial_y \partial_x  \Lambda^{-2/3}f }_{L^2([0,T]\times\mathbb R_+^2)}} \\
     \leq & ~ \eps_1  \int_0^T\int_{\mathbb R}\abs{ \inner{  \partial_y^2      \Lambda^{-1/2}f}(t,x,0)  }^2dxdt+C_{\eps_1} \norm{\comii y\Lambda^{-1/3} \partial_y g }_{L^2([0,T]\times\mathbb R_+^2)}^2
    \\
 &~+\eps_1^{-1}C    \inner{\norm{ \comii y^{-\sigma/2}    \Lambda^{1/6} f }_{L^2([0,T]\times \mathbb R_+^2)}^2+ \norm{\comii y^{-\sigma/2}  \partial_y      \Lambda^{1/6} f}_{L^2([0,T]\times\mathbb R_+^2)}^2}.
\end{split}
 \end{eqnarray}
 For this purpose we 
integrate by parts again and observe the boundary  condition \reff{bdc} ,  to compute 
\begin{eqnarray*}
&& \inner{ \partial_t     \Lambda^{-2/3}f,~
   \partial_y \partial_x  \Lambda^{-2/3}f }_{L^2([0,T]\times\mathbb R_+^2)}  \\
     &=&-\inner{ \Lambda^{-2/3}f,~
     \partial_t     \partial_y\partial_x  \Lambda^{-2/3}f }_{L^2([0,T]\times\mathbb R_+^2)} \\
     &=&\inner{\partial_x  \Lambda^{-2/3}f,~
     \partial_t   \partial_y   \Lambda^{-2/3}f }_{L^2([0,T]\times\mathbb R_+^2)} \\
     &=&-\inner{\partial_y\partial_x   \Lambda^{-2/3}f,~
     \partial_t      \Lambda^{-2/3}f }_{L^2([0,T]\times\mathbb R_+^2)}+\int_0^T\int_{\mathbb R}\inner{   \partial_t      \Lambda^{-2/3}f(t,x,0)  } \inner{ \partial_x     \Lambda^{-2/3} f(t,x,0)}dxdt,
     \end{eqnarray*}
which, along with the fact that
\begin{eqnarray*}
&& 2 ~{\rm Re}~\inner{ \partial_t    \Lambda^{-2/3}f,~
   \partial_y \partial_x  \Lambda^{-2/3}f }_{L^2([0,T]\times\mathbb R_+^2)}\\
     &=&\inner{ \partial_t     \Lambda^{-2/3}f,~
   \partial_y \partial_x  \Lambda^{-2/3}f }_{L^2([0,T]\times\mathbb R_+^2)}+\inner{\partial_y\partial_x   \Lambda^{-2/3}f,~
     \partial_t      \Lambda^{-2/3}f }_{L^2([0,T]\times\mathbb R_+^2)},
\end{eqnarray*}
yields,  for any $\eps_1>0,$
\begin{eqnarray}\label{reptl}
\begin{split}
& \abs{{\rm Re}~\inner{ \partial_t      \Lambda^{-2/3}f,~
   \partial_y \partial_x  \Lambda^{-2/3}f }_{L^2([0,T]\times\mathbb R_+^2)}}\\
     &=\frac{1}{2}\abs{\int_0^T\int_{\mathbb R}\inner{   \partial_t      \Lambda^{-2/3}f(t,x,0)  } \inner{ \partial_x     \Lambda^{-2/3} f(t,x,0)}dxdt}\\
     &=\frac{1}{2}\abs{\int_0^T\int_{\mathbb R}\inner{  \Lambda^{1/6} \partial_t      \Lambda^{-2/3}f(t,x,0)  } \inner{ \Lambda^{-1/6}\partial_x     \Lambda^{-2/3} f(t,x,0)}dxdt} \\
     &\leq \eps_1 \int_0^T\int_{\mathbb R}\inner{   \partial_t      \Lambda^{-1/2}f(t,x,0)  }^2dxdt+\eps_1^{-1}    \int_0^T\int_{\mathbb R}  \inner{      \Lambda^{1/6} f(t,x,0)}^2 dxdt.
\end{split}
\end{eqnarray}
Moreover observing
\begin{eqnarray*}
    \Lambda^{1/6} f(t,x,0)=\inner{\comii y^{-\sigma/2}    \Lambda^{1/6} f}(t,x,0),
\end{eqnarray*}
it then follows from Sobolev inequality that 
\begin{eqnarray*}
\abs{    \Lambda^{1/6} f(t,x,0)}^2&\leq& C \inner{\norm{ \comii y^{-\sigma/2}    \Lambda^{1/6} f }_{L^2(\mathbb R_+)}^2+ \norm{\partial_y \comii y^{-\sigma/2}    \Lambda^{1/6} f}_{L^2(\mathbb R_+)}^2} \\
&\leq& C \inner{\norm{ \comii y^{-\sigma/2}    \Lambda^{1/6} f }_{L^2(\mathbb R_+)}^2+ \norm{ \comii y^{-\sigma/2}  \partial_y  \Lambda^{1/6} f}_{L^2(\mathbb R_+)}^2} 
\end{eqnarray*}
with $C$ a constant independent of $t,x$.  And thus
\begin{eqnarray}\label{lamfone}
\begin{split}	
&  \int_0^T\int_{\mathbb R}  \inner{      \Lambda^{1/6} f(t,x,0)}^2 dxdt\\
  &\leq  C \inner{\norm{ \comii y^{-\sigma/2}    \Lambda^{1/6} f }_{L^2([0,T]\times \mathbb R_+^2)}^2+ \norm{\partial_y \comii y^{-\sigma/2}    \Lambda^{1/6} f}_{L^2([0,T]\times\mathbb R_+^2)}^2}\\
  &\leq  C \inner{\norm{ \comii y^{-\sigma/2}    \Lambda^{1/6} f }_{L^2([0,T]\times \mathbb R_+^2)}^2+ \norm{ \comii y^{-\sigma/2}    \Lambda^{1/6}\partial_y f}_{L^2([0,T]\times\mathbb R_+^2)}^2}.
\end{split}
\end{eqnarray}
Using the fact that 
\begin{eqnarray*}
\partial_t    \Lambda^{-1/2} f(t,x,0)=\inner{\partial_y^2    \Lambda^{-1/2} f}(t,x,0)+ \Lambda^{-1/2}g(t,x,0)
\end{eqnarray*}
due to assumption \reff{bdc+}, we conclude
\begin{eqnarray*}
&& \int_0^T\int_{\mathbb R}\inner{   \partial_t      \Lambda^{-1/2}f(t,x,0)  }^2dxdt  \\
&\leq &  \int_0^T\int_{\mathbb R}\abs{ \inner{  \partial_y^2      \Lambda^{-1/2}f}(t,x,0)  }^2 dxdt+ \int_0^T\int_{\mathbb R}\abs{   \Lambda^{-1/2}g(t,x,0)  }^2 dxdt.
\end{eqnarray*}
Moreover observe
\begin{eqnarray*}
	\abs{ \Lambda^{-1/2}g(t,x,0)} &=&\abs{-\int_0^{+\infty} \partial_{\tilde y} \Lambda^{-1/2} g(t,x,\tilde y)d\tilde y}\\
	&\leq & \inner{ \int_0^{+\infty} \comii{\tilde y}^{-2\sigma}d\tilde y}^{1/2}\inner{ \int_0^{+\infty} \comii{\tilde y}^{2\sigma}\abs{\Lambda^{-1/2} \partial_{\tilde y} g(t,x,\tilde y)}^2 d\tilde y}^{1/2},
\end{eqnarray*}
which implies 
\begin{eqnarray*}
	  \int_0^T\int_{\mathbb R}\abs{   \Lambda^{-1/2}g(t,x,0)  }^2 dxdt  \leq C\norm{\comii y^\sigma\Lambda^{-1/2} \partial_yg }_{L^2([0,T]\times\mathbb R_+^2)}^2 \leq C\norm{\comii y^\sigma\Lambda^{-1/3} \partial_yg }_{L^2([0,T]\times\mathbb R_+^2)}^2 ,
\end{eqnarray*}
and thus
\begin{eqnarray*}
&& \int_0^T\int_{\mathbb R}\inner{   \partial_t      \Lambda^{-1/2}f(t,x,0)  }^2dxdt  \\
&\leq &  \int_0^T\int_{\mathbb R}\abs{ \inner{  \partial_y^2      \Lambda^{-1/2}f}(t,x,0)  }^2 dxdt+C\norm{\comii y^\sigma\Lambda^{-1/3} \partial_yg }_{L^2([0,T]\times\mathbb R_+^2)}^2.
\end{eqnarray*}
This along with \reff{reptl} and  \reff{lamfone} yields the desired \reff{stepc}.

\medskip
{\it Step 4)}  Combining  \reff{step3s} and  \reff{stepc},   we have, for any $\eps_1>0,$
\begin{eqnarray*}
&&\norm{\inner{\partial_y u}^{1/2} \partial_x   \Lambda^{-2/3}f}_{L^2(   [0, T]\times \mathbb R_+^2)}^2+\norm{\partial_y^2   \Lambda^{-1/3} f}_{L^2(  [0, T]\times \mathbb R_+^2)} ^2\\ & \leq & \eps_1 \int_0^T\int_{\mathbb R}\abs{ \inner{  \partial_y^2      \Lambda^{-1/2}f}(t,x,0)  }^2dxdt+C_{\eps_1} \norm{\comii y^\sigma \Lambda^{-1/3} \partial_yg }_{L^2([0,T]\times\mathbb R_+^2)}^2\\
 &&~+ \eps_1^{-1}C    \inner{\norm{ \comii y^{-\sigma/2}    \Lambda^{1/6} f }_{L^2([0,T]\times \mathbb R_+^2)}^2+ \norm{\comii y^{-\sigma/2}   \partial_y   \Lambda^{1/6} f}_{L^2([0,T]\times\mathbb R_+^2)}^2} \\
    &&+C \inner{ \norm{\partial_y f}_{L^{2}(  [0, T]\times \mathbb R_+^2)}^{2}+\norm{ f}_{L^{2}(  [0, T]\times \mathbb R_+^2)}^{2} +\norm{
       \Lambda^{-1/3} h }_{L^2(\mathbb
      R_+^3)}^2}.
\end{eqnarray*}
Moreover we use the monotonicity condition and interpolation inequality to get, for any $\eps_2>0$
\begin{eqnarray*}
	\norm{ \comii y^{-\sigma/2}    \Lambda^{1/6} f }_{L^2([0,T]\times \mathbb R_+^2)}^2&\leq &\eps_2 \norm{\comii y^{-\sigma/2} \Lambda^{1/3} f }_{L^2([0,T]\times \mathbb R_+^2)}^2+ \eps_2^{-1} \norm{\comii y^{-\sigma/2}   f }_{L^2([0,T]\times \mathbb R_+^2)}^2\\
	&\leq & \eps_2 \norm{\comii y^{-\sigma/2}  \partial_x \Lambda^{-2/3} f }_{L^2([0,T]\times \mathbb R_+^2)}^2+ C_{\eps_2} \norm{\comii y^{-\sigma/2}   f }_{L^2([0,T]\times \mathbb R_+^2)}^2\\
	&\leq & \eps_2 \norm{\inner{\partial_y u} ^{1/2} \partial_x \Lambda^{-2/3} f }_{L^2([0,T]\times \mathbb R_+^2)}^2+ C_{\eps_2} \norm{   f }_{L^2([0,T]\times \mathbb R_+^2)}^2.
\end{eqnarray*}
From the above inequalities,  we infer that,  choosing $\eps_2$ small enough, 
\begin{eqnarray}\label{fiter}
	\begin{split}
		&\norm{\inner{\partial_y u}^{1/2} \partial_x   \Lambda^{-2/3}f}_{L^2(   [0, T]\times \mathbb R_+^2)}^2+\norm{\partial_y^2   \Lambda^{-1/3} f}_{L^2(  [0, T]\times \mathbb R_+^2)} ^2\\ 
		 \leq &\eps_1  \int_0^T\int_{\mathbb R}\abs{ \inner{  \partial_y^2      \Lambda^{-1/2}f}(t,x,0)  }^2dxdt\\
		&+C_{\eps_1}   \inner{  \norm{\comii y^{-\sigma/2}   \partial_y   \Lambda^{1/6} f}_{L^2([0,T]\times\mathbb R_+^2)}^2+\norm{\comii y^\sigma\Lambda^{-1/3} \partial_yg }_{L^2([0,T]\times\mathbb R_+^2)}^2}\\
    &+C_{\eps_1} \inner{ \norm{\partial_y f}_{L^{2}(  [0, T]\times \mathbb R_+^2)}^{2}+\norm{ f}_{L^{2}(  [0, T]\times \mathbb R_+^2)}^{2} +\norm{
       \Lambda^{-1/3} h }_{L^2(\mathbb
      R_+^3)}^2}.
	\end{split}
\end{eqnarray}

\medskip
{\it Step 5)}  
 In this step we treat the first term on the right side of \reff{fiter},   and show that,  for any  $0<\eps<1,$
\begin{eqnarray}\label{step++}
\begin{split}
& \int_0^T\int_{\mathbb R}\abs{ \inner{  \partial_y^2      \Lambda^{-1/2}f}(t,x,0)  }^2dxdt\\
\leq & C \norm{(\partial_y u)^{1/2}\partial_x  \Lambda^{-2/3}  f}_{L^2\inner{[0,T]\times\mathbb R_+^2}}^2+\eps C \norm{\Lambda^{-2/3}\partial_y  h}_{L^2\inner{[0,T]\times\mathbb R_+^2}}^2\\
    &+ C_\eps \inner{\norm{f}_{L^2\inner{[0,T]\times\mathbb R_+^2}}^2  +\norm{\partial_y  f}_{L^2\inner{[0,T]\times\mathbb R_+^2}}^2}.
\end{split}
\end{eqnarray}
To do so,  we   integrate by parts to get
\begin{eqnarray*}
 \int_0^T\int_{\mathbb R}\abs{ \inner{  \partial_y^2      \Lambda^{-1/2}f}(t,x,0)  }^2dxdt
&=&   2{\rm Re}~ \inner{  \partial_y^3   \Lambda^{-1/2}f ,~   \partial_y^2 \Lambda^{-1/2}f }_{L^2\inner{[0,T]\times\mathbb R_+^2}}\\
&=& 2  {\rm Re}~ \inner{  \partial_y^3   \Lambda^{-2/3}f ,~   \partial_y^2 \Lambda^{-1/3}f }_{L^2\inner{[0,T]\times\mathbb R_+^2}}.
\end{eqnarray*}
This yields
\begin{eqnarray}\label{stepsdss}
\begin{split}
& \int_0^T\int_{\mathbb R}\abs{ \inner{  \partial_y^2      \Lambda^{-1/2}f}(t,x,0)  }^2dxdt\\
\leq &   ~   \frac{\eps}{2}\norm{  \partial_y^3  \Lambda^{-2/3}f }_{L^2\inner{[0,T]\times\mathbb R_+^2}}^2+ 2\eps^{-1}\norm{ \partial_y^2   \Lambda^{-1/3}f }_{L^2\inner{[0,T]\times\mathbb R_+^2}}^2\\
\leq & ~    \eps \norm{  \partial_y^3  \Lambda^{-2/3}f }_{L^2\inner{[0,T]\times\mathbb R_+^2}}^2+ C_\eps \norm{ \partial_yf }_{L^2\inner{[0,T]\times\mathbb R_+^2}}^2,
\end{split}
\end{eqnarray}
 the last inequality holding because we can use \reff{bdc+}  to integrate by parts and then obtain
\begin{eqnarray}\label{eqalin}
\begin{split}
	&\norm{ \partial_y^2
     \Lambda^{-1/3} f}_{L^2\inner{[0,T]\times\mathbb R_+^2}}^2=\inner{\partial_y^2
     \Lambda^{-2/3} f,~\partial_y^2
     f}_{L^2\inner{[0,T]\times\mathbb R_+^2}} \leq \abs{\inner{\partial_y^3
     \Lambda^{-2/3} f,~\partial_y
     f}_{L^2\inner{[0,T]\times\mathbb R_+^2}}}\\
     & \leq \norm{ \partial_y^3
     \Lambda^{-2/3} f}_{L^2\inner{[0,T]\times\mathbb R_+^2}} \norm{ \partial_y  f} _{L^2\inner{[0,T]\times\mathbb R_+^2}}.
\end{split}
\end{eqnarray}
Thus in order to prove \reff{step++} it suffices to estimate $\norm{  \partial_y^3  \Lambda^{-2/3}f }_{L^2\inner{[0,T]\times\mathbb R_+^2}}.$  We study the equation
\begin{eqnarray*}
&&\partial_t  \Lambda^{-2/3}\partial_y  f+u\partial_x  \Lambda^{-2/3} \partial_y f +v\partial_y  \Lambda^{-2/3}  \partial_y f -\partial_y^3
     \Lambda^{-2/3}  f \\
   & =&~ \Lambda^{-2/3}  \partial_y h+ \com{u\partial_x+v\partial_y, ~  \Lambda^{-2/3} }\partial_y f-  \Lambda^{-2/3} (\partial_y u)\partial_x f-   \Lambda^{-2/3}(\partial_y v) \partial_y f,
\end{eqnarray*}
which implies,  by taking $L^2$ inner product with  $-\partial_y^3
     \Lambda^{-2/3}  f,$
    \begin{eqnarray}
   &    \norm{\partial_y^3
     \Lambda^{-2/3}  f}_{L^2\inner{[0,T]\times\mathbb R_+^2}}^2
    =&-{\rm Re}~\inner{\partial_t  \Lambda^{-2/3}\partial_y  f,~\partial_y^3
     \Lambda^{-2/3}  f}_{L^2\inner{[0,T]\times\mathbb R_+^2}}\nonumber\\
    && -{\rm Re}~\inner{u\partial_x  \Lambda^{-2/3}\partial_y  f+v\partial_y  \Lambda^{-2/3}\partial_y  f,~\partial_y^3
     \Lambda^{-2/3}  f}_{L^2\inner{[0,T]\times\mathbb R_+^2}}\nonumber\\
    &&+{\rm Re}~\inner{ \Lambda^{-2/3}  \partial_y h,~\partial_y^3
     \Lambda^{-2/3}  f}_{L^2\inner{[0,T]\times\mathbb R_+^2}}\label{3.11}\\
    &&   +~{\rm Re}~\inner{\com{u\partial_x+v\partial_y, ~  \Lambda^{-2/3} }\partial_yf,~\partial_y^3
     \Lambda^{-2/3}  f}_{L^2\inner{[0,T]\times\mathbb R_+^2}}\nonumber\\
    &&-{\rm Re}~\inner{ \Lambda^{-2/3}(\partial_y u)\partial_x   f,~\partial_y^3
     \Lambda^{-2/3}  f}_{L^2\inner{[0,T]\times\mathbb R_+^2}}\nonumber\\
    && -~{\rm Re}~\inner{ \Lambda^{-2/3}(\partial_y v)\partial_y   f,~\partial_y^3
     \Lambda^{-2/3}  f}_{L^2\inner{[0,T]\times\mathbb R_+^2}}.\nonumber    \end{eqnarray}
Next we will treat the terms on the right hand side.  Observing 
\begin{eqnarray*}
\partial_t  \Lambda^{-2/3}\partial_y  f\big|_{y=0}=0
\end{eqnarray*}
due to \reff{bdc+},  we integrate by part to compute
\begin{eqnarray*}
\begin{split}
&-{\rm Re}~\inner{\partial_t  \Lambda^{-2/3}\partial_y  f,~-\partial_y^3
     \Lambda^{-2/3}  f}_{L^2\inner{[0,T]\times\mathbb R_+^2}}\\
    =&-{\rm Re}~\inner{\partial_t \partial_y^2 \Lambda^{-2/3}  f,~\partial_y^2
     \Lambda^{-2/3}  f}_{L^2\inner{[0,T]\times\mathbb R_+^2}}\\
    =&~0,
\end{split}
\end{eqnarray*}
the last equality holding because  
\begin{eqnarray*}
\partial_y^2 \Lambda^{-2/3}  f\big|_{t=0}=\partial_y^2 \Lambda^{-2/3}  f\big|_{t=T}=0
\end{eqnarray*}
due to \reff{ib}.   Since $u\big|_{y=0}$ then integrating by parts gives
\begin{eqnarray*}
&&-~{\rm Re}~\inner{u\partial_x  \Lambda^{-2/3}\partial_y  f,~-\partial_y^3
     \Lambda^{-2/3}  f}_{L^2\inner{[0,T]\times\mathbb R_+^2}}\\
    &=&-~{\rm Re}~\inner{u\partial_x  \Lambda^{-2/3}\partial_y^2  f,~\partial_y^2
     \Lambda^{-2/3}  f}_{L^2\inner{[0,T]\times\mathbb R_+^2}}\\
    &&-~{\rm Re}~\inner{(\partial_y u)\partial_x  \Lambda^{-2/3}\partial_y  f,~\partial_y^2
     \Lambda^{-2/3}  f}_{L^2\inner{[0,T]\times\mathbb R_+^2}}\\
    &=&\frac{1}{2} \inner{(\partial_x u)  \Lambda^{-2/3}\partial_y^2  f,~
     \Lambda^{-2/3} \partial_y^2 f}_{L^2\inner{[0,T]\times\mathbb R_+^2}}\\
    &&-~{\rm Re}~\inner{(\partial_y u)\partial_x  \Lambda^{-2/3}\partial_y  f,~\partial_y^2
     \Lambda^{-2/3}  f}_{L^2\inner{[0,T]\times\mathbb R_+^2}}\\
    &\leq& \frac{1}{2}\norm{\partial_x u}_{L^\infty} \norm{  \Lambda^{-2/3}\partial_y^2  f}_{L^2\inner{[0,T]\times\mathbb R_+^2}}\\
    &&+ \norm{\Lambda^{-1/3} (\partial_y u)\partial_x  \Lambda^{-2/3}\partial_y  f}_{L^2\inner{[0,T]\times\mathbb R_+^2}}^2+ \norm{ \partial_y^2
     \Lambda^{-1/3}  f}_{L^2\inner{[0,T]\times\mathbb R_+^2}}^2.
\end{eqnarray*}
On the other hand,  using Lemma \ref{lemma3.1} gives
\begin{eqnarray*}
 &&\norm{\Lambda^{-1/3} (\partial_y u)\partial_x  \Lambda^{-2/3}\partial_y  f}_{L^2\inner{[0,T]\times\mathbb R_+^2}}^2\\
 &\leq &  2\norm{\Lambda^{-1/3} \partial_x  \Lambda^{-2/3}(\partial_y u)\partial_y  f}_{L^2\inner{[0,T]\times\mathbb R_+^2}}^2 + 2\norm{\Lambda^{-1/3} \com{\partial_y u,~\partial_x  \Lambda^{-2/3}}\partial_y  f}_{L^2\inner{[0,T]\times\mathbb R_+^2}}^2\\
 &\leq &C \norm{\partial_y  f}_{L^2\inner{[0,T]\times\mathbb R_+^2}}^2.
\end{eqnarray*}
Thus
\begin{eqnarray*}
\begin{split}
&-~{\rm Re}~\inner{u\partial_x  \Lambda^{-2/3}\partial_y  f,~-\partial_y^3
     \Lambda^{-2/3}  f}_{L^2\inner{[0,T]\times\mathbb R_+^2}}\\
    \leq&~ C ~\inner{ \norm{ \partial_y  f}_{L^2\inner{[0,T]\times\mathbb R_+^2}}^2+\norm{ \partial_y^2
     \Lambda^{-1/3} f}_{L^2\inner{[0,T]\times\mathbb R_+^2}}^2}\\
    \leq&  ~\tilde\eps ~ \norm{ \partial_y^3
     \Lambda^{-2/3} f}_{L^2\inner{[0,T]\times\mathbb R_+^2}}^2+C_{\tilde\eps}\norm{ \partial_y  f} _{L^2\inner{[0,T]\times\mathbb R_+^2}}^2,
\end{split}
\end{eqnarray*}
where the last inequality using \reff{eqalin}.  Using \reff{eqalin} we conclude   
\begin{eqnarray*}
\begin{split}
& -~{\rm Re}~\inner{v\partial_y  \Lambda^{-2/3}\partial_y  f,~-\partial_y^3
     \Lambda^{-2/3}  f}_{L^2\inner{[0,T]\times\mathbb R_+^2}}\\
& \leq ~\frac{\tilde\eps}{2}\norm{ \partial_y^3  \Lambda^{-2/3}   f}_{L^2\inner{[0,T]\times\mathbb R_+^2}}^2+C_{\tilde\eps } \norm{ \partial_y^2  \Lambda^{-1/3}   f}_{L^2\inner{[0,T]\times\mathbb R_+^2}}^2 \\
  &  \leq ~\tilde\eps ~ \norm{ \partial_y^3
     \Lambda^{-2/3} f}_{L^2\inner{[0,T]\times\mathbb R_+^2}}^2+C_{\tilde\eps}\norm{ \partial_y  f} _{L^2\inner{[0,T]\times\mathbb R_+^2}}^2.
\end{split}
\end{eqnarray*}
Cauchy-Schwarz inequality gives,  for any $\tilde\eps>0,$
\begin{eqnarray*}
\begin{split}
&{\rm Re}~\inner{ \Lambda^{-2/3}  \partial_y h,~-\partial_y^3
     \Lambda^{-2/3}  f}_{L^2\inner{[0,T]\times\mathbb R_+^2}}\\
   \leq &~\tilde \eps \norm{\partial_y^3
     \Lambda^{-2/3}  f}_{L^2\inner{[0,T]\times\mathbb R_+^2}}^2+\tilde\eps^{-1} \norm{\Lambda^{-2/3}  \partial_y h}_{L^2\inner{[0,T]\times\mathbb R_+^2}}^2,
\end{split}
\end{eqnarray*}
and 
\begin{eqnarray*}
\begin{split}
&-~{\rm Re}~\inner{ \Lambda^{-2/3}(\partial_y v)\partial_y   f,~-\partial_y^3
     \Lambda^{-2/3}  f}_{L^2\inner{[0,T]\times\mathbb R_+^2}}\\
    \leq &~ \tilde \eps \norm{\partial_y^3
     \Lambda^{-2/3}  f}_{L^2\inner{[0,T]\times\mathbb R_+^2}}^2+\tilde\eps^{-1} \norm{\partial_y v}_{L^\infty} ^2\norm{\partial_y f}_{L^2\inner{[0,T]\times\mathbb R_+^2}}^2
\end{split}
\end{eqnarray*}
and
\begin{eqnarray}\label{sq6}
\begin{split}
& {\rm Re}~\inner{\com{u\partial_x+v\partial_y, ~  \Lambda^{-2/3} }\partial_yf,~-\partial_y^3
     \Lambda^{-2/3}  f}_{L^2\inner{[0,T]\times\mathbb R_+^2}}\\
    \leq &~ \frac{\tilde \eps}{2} \norm{\partial_y^3
     \Lambda^{-2/3}  f}_{L^2\inner{[0,T]\times\mathbb R_+^2}}^2+2\tilde\eps^{-1} \norm{\com{u\partial_x+v\partial_y, ~  \Lambda^{-2/3} }\partial_y f}_{L^2\inner{[0,T]\times\mathbb R_+^2}}^2\\
    \leq &~\frac {\tilde \eps}{2} \norm{\partial_y^3
     \Lambda^{-2/3}  f}_{L^2\inner{[0,T]\times\mathbb R_+^2}}^2+C_{ \tilde\eps} \norm{ \partial_y f}_{L^2\inner{[0,T]\times\mathbb R_+^2}}^2+C_{ \tilde\eps} \norm{ \partial_y^2\Lambda^{-1/3} f}_{L^2\inner{[0,T]\times\mathbb R_+^2}}^2\\
    \leq &~ \tilde \eps~\norm{\partial_y^3
     \Lambda^{-2/3}  f}_{L^2\inner{[0,T]\times\mathbb R_+^2}}^2+C_{ \tilde\eps} \norm{ \partial_y f}_{L^2\inner{[0,T]\times\mathbb R_+^2}}^2,
\end{split}
\end{eqnarray}
the second inequality using  Lemma \ref{lemma3.1}, while the last inequality following from \reff{eqalin}.   Finally, 
\begin{eqnarray*}
&&-~{\rm Re}~\inner{ \Lambda^{-2/3}(\partial_y u)\partial_x   f,~-\partial_y^3
     \Lambda^{-2/3}  f}_{L^2\inner{[0,T]\times\mathbb R_+^2}}\\
    &\leq &  \tilde \eps \norm{\partial_y^3
     \Lambda^{-2/3}  f}_{L^2\inner{[0,T]\times\mathbb R_+^2}}^2+\tilde\eps^{-1} \norm{ \Lambda^{-2/3}(\partial_y u)\partial_x   f}_{L^2\inner{[0,T]\times\mathbb R_+^2}}^2\\
    &\leq &  \tilde \eps \norm{\partial_y^3
     \Lambda^{-2/3}  f}_{L^2\inner{[0,T]\times\mathbb R_+^2}}^2+  \tilde\eps^{-1}   \norm{(\partial_y u)\partial_x  \Lambda^{-2/3}  f}_{L^2\inner{[0,T]\times\mathbb R_+^2}}^2\\
    &&+\tilde\eps^{-1}  \norm{\com{\partial_y u,~ \Lambda^{-2/3}}   \partial_xf}_{L^2\inner{[0,T]\times\mathbb R_+^2}}^2\\
    &\leq &  \tilde \eps \norm{\partial_y^3
     \Lambda^{-2/3}  f}_{L^2\inner{[0,T]\times\mathbb R_+^2}}^2+  C_{\tilde\eps}  \norm{(\partial_y u)^{1/2}\partial_x  \Lambda^{-2/3}  f}_{L^2\inner{[0,T]\times\mathbb R_+^2}}^2\\
    &&+C_{\tilde\eps}  \norm{f}_{L^2\inner{[0,T]\times\mathbb R_+^2}}^2.
\end{eqnarray*}
This, along with \reff{3.11} -\reff{sq6}, yields, for any $\tilde\eps>0,$
 \begin{eqnarray*}
    &&\norm{\partial_y^3
     \Lambda^{-2/3}  f}_{L^2\inner{[0,T]\times\mathbb R_+^2}}^2\\
    &\leq & \tilde \eps \norm{\partial_y^3
     \Lambda^{-2/3}  f}_{L^2\inner{[0,T]\times\mathbb R_+^2}}^2+C_{\tilde\eps}  \norm{(\partial_y u)^{1/2}\partial_x  \Lambda^{-2/3}  f}_{L^2\inner{[0,T]\times\mathbb R_+^2}}^2\\
    &&+ C_{\tilde\eps}\inner{\norm{\Lambda^{-2/3}\partial_y  h}_{L^2\inner{[0,T]\times\mathbb R_+^2}}^2+\norm{\partial_y  f}_{L^2\inner{[0,T]\times\mathbb R_+^2}}^2+\norm{f}_{L^2\inner{[0,T]\times\mathbb R_+^2}}^2}.
    \end{eqnarray*}
Thus letting $\tilde \eps$ be small enough, we have
\begin{eqnarray}\label{lam3}
\begin{split}
 &\norm{\partial_y^3
     \Lambda^{-2/3}  f}_{L^2\inner{[0,T]\times\mathbb R_+^2}}^2\\
   & \leq  ~ C  \norm{(\partial_y u)^{1/2}\partial_x  \Lambda^{-2/3}  f}_{L^2\inner{[0,T]\times\mathbb R_+^2}}^2
    \\
    &\qquad+C\inner{\norm{\Lambda^{-2/3}\partial_y  h}_{L^2\inner{[0,T]\times\mathbb R_+^2}}^2+\norm{f}_{L^2\inner{[0,T]\times\mathbb R_+^2}}^2  +\norm{\partial_y  f}_{L^2\inner{[0,T]\times\mathbb R_+^2}}^2}.
\end{split}
\end{eqnarray}
This along with \reff{stepsdss} yields the desired estimate \reff{step++}.  

\medskip
{\it Step 6)}   Now we combine \reff{fiter} and \reff{step++} to conclude for any $0<\eps, \eps_1<1,$
\begin{eqnarray*}
	&&\norm{\inner{\partial_y u}^{1/2} \partial_x   \Lambda^{-2/3}f}_{L^2(   [0, T]\times \mathbb R_+^2)}^2+\norm{\partial_y^2   \Lambda^{-1/3} f}_{L^2(  [0, T]\times \mathbb R_+^2)} ^2\\ 
		&\leq& \eps_1  C\norm{\inner{\partial_y u}^{1/2} \partial_x   \Lambda^{-2/3}f}_{L^2(   [0, T]\times \mathbb R_+^2)}^2 +\eps_1\eps C \norm{\Lambda^{-2/3}\partial_y  h}_{L^2\inner{[0,T]\times\mathbb R_+^2}}^2\\
		&&+C_{\eps_1, \eps }   \inner{  \norm{\comii y^{-\sigma/2}   \partial_y   \Lambda^{1/6} f}_{L^2([0,T]\times\mathbb R_+^2)}^2+\norm{\comii y^\sigma\Lambda^{-1/3} \partial_yg }_{L^2([0,T]\times\mathbb R_+^2)}^2}\\
    &&+C_{\eps_1, \eps} \inner{ \norm{\partial_y f}_{L^{2}(  [0, T]\times \mathbb R_+^2)}^{2}+\norm{ f}_{L^{2}(  [0, T]\times \mathbb R_+^2)}^{2}+\norm{
       \Lambda^{-1/3} h }_{L^2(\mathbb
      R_+^3)}^2},
\end{eqnarray*}
which implies,  choosing $\eps_1>0$ sufficiently small, 
\begin{eqnarray*}
	&&\norm{\inner{\partial_y u}^{1/2} \partial_x   \Lambda^{-2/3}f}_{L^2(   [0, T]\times \mathbb R_+^2)}^2+\norm{\partial_y^2   \Lambda^{-1/3} f}_{L^2(  [0, T]\times \mathbb R_+^2)} ^2\\ 
		& \leq &\, \eps \,\norm{\Lambda^{-2/3}\partial_y  h}_{L^2\inner{[0,T]\times\mathbb R_+^2}}^2\\
		&&+C_{\eps}  \inner{  \norm{\comii y^{-\sigma/2}   \partial_y   \Lambda^{1/6} f}_{L^2([0,T]\times\mathbb R_+^2)}^2+\norm{\comii y^\sigma\Lambda^{-1/3} \partial_yg }_{L^2([0,T]\times\mathbb R_+^2)}^2}\\
    &&+C_{\eps} \inner{ \norm{\partial_y f}_{L^{2}(  [0, T]\times \mathbb R_+^2)}^{2}+\norm{ f}_{L^{2}(  [0, T]\times \mathbb R_+^2)}^{2} +\norm{
       \Lambda^{-1/3} h }_{L^2(\mathbb
      R_+^3)}^2},
\end{eqnarray*}
with $\eps>0$ arbitrarily small. 
This,  along with 
\begin{eqnarray*}
	\norm{\comii y^{-\sigma/2}   \Lambda^{1/3}f}_{L^2(   [0, T]\times \mathbb R_+^2)}^2 &\leq& C\norm{\inner {\partial_yu }^{1/2}  \Lambda^{1/3}f}_{L^2(   [0, T]\times \mathbb R_+^2)}^2\\
	&\leq& C\norm{\comii y^{-\sigma/2} \partial_x   \Lambda^{-2/3}f}_{L^2(   [0, T]\times \mathbb R_+^2)}^2+C\norm{f}_{L^2(   [0, T]\times \mathbb R_+^2)}^2
\end{eqnarray*}
due to \reff{1.3},
 implies, for any $\eps>0,$
\begin{eqnarray*}
	&&\norm{\comii y^{-\sigma/2}   \Lambda^{1/3}f}_{L^2(   [0, T]\times \mathbb R_+^2)}^2+\norm{\partial_y^2 \Lambda^{-1/3}f}_{L^2(   [0, T]\times \mathbb R_+^2)}^2 \\ 
		& \leq &\eps  \norm{\Lambda^{-2/3}\partial_y  h}_{L^2\inner{[0,T]\times\mathbb R_+^2}}^2\\
		&&+C_{\eps}  \inner{  \norm{\comii y^{-\sigma/2}   \partial_y   \Lambda^{1/6} f}_{L^2([0,T]\times\mathbb R_+^2)}^2+\norm{\comii y^\sigma\Lambda^{-1/3} \partial_yg }_{L^2([0,T]\times\mathbb R_+^2)}^2}\\
    &&+C_{\eps} \inner{ \norm{\partial_y f}_{L^{2}(  [0, T]\times \mathbb R_+^2)}^{2}+\norm{ f}_{L^{2}(  [0, T]\times \mathbb R_+^2)}^{2} +\norm{
       \Lambda^{-1/3} h }_{L^2(\mathbb
      R_+^3)}^2}.
\end{eqnarray*}
This is just the first  estimate in Proposition \ref{subtang}.   And the second estimate follows from   \reff{lam3} since $\abs{\partial_yu}$ is bounded from above by $\comii y^{-\sigma}$.    Thus the proof of  Proposition \ref{subtang} is complete. 
\end{proof}

\section{Property of inducative weight functions}\label{section4}

This section is devoted to  proving the Lemma \ref{lemma2.1},  Lemma \ref{le+2.2+} and Lemma \ref{lemma261},  used in Section \ref{section2}.

Recall, for $m\geq N_0+1$ and $0\leq \ell\leq 3, y>0, 0\le t\le T<1,$
\begin{eqnarray*}
W_m^\ell=e^{2cy}\inner{1+\frac{2cy}{(3m+\ell)\sigma} }^{-\frac{(3m+\ell)\sigma}{2}}
(1+cy)^{-1} \Lambda^{{\ell\over 3}},\qquad
\phi_m^\ell=\phi^{3(m-N_0-1)+\ell}.
\end{eqnarray*}
thus
\begin{eqnarray}\label{phi}
	\phi_{m_1}^{\ell_1}\leq \phi_{m_2}^{\ell_2}
\end{eqnarray}
provided $N_0+1\leq m_2\leq m_1$ and $0\leq \ell_2\leq \ell_1\leq 3$.
 
 Next we list    some inequalities for the weight $W_m^\ell$.  Observe  the function 
\begin{eqnarray*}
\gamma\longrightarrow \inner{1+\frac{cy}{\gamma}}^{-\gamma}
\end{eqnarray*}
is a monotonically  decreasing function as $\gamma$ varies in the interval $[1,+\infty[$ for $y\geq 0$.  Thus 
\begin{eqnarray}\label{wmi}
 0\le \ell\le 3, \quad \norm{W_{m_1} ^\ell f }_{L^2(\mathbb R_x)}\leq \norm{W_{m_2} ^\ell f }_{L^2(\mathbb R_x)}
\end{eqnarray}
and 
\begin{eqnarray}\label{+wmi}
 \forall~0\leq \ell\leq i\leq 3,   \quad \norm{W_{m_1} ^{i}f }_{L^2(\mathbb R_x)}\leq \norm{W_{m_2}^{i-\ell}\Lambda^{\ell/3}f }_{L^2(\mathbb R_x)}\leq \norm{W_{m_3}^{i} f }_{L^2(\mathbb R_x)},
\end{eqnarray}
provided that $m_1\geq m_2\geq 1,$ and that $3m_2+i-\ell \geq 3m_3+i.$   Moreover, since
\begin{eqnarray*}
	\forall~0\leq \alpha\leq 3, ~\forall~\gamma\geq 1,\quad \abs{\partial_y^\alpha e^{2cy} \inner{1+\frac{cy}{\gamma}}^{-\gamma}(1+cy)^{-1}}\leq   C_\alpha   e^{2cy} \inner{1+\frac{cy}{\gamma}}^{-\gamma}(1+cy)^{-1},
\end{eqnarray*}
with $  C_\alpha  $ a constant independent of  $\gamma,$ then the following estimates: 
\begin{eqnarray}\label{comfm}
	 \norm{\com{\partial_y,~ W_{m} ^{i}} f }_{L^2(\mathbb R_+^2)}\leq C\norm{W_{m} ^{i}f }_{L^2(\mathbb R_+^2)},
\end{eqnarray}
\begin{eqnarray}\label{comfm+}
\begin{split}
	 \norm{\com{\partial_y^2,~ W_{m} ^{i}} f }_{L^2(\mathbb R_+^2)} &\leq C\inner{\norm{W_{m} ^{i}f }_{L^2(\mathbb R_+^2)}+\norm{W_{m} ^{i} \partial_y  f }_{L^2(\mathbb R_+^2)}}\\
	 & \leq \tilde C\inner{\norm{W_{m} ^{i}f }_{L^2(\mathbb R_+^2)}+\norm{\partial_y W_{m} ^{i} f }_{L^2(\mathbb R_+^2)}}
\end{split}
\end{eqnarray} 
\begin{eqnarray}\label{+comfm++}
\begin{split}
	 \norm{\com{\partial_y^3,~ W_{m} ^{i}} f }_{L^2(\mathbb R_+^2)} & \leq C\inner{\norm{W_{m} ^{i}f }_{L^2(\mathbb R_+^2)}+\norm{W_{m} ^{i} \partial_y  f }_{L^2(\mathbb R_+^2)}+\norm{W_{m} ^{i} \partial_y^2  f }_{L^2(\mathbb R_+^2)}}\\
	 & \leq \tilde C\inner{\norm{W_{m} ^{i}f }_{L^2(\mathbb R_+^2)}+\norm{\partial_y W_{m} ^{i}   f }_{L^2(\mathbb R_+^2)}+\norm{\partial_y^2W_{m} ^{i}   f }_{L^2(\mathbb R_+^2)}}
	 \end{split}
\end{eqnarray} 
hold for all integers $m, i$ with $m\geq 1$ and   $0\leq i\leq 3$, where $C, \tilde C$ are two constants independent of $m.$

\begin{lemma}\label{uf} 
Under the assumption \eqref{1.3} and \eqref{1.5}.   Let $c$ be the constant given in  \reff{wmi+++++},  and  $\Lambda^{\tau_1}, \Lambda_\delta^{\tau_2}$  be the Fourier multiplier associate with the symbols $\comii{\xi }^{\tau_1}$ and $\comii{\delta\xi }^{\tau_2}$, respectively.     Then there exists a  constant $C$,   such that  for any $m, n \geq 1, 0\leq \ell\leq 3,$ and for any $0<\tilde c<c$,  we have
 \begin{eqnarray}\label{equ1}
 \norm{ e^{\tilde cy} \Lambda^{\tau_1}\Lambda_\delta^{\tau_2} \p^{m} u}_{L^2(\RR^2_+)} \leq C\big\| \Lambda^{\tau_1}\Lambda_\delta^{\tau_2} W_{n}^\ell  f_m\big\|_{L^2(\RR^2_+)}, 
 \end{eqnarray}
 and 
  \begin{eqnarray}\label{equ2}
 \norm{ \Lambda^{\tau_1}\Lambda_\delta^{\tau_2}   \p^{m} v}_{L^\infty\inner{\mathbb R_+;~L^2(\RR_x)}} \leq C\big\| \Lambda^{\tau_1}\Lambda_\delta^{\tau_2} W_{n}^{\ell}f_{m+1}\big\|_{L^2(\RR^2_+)}.
 \end{eqnarray}
\end{lemma}

\begin{proof} In the proof we use $C$ to denote different constants which are independent of $m$. 
 Observe     $\omega\in L^\infty$ and $\omega>0$ then   
\begin{eqnarray*}
	\norm{ e^{\tilde cy } \Lambda^{\tau_1}\Lambda_\delta^{\tau_2} \p^{m} u}_{L^2(\RR^2_+)} \leq C\big\|e^{\tilde cy } \Lambda^{\tau_1}\Lambda_\delta^{\tau_2} \frac{\p^{m} u}{\omega}\big\|_{L^2(\RR^2_+)}
\end{eqnarray*}
On the other hand,  integrating by parts we have
\begin{eqnarray*}
	\big\|e^{\tilde cy }  \Lambda^{\tau_1}\Lambda_\delta^{\tau_2}  \frac{\p^{m} u}{\omega}\big\|_{L^2(\RR^2_+)}^2&=&\int_\mathbb R \int_0^\infty e^{2\tilde cy }\inner{ \Lambda^{\tau_1}\Lambda_\delta^{\tau_2}  \frac{\p^{m} u}{\omega}}\overline{ \Lambda^{\tau_1}\Lambda_\delta^{\tau_2}  \frac{\p^{m} u}{\omega}}  dydx\\
	&=&\frac{1}{2\tilde c}\int_\mathbb R \int_0^\infty \inner{\partial_y e^{2\tilde cy }} \inner{ \Lambda^{\tau_1}\Lambda_\delta^{\tau_2}  \frac{\p^{m} u}{\omega}}\overline{ \Lambda^{\tau_1}\Lambda_\delta^{\tau_2}  \frac{\p^{m} u}{\omega}}  dydx\\
	&=&-\frac{1}{2\tilde c}\int_\mathbb R \int_0^\infty   e^{2\tilde cy} \bigg[\partial_y \inner{ \Lambda^{\tau_1}\Lambda_\delta^{\tau_2}  \frac{\p^{m} u}{\omega}}\bigg]\overline{ \Lambda^{\tau_1}\Lambda_\delta^{\tau_2}  \frac{\p^{m} u}{\omega}}  dydx\\
	&&-\frac{1}{2\tilde c}\int_\mathbb R \int_0^\infty   e^{2\tilde cy}  \inner{ \Lambda^{\tau_1}\Lambda_\delta^{\tau_2}  \frac{\p^{m} u}{\omega}} \overline{\partial_y \Lambda^{\tau_1}\Lambda_\delta^{\tau_2}  \frac{\p^{m} u}{\omega}}  dydx\\
	&\leq &\frac{1}{\tilde c} \norm {e^{\tilde cy }  \Lambda^{\tau_1}\Lambda_\delta^{\tau_2}  \frac{\p^{m} u}{\omega}}_{L^2(\RR^2_+)}
 \norm{ e^{\tilde cy }  \Lambda^{\tau_1}\Lambda_\delta^{\tau_2}\partial_y\inner{\frac{\p^{m} u}{\omega}} }_{L^2(\RR^2_+)},  
\end{eqnarray*}
which implies
\begin{eqnarray*}
	\big\|e^{\tilde cy } \Lambda^{\tau_1}\Lambda_\delta^{\tau_2}  \frac{\p^{m} u}{\omega}\big\|_{L^2(\RR^2_+)}&\leq & 
 \norm{ e^{\tilde cy }   \Lambda^{\tau_1}\Lambda_\delta^{\tau_2}  \partial_y\inner{\frac{\p^{m} u}{\omega}} }_{L^2(\RR^2_+)}\\
& = &   
 \norm{ \Lambda^{\tau_1}\Lambda_\delta^{\tau_2}   e^{\tilde cy }  \omega^{-1}\omega \partial_y\inner{\frac{\p^{m} u}{\omega}} }_{L^2(\RR^2_+)}\\
& \leq  &  
 \norm{ e^{\tilde cy }  \omega^{-1}  \Lambda^{\tau_1}\Lambda_\delta^{\tau_2}  \omega \partial_y\inner{\frac{\p^{m} u}{\omega}} }_{L^2(\RR^2_+)}+\norm{ \com{e^{\tilde cy }  \omega^{-1}, ~ \Lambda^{\tau_1}\Lambda_\delta^{\tau_2} } \omega \partial_y\inner{\frac{\p^{m} u}{\omega}} }_{L^2(\RR^2_+)}.
\end{eqnarray*}
Thus we have, by the above inequalities, 
\begin{eqnarray*}
	\norm{e^{\tilde cy }  \Lambda^{\tau_1}\Lambda_\delta^{\tau_2} \p^{m} u}_{L^2(\RR^2_+)} \leq C\norm{ e^{\tilde cy }  \omega^{-1}  \Lambda^{\tau_1}\Lambda_\delta^{\tau_2}  \omega \partial_y\inner{\frac{\p^{m} u}{\omega}} }_{L^2(\RR^2_+)}+C\norm{ \com{e^{\tilde cy }  \omega^{-1}, ~ \Lambda^{\tau_1}\Lambda_\delta^{\tau_2} } \omega \partial_y\inner{\frac{\p^{m} u}{\omega}} }_{L^2(\RR^2_+)}.
\end{eqnarray*}
On the other hand,  \eqref{1.3} and \eqref{1.5}  enables us to use Lemma \ref{lemma3.1}  to obtain
\begin{eqnarray*}
	\norm{ \com{e^{\tilde cy }  \omega^{-1}, ~ \Lambda^{\tau_1}\Lambda_\delta^{\tau_2} } \omega \partial_y\inner{\frac{\p^{m} u}{\omega}} }_{L^2(\RR^2_+)}&\leq& C \norm{  \Lambda^{\tau_1}\Lambda_\delta^{\tau_2}  \omega \partial_y\inner{\frac{\p^{m} u}{\omega}} }_{L^2(\RR^2_+)}\\
	&\leq& C \norm{ e^{\tilde cy }  \omega^{-1}  \Lambda^{\tau_1}\Lambda_\delta^{\tau_2}  \omega \partial_y\inner{\frac{\p^{m} u}{\omega}} }_{L^2(\RR^2_+)}.
\end{eqnarray*}
As a result, 
\begin{eqnarray*}
	\norm{ e^{\tilde cy } \Lambda^{\tau_1}\Lambda_\delta^{\tau_2}  \p^{m} u}_{L^2(\RR^2_+)}  \leq   C
 \norm{ e^{\tilde cy }  \omega^{-1}  \Lambda^{\tau_1}\Lambda_\delta^{\tau_2}  \omega \partial_y\inner{\frac{\p^{m} u}{\omega}} }_{L^2(\RR^2_+)}\leq   C
 \norm{    \Lambda^{\tau_1}\Lambda_\delta^{\tau_2}  W_n^\ell f_m }_{L^2(\RR^2_+)},
\end{eqnarray*}
the last inequality using the fact that $f_m=\omega \partial_y\inner{\frac{\p^{m} u}{\omega}}$,  and that 
\begin{eqnarray*}
	 e^{\tilde cy }  \omega^{-1}\leq C e^{\tilde cy } (1+y)^{\sigma}\leq C e^{2cy}\inner{1+\frac{2cy}{\gamma}}^{-\gamma/2}
\end{eqnarray*}
for any $\gamma\geq 1.$ 
This is just the desired \reff{equ1}.  Now  we prove \reff{equ2}. 
Recall $v(t, x,y)=-\int_0^y \p u(t, x,y')dy'$. Then   we have
\begin{eqnarray*}
\Lambda^{\tau_1}\Lambda_\delta^{\tau_2} \p^{m}v=-\int_{0}^{y} \Lambda^{\tau_1}\Lambda_\delta^{\tau_2} \p^{m+1} u(x,y')dy'
\end{eqnarray*}
Therefore
\begin{eqnarray*}
\|\Lambda^{\tau_1}\Lambda_\delta^{\tau_2} \p^{m}v\|_{L^\infty\inner{\mathbb R_+;~L^2(\mathbb R_x)}}&\leq &\norm{e^{-\tilde cy} }_{L^2\inner{\mathbb R_+}}\norm{e^{\tilde cy} \Lambda^{\tau_1}\Lambda_\delta^{\tau_2} \partial_x^{m+1} u}_{L^2\inner{\mathbb R_+^2}}\\
&\leq &C \norm{\Lambda^{\tau_1}\Lambda_\delta^{\tau_2}  W^\ell_{n} f_{m+1}}_{L^2\inner{\mathbb R_+^2}},
\end{eqnarray*}
the last inequality using \reff{equ1}. Thus the desired \reff{equ2} follows and the proof of Lemma \ref{uf} is complete. 
\end{proof}

We prove now Lemma \ref{lemma2.1}, recall
$$
f_m=\partial_x^m \omega-\frac{\partial_y \omega}{\omega}\partial_x^m u=\omega\partial_y\inner{\frac{\partial_x^mu}{\omega}}\,. 
$$
\begin{lemma}\label{lemma4.2}
  There exists a  constant $C$,   such that
 \begin{eqnarray}\label{equ1+}
 \norm{  \comii y^{-1}W_m^\ell \p^{m} u}_{L^2(\RR^2_+)}+\norm{  \comii y^{-1}W_m^\ell \p^{m}\omega}_{L^2(\RR^2_+)} \leq C\big\|  W_m^\ell f_m \big\|_{L^2(\RR^2_+)}.
 \end{eqnarray}
 As a result,  for some constant $\tilde C,$
 \begin{eqnarray*}
\big\|  \Lambda^{-1} W^0_{m}  f_{m+1} \big\|_{L^2(\RR^2_+)}\leq \tilde {C} \norm{   W^0_{m} f_{m}}_{L^2(\RR^2_+)},
 \end{eqnarray*}
 and 
  \begin{eqnarray*} 
\big\|  \Lambda^{-1} \partial_y  W^0_{m}  f_{m+1} \big\|_{L^2(\RR^2_+)}\leq \tilde {C}\inner{ \norm{ \partial_y  W^0_{m} f_{m}}_{L^2(\RR^2_+)}+\big\|     W^0_{m}  f_{m} \big\|_{L^2(\RR^2_+)}}. 
 \end{eqnarray*}
\end{lemma}

\begin{proof}
In the proof we use $C$ to denote different constants  which  depend  only on $\sigma,$ $c$, and $ C_*$  and are independent of $m$.  We first prove \reff{equ1+}.  Observe 
\begin{eqnarray*}
	\omega \comii y^{-1} \inner{1+\frac{2cy}{(3m+\ell)\sigma}}^{-(3m+\ell)\sigma/2} (1+cy)^{-1} &\leq& C (1+y)^{-\sigma-1}\inner{1+\frac{2cy}{(3m+\ell)\sigma}}^{-(3m+\ell)\sigma/2}\\
	&\leq& C R^{\sigma+1} (R+y)^{-\sigma-1}\inner{1+\frac{2cy}{(3m+\ell)\sigma}}^{-(3m+\ell)\sigma/2},
\end{eqnarray*}
where
  $R\geq1$ is a large number to be determined later. Thus using the notation 
  \begin{eqnarray*}
	b_{m,\ell}^R(y)=\inner{1+\frac{2cy}{(3m+\ell)\sigma}}^{-(3m+\ell)\sigma/2}(R+y)^{-\sigma-1},
\end{eqnarray*}
we have
\begin{eqnarray*}
\norm{\comii y^{-1}W_m^\ell \p^{m} u}_{L^2(\RR^2_+)} &=&\norm{\comii y^{-1}W_m^\ell(\omega \frac{\p^{m} u}{\omega})}_{L^2(\RR^2_+)}\\
&\leq&\norm{\omega\comii y^{-1}W^{\ell}_{m}\frac{\p^m u}{\omega}}_{L^2(\RR^2_+)}
+\norm{\comii y^{-1}\com{W^{\ell}_m,~\omega}\frac{\p^m u}{\omega}}_{L^2(\RR^2_+)}\\
&\leq& CR^{\sigma+1}\norm{ e^{2cy} b_{m,\ell}^R \frac{\Lambda^{\ell/3}\p^{m} u}{\omega}}_{L^2(\RR^2_+)}
+\norm{\comii y^{-1}\com{W^{\ell}_m,~\omega}\frac{\p^m u}{\omega}}_{L^2(\RR^2_+)}.
\end{eqnarray*}
  On the other hand,  using Lemma \ref{lemma3.1} 
\begin{eqnarray*}
\norm{\comii y^{-1}\com{W^{\ell}_m,~\omega}\frac{\p^m u}{\omega}}_{L^2(\RR^2_+)}\leq R\|\com{\Lambda^{\frac{l}{3}},~\omega}e^{2cy}b^{R}_{m,\ell}\frac{\p^m u}{\omega}\|_{L^2(\RR^2_+)}
\leq CR \|e^{2cy}b^{R}_{m,\ell}\frac{\p^m u}{\omega}\|_{L^2(\RR^2_+)}
\end{eqnarray*}
Combining these inequalities we conclude
		\begin{eqnarray}\label{lamd+}
		\norm{\comii y^{-1}W_m^\ell \p^{m} u}_{L^2(\RR^2_+)} \leq  CR^{\sigma+1} \norm{ e^{2cy} b_{m,\ell}^R \Lambda^{\ell/3} \frac{\p^{m} u}{\omega}}_{L^2(\RR^2_+)}.
	\end{eqnarray}
Moreover, observe $u\big|_{y=0}=0$ and thus we have, by integrating by parts, 
\begin{eqnarray*}
	 \norm{e^{2cy} b_{m,\ell}^R \Lambda^{\ell/3} \frac{ \p^{m} u}{\omega}}_{L^2(\RR^2_+)}^2&=&\int_\mathbb R \int_0^\infty e^{4cy}\inner{b_{m,\ell}^R(y)}^2\inner{ \Lambda^{\ell/3} \frac{ \p^{m} u}{\omega}}  \overline{ \Lambda^{\ell/3}  \frac{  \p^{m} u}{\omega}}  dydx\\
	&=&\frac{1}{4c}\int_\mathbb R \int_0^\infty \inner{\partial_y e^{4cy}}\inner{b_{m,\ell}^R(y)}^2\inner{\Lambda^{\ell/3}\frac{ \p^{m} u}{\omega}}  \overline{ \Lambda^{\ell/3}\frac{ \p^{m} u}{\omega}}  dydx\\
	&=&-\frac{1}{2c}\int_\mathbb R \int_0^\infty    e^{4cy} b_{m,\ell}^R(y) \inner{\partial_y b_{m,\ell}^R(y)} \inner{ \Lambda^{\ell/3} \frac{  \p^{m} u}{\omega}}  \overline{ \Lambda^{\ell/3} \frac{  \p^{m} u}{\omega}}    dydx\\
	&&-\frac{1}{4c}\int_\mathbb R \int_0^\infty    e^{4cy}  \inner{  b_{m,\ell}^R(y)}^2 \bigg[\partial_y\inner{\Lambda^{\ell/3} \frac{  \p^{m} u}{\omega}}\bigg]  \overline{ \Lambda^{\ell/3}\frac{ \p^{m} u}{\omega}}    dydx\\
	&&-\frac{1}{4c}\int_\mathbb R \int_0^\infty    e^{4cy}  \inner{  b_{m,\ell}^R(y)}^2  \inner{\Lambda^{\ell/3} \frac{  \p^{m} u}{\omega}}  \overline{\partial_y\Lambda^{\ell/3} \inner{ \frac{  \p^{m} u}{\omega}}}    dydx,  
\end{eqnarray*}
which, along with the  estimate
\begin{eqnarray*} 
\abs{\partial_y b^R_{m,\ell}}\leq \inner{c+(\sigma+1)R^{-1}} b_{m,\ell}^R,
\end{eqnarray*}
gives
\begin{eqnarray*}
	 \norm{e^{2cy} b_{m,\ell}^R \Lambda^{\ell/3} \frac{ \p^{m} u}{\omega}}_{L^2(\RR^2_+)}^2	&\leq & \frac{c+(\sigma+1)R^{-1}}{2c} \norm{e^{2cy} b_{m,\ell}^R\Lambda^{\ell/3} \frac{   \p^{m} u}{\omega}}_{L^2(\RR^2_+)}^2  \\
		&&+\frac{1}{2c} \norm{e^{2cy} b_{m,\ell}^R\Lambda^{\ell/3} \frac{  \p^{m} u}{\omega}}_{L^2(\RR^2_+)}\norm{e^{2cy} b_{m,\ell}^R\partial_y\inner{\Lambda^{\ell/3} \frac{  \p^{m} u}{\omega}}}_{L^2(\RR^2_+)}.
\end{eqnarray*}
Now we choose $R=1+2(\sigma+1)c^{-1}$, which gives $R\geq 1$ and 
\begin{eqnarray*}
	(\sigma+1)R^{-1}\leq \frac{c}{2}.
\end{eqnarray*}
Then we deduce, from the above inequalities, 
 \begin{eqnarray*}
	 \norm{e^{2cy} b_{m,\ell}^R\Lambda^{\ell/3} \frac{   \p^{m} u}{\omega}}_{L^2(\RR^2_+)} 
	\leq \frac{2}{c}\norm{e^{2cy} b_{m,\ell}^R\partial_y\Lambda^{\ell/3} \inner{\frac{  \p^{m} u}{\omega}}}_{L^2(\RR^2_+)}.
	\end{eqnarray*}
Moreover, observe $R\geq c^{-1}+1$ and the monotonicity  assumption  $\omega\geq C_*^{-1}(1+y)^{-\sigma}$,   and thus 
\begin{eqnarray*}
	  b_{m,\ell}^R& \leq& c(1+y)^{-\sigma}(1+cy)^{-1}\inner{1+\frac{2cy}{(3m+\ell)\sigma}}^{-(3m+\ell)\sigma/2}\\
	&\leq & c\,C_* \omega(1+cy)^{-1}\inner{1+\frac{2cy}{(3m+\ell)\sigma}}^{-(3m+\ell)\sigma/2}. 
\end{eqnarray*}
As a result,  we obtain 
 \begin{eqnarray*}
	 \norm{e^{2cy} b_{m,\ell}^R\Lambda^{\ell/3} \frac{   \p^{m} u}{\omega}}_{L^2(\RR^2_+)} 
	&\leq & C_* \norm{\omega W^\ell_m   \partial_y \inner{\frac{  \p^{m} u}{\omega}}}_{L^2(\RR^2_+)},  
\end{eqnarray*}
which along with  \reff{lamd+} gives
\begin{eqnarray*}
	&&\norm{\comii y^{-1}W_m^\ell \p^{m} u}_{L^2(\RR^2_+)} \\
	&\leq &C\norm{\omega W^\ell_m   \partial_y \inner{\frac{  \p^{m} u}{\omega}}}_{L^2(\RR^2_+)}\\
	&\leq &C\norm{ W^\ell_m  \omega  \partial_y \inner{\frac{  \p^{m} u}{\omega}}}_{L^2(\RR^2_+)}+C\norm{\com{\omega, ~W^\ell_m}   \partial_y \inner{\frac{  \p^{m} u}{\omega}}}_{L^2(\RR^2_+)}.
\end{eqnarray*}
Using the notation $\rho_{m,\ell}(y)=e^{2cy}\inner{1+\frac{2cy}{(3m+\ell)\sigma}}^{-(3m+\ell)\sigma/2}(1+cy)^{-1}$,
\begin{eqnarray*}
	\norm{\com{\omega, ~W^\ell_m}   \partial_y \inner{\frac{  \p^{m} u}{\omega}}}_{L^2(\RR^2_+)} &=& \norm{\com{\omega , ~\Lambda^{\ell/3}} \rho_{m,\ell}(y) \partial_y \inner{\frac{  \p^{m} u}{\omega}}}_{L^2(\RR^2_+)}\\
	&=& \norm{\com{\omega \comii y^\sigma, ~\Lambda^{\ell/3}}\comii y^{-\sigma} \rho_{m,\ell}(y) \partial_y \inner{\frac{  \p^{m} u}{\omega}}}_{L^2(\RR^2_+)}\\
		&\le & C \norm{ \comii y^{-\sigma}\rho_{m,\ell}(y) \partial_y \inner{\frac{  \p^{m} u}{\omega}}}_{L^2(\RR^2_+)}\\
	&\leq & C \norm{ \rho_{m,\ell}(y) \omega \partial_y \inner{\frac{  \p^{m} u}{\omega}}}_{L^2(\RR^2_+)}\\
	&\leq & C \norm{  W_m^\ell  \omega \partial_y \inner{\frac{  \p^{m} u}{\omega}}}_{L^2(\RR^2_+)}.
\end{eqnarray*}
Then, combining these inequalities we conclude, 
\begin{eqnarray*}
	 \norm{\comii y^{-1}W_m^\ell \p^{m} u}_{L^2(\RR^2_+)}  
	\leq C\norm{ W^\ell_m  \omega  \partial_y \inner{\frac{  \p^{m} u}{\omega}}}_{L^2(\RR^2_+)}=C\norm{ W^\ell_m  f_m}_{L^2(\RR^2_+)}.
	\end{eqnarray*}
	For the other terms in \reff{equ1+}, we have 
	\begin{eqnarray*}
	\begin{split}
	& \norm{\comii y^{-1}W_m^\ell \p^{m} \omega}_{L^2(\RR^2_+)} \\
	\leq&  ~\norm{\comii y^{-1}W_m^\ell f_m}_{L^2(\RR^2_+)}+ \norm{\comii y^{-1}W_m^\ell \big((\partial_y\omega)/\omega\big)\p^{m} u}_{L^2(\RR^2_+)}\\   
	\leq&  ~\norm{\comii y^{-1}W_m^\ell f_m}_{L^2(\RR^2_+)}+ \norm{\big((\partial_y\omega)/\omega\big) \comii y^{-1}W_m^\ell \p^{m} u}_{L^2(\RR^2_+)}\\   
	 &+ \norm{\com{(\partial_y\omega)/\omega,~W_m^\ell }\comii y^{-1}\p^{m} u}_{L^2(\RR^2_+)}\\
	  \leq& ~ \norm{\comii y^{-1}W_m^\ell f_m}_{L^2(\RR^2_+)}+ C\norm{  \comii y^{-1}W_m^\ell \p^{m} u}_{L^2(\RR^2_+)}\\
	\leq &   ~C\norm{ W^\ell_m  f_m}_{L^2(\RR^2_+)},
	\end{split}
	\end{eqnarray*}
	Thus the desired estimate \reff{equ1+} follows.   As a result,   we have
	\begin{eqnarray*}
		\norm{\Lambda^{-1}    W^0_{m}  f_{m+1}}_{L^2(\RR^2_+)}&\leq  &  \norm{\Lambda^{-1}   W^0_{m}    \partial_x f_{m}}_{L^2(\RR^2_+)}+   \norm{\Lambda^{-1}   W^0_{m}   \Big[ \partial_x\big((\partial_y\omega)/\omega\big) \Big]\partial_x^m u}_{L^2(\RR^2_+)}\\
		&\leq  &  \norm{    W^0_{m}   f_{m}}_{L^2(\RR^2_+)}+   \norm{\comii y ^{-1}   W^0_{m} \partial_x^m u}_{L^2(\RR^2_+)}\leq C\norm{    W^0_{m}   f_{m}}_{L^2(\RR^2_+)}. 
	\end{eqnarray*}
	Similarly,  we can deduce that,  using \reff{comfm}, 
	 \begin{eqnarray*}
\big\|  \Lambda^{-1} \partial_y  W^0_{m}  f_{m+1} \big\|_{L^2(\RR^2_+)}\leq C \inner{ \norm{ \partial_y  W^0_{m} f_{m}}_{L^2(\RR^2_+)}+\big\|     W^0_{m}  f_{m} \big\|_{L^2(\RR^2_+)}}. 
 \end{eqnarray*}
	Thus the proof of Lemma \ref{lemma4.2} is   complete. 
 \end{proof}

We prove now the Lemma \ref{le+2.2+} by the following 2 lemmas . 
 
\begin{lemma}\label{equ+}
There exists a constant $C$ such that, for any  $m\geq 1$  and $1\leq \ell \leq 3$,     
	\begin{eqnarray*}
		 \norm{\comii y^{-\sigma/2} \partial_y\Lambda^{1/3}\Lambda^{-2}_\delta W_m^{\ell-1}f_m}_{L^2(\mathbb R_+^2)}\leq  C\norm{  \partial_y \Lambda^{-2}_\delta W_m^{\ell}f_m}_{L^2(\mathbb R_+^2)}
+C \norm{  \Lambda^{-2}_\delta W_m^{\ell}f_m}_{L^2(\mathbb R_+^2)}.
	\end{eqnarray*}
\end{lemma}

\begin{proof} We can write
	\begin{eqnarray*}
		\Lambda^{1/3}\Lambda^{-2}_\delta W_m^{\ell-1}=e^{2cy}\inner{1+\frac{2cy}{(3m+\ell-1)\sigma} }^{-\frac{(3m+\ell-1)\sigma}{2}}
(1+cy)^{-1} \Lambda^{{\ell\over 3}}\Lambda^{-2}_\delta =a_{m,\ell}(y)\Lambda^{-2}_\delta W_m^\ell,
	\end{eqnarray*}
	where 
	\begin{eqnarray*}
		a_{m,\ell}(y)=\inner{1+\frac{2cy}{(3m+\ell-1)\sigma} }^{-\frac{(3m+\ell-1)\sigma}{2}}\inner{1+\frac{2cy}{(3m+\ell)\sigma} }^{\frac{(3m+\ell)\sigma}{2}}.
	\end{eqnarray*}
Direct computation gives
\begin{eqnarray*}
	\abs{a_{m,\ell}(y)}&=&\inner{1+\frac{2cy}{(3m+\ell-1)\sigma} }^{\sigma/2} \inner{1+\frac{2cy}{(3m+\ell-1)\sigma} }^{-\frac{(3m+\ell)\sigma}{2}}\inner{1+\frac{2cy}{(3m+\ell)\sigma} }^{\frac{(3m+\ell)\sigma}{2}}\\
	&\leq& \inner{1+\frac{2cy}{(3m+\ell-1)\sigma} }^{\sigma/2}  \leq  C \comii{y}^{\sigma/2}.
	\end{eqnarray*}
Moreover observe  $\abs{\partial_y a_{m,\ell}(y)}\leq 2c\abs{a_{m,\ell}(y)}$, and thus 
	\begin{eqnarray*}
	\abs{\partial_y a_{m,\ell}(y)} \leq     C \comii{y}^{\sigma/2}.
	\end{eqnarray*}
As a result, 
\begin{eqnarray*}
	&&\norm{\comii y^{-\sigma/2} \partial_y\Lambda^{1/3}\Lambda^{-2}_\delta W_m^{\ell-1}f_m}_{L^2(\mathbb R_+^2)}=\norm{\comii y^{-\sigma/2} \partial_y\left(a_{m,\ell}\Lambda^{-2}_\delta W_m^{\ell}f_m\right)}_{L^2(\mathbb R_+^2)}\\
	&\leq & \norm{\comii y^{-\sigma/2} a_{m,\ell} \partial_y \Lambda^{-2}_\delta W_m^{\ell}f_m}_{L^2(\mathbb R_+^2)}+\norm{\comii y^{-\sigma/2} \inner{\partial_y a_{m,\ell}}\Lambda^{-2}_\delta W_m^{\ell}f_m}_{L^2(\mathbb R_+^2)}\\
	&\leq & C\bigg(\norm{  \partial_y \Lambda^{-2}_\delta W_m^{\ell}f_m}_{L^2(\mathbb R_+^2)}
+ \norm{   \Lambda^{-2}_\delta W_m^{\ell}f_m}_{L^2(\mathbb R_+^2)}\bigg).
\end{eqnarray*}
The  proof of Lemma \ref{equ+} is thus complete. 
\end{proof}

\begin{lemma}\label{E1}
There exists a constant $C$,  depending only on $\sigma,$ $c$, and $C_*$ ,    such that for any integers $m\geq N_0+1$,   we have
\begin{eqnarray*}
	&&\norm{  \phi_{m+1}^0
W^0_{m+1} f_{m+1}}_{L^\infty( [0,T];~L^2( \mathbb
  R_+^2))}+\sum_{j=1}^2 \norm{\partial_y^j \Lambda^{-\frac{2(j-1)}{3}} \phi_{m+1}^0
W^0_{m+1} f_{m+1}}_{L^2( [0,T]\times \mathbb
  R_+^2)} \\
  &\leq& ~C\norm{    \phi_{m}^{3}
W^{3}_{m}  f_{m}}_{L^\infty( [0,T];~ L^2(\mathbb
  R_+^2))} +C \sum_{j=1}^3\norm{\partial_y^j\Lambda^{-\frac{2(j-1)}{3}} \phi_{m}^3
W^{3}_{m} f_{m}}_{L^2( [0,T]\times \mathbb
  R_+^2)},
\end{eqnarray*}
and
\begin{eqnarray*}
	&& \norm{\partial_y^3 \Lambda^{-1} \phi_{m+1}^0
W^0_{m+1} f_{m+1}}_{L^2( [0,T]\times \mathbb
  R_+^2)} \\
&  \leq& ~C\norm{    \phi_{m}^{3}
W^{3}_{m}  f_{m}}_{L^\infty( [0,T];~ L^2(\mathbb
  R_+^2))} +C\sum_{j=1}^2 \norm{\partial_y^j \Lambda^{-\frac{2(j-1)}{3}} \phi_{m}^3
W^{3}_{m} f_{m}}_{L^2( [0,T]\times \mathbb
  R_+^2)}\\
  &&+C \norm{\partial_y^3\Lambda^{-1} \phi_{m}^3
W^{3}_{m} f_{m}}_{L^2( [0,T]\times \mathbb
  R_+^2)}.
\end{eqnarray*}
\end{lemma}

\begin{proof} 
In the proof we use $C$ to denote different constants which are independent of $m$.  In view of the definition \reff{fmfun+} of $f_m$, 
we have, observing \reff{phi}, 
\begin{eqnarray*}
	&& \norm{  \phi_{m+1}^0
W^0_{m+1} f_{m+1}}_{L^\infty( [0,T];~L^2( \mathbb
  R_+^2))} \\
  &\leq & \norm{  \phi_{m+1}^0
W^0_{m+1} \partial_x f_{m}}_{L^\infty( [0,T];~L^2( \mathbb
  R_+^2))}+\norm{  \phi_{m+1}^0
W^0_{m+1} \big[ \partial_x \big((\partial_y\omega)/\omega \big) \big]\partial_x^m u}_{L^\infty( [0,T];~L^2( \mathbb
  R_+^2))}\\
  &\leq & \norm{  \phi_{m}^3
W^0_{m+1} \Lambda^1 f_{m}}_{L^\infty( [0,T];~L^2( \mathbb
  R_+^2))}+C \norm{  \comii y^{-1} \phi_{m}^3
W^3_{m}  \partial_x^m u}_{L^\infty( [0,T];~L^2( \mathbb
  R_+^2))}\\
  &\leq &C  \norm{  \phi_{m}^3
W^3_{m} f_{m}}_{L^\infty( [0,T];~L^2( \mathbb
  R_+^2))},
\end{eqnarray*}
the last inequality using \reff{equ1+} and \reff{+wmi}.   Similarly,  using  \reff{comfm},  we can deduce that 
\begin{eqnarray*}
	 \norm{  \partial_y \phi_{m+1}^0
W^0_{m+1} f_{m+1}}_{L^\infty( [0,T];~L^2( \mathbb
  R_+^2))} \leq  C \norm{ \partial_y  \phi_{m}^3
W^3_{m} f_{m}}_{L^\infty( [0,T];~L^2( \mathbb
  R_+^2))}+C \norm{  \phi_{m}^3
W^3_{m} f_{m}}_{L^\infty( [0,T];~L^2( \mathbb
  R_+^2))}. 
\end{eqnarray*}
The other terms \[ \norm{  \partial_y^2\Lambda^{-2/3} \phi_{m+1}^0
W^0_{m+1} f_{m+1}}_{L^\infty( [0,T];~L^2( \mathbb
  R_+^2))}, \quad  \norm{  \partial_y^3\Lambda^{-1} \phi_{m+1}^0
W^0_{m+1} f_{m+1}}_{L^\infty( [0,T];~L^2( \mathbb
  R_+^2))}\] 
  can  treated in the same way,  thanks to \reff{comfm+} and  \reff{+comfm++}.   So we omit it here.  
 Thus the  proof  of Lemma \ref{E1} is   complete. 
\end{proof}

\begin{proof}[Proof of Lemma \ref{lemma261}]
Observe
\begin{eqnarray*}
(1+y)^{-\frac{\sigma}{2}}&=&\inner{\frac{(3m+\ell-1)\sigma}{2c} }^{-\frac{\sigma}{2}}
\inner{\frac{2c}{(3m+\ell-1)\sigma}
+\frac{2cy}{(3m+\ell-1)\sigma} }^{-\frac{\sigma}{2}}\\
&\geq& C \inner{\frac{(3m+\ell-1)\sigma}{2c} }^{-\frac{\sigma}{2}}
\inner{1+\frac{2cy}{(3m+\ell-1)\sigma} }^{-\frac{\sigma}{2}}\\
&\geq& Cm ^{-\frac{\sigma}{2}}
\inner{1+\frac{2cy}{(3m+\ell-1)\sigma} }^{-\frac{\sigma}{2}}.
\end{eqnarray*}
Then 
\begin{eqnarray}\label{lowb}
(1+y)^{-\frac{\sigma}{2}}\inner{1+\frac{2cy}{(3m+\ell-1)\sigma}}^{-\frac{(3m+\ell-1)\sigma}{2}}
\geq C m^{-\frac{\sigma}{2}}
\inner{1+\frac{2cy}{(3m+\ell-1)\sigma}}^{-\frac{(3m+\ell)\sigma}{2}}.
\end{eqnarray}
Moreover we find
\begin{eqnarray*}
\inner{1+\frac{2cy}{(3m+\ell-1)\sigma}}^{-\frac{(3m+\ell)\sigma}{2}}
&=&\inner{\frac{1}{(3m+\ell-1)\sigma }}^{-\frac{(3m+\ell)\sigma}{2}}
\inner{(3m+\ell-1)\sigma+2cy}^{-\frac{(3m+\ell)\sigma}{2}}\\
&=&\inner{\frac{(3m+\ell)\sigma}{(3m+\ell-1)\sigma} }^{-\frac{(3m+\ell)\sigma}{2}}
\inner{\frac{(3m+\ell-1)}{(3m+\ell)}+\frac{2cy}{(3m+\ell)\sigma} }^{-\frac{(3m+\ell)\sigma}{2}}\\
&\geq&\inner{\frac{ 3m+\ell}{3m+\ell-1} }^{-\frac{(3m+\ell)\sigma}{2}}
\inner{1+\frac{2cy}{(3m+\ell)\sigma} }^{-\frac{(3m+\ell)\sigma}{2}}\\
&\geq&C
\inner{1+\frac{2cy}{(3m+\ell)\sigma} }^{-\frac{(3m+\ell)\sigma}{2}},
\end{eqnarray*}
which along with \reff{lowb} gives
\begin{eqnarray*}
 \inner{1+\frac{2cy}{(3m+\ell)\sigma} }^{-\frac{(3m+\ell)\sigma}{2}}\leq C m^{\frac{\sigma}{2}}(1+y)^{-\frac{\sigma}{2}}\inner{1+\frac{2cy}{(3m+\ell-1)\sigma}}^{-\frac{(3m+\ell-1)\sigma}{2}}.
\end{eqnarray*}
As a result,  recalling
\begin{eqnarray*}
(1+y)^{-\frac{\sigma}{2}}\Lambda^{1/3}W^{\ell -1}_m
=(1+y)^{-\frac{\sigma}{2}}e^{2cy}\inner{1+\frac{2cy}{(3m+\ell-1)\sigma}}^{-\frac{(3m+\ell-1)\sigma}{2}}
(1+cy)^{-1}\Lambda^{\frac{\ell }{3}},
\end{eqnarray*}
we have, observing  $\phi^{-\frac{1}{2}}\phi^\ell_m=\phi^{\frac{1}{2}}\phi^{\ell-1}_m $
\begin{eqnarray*}
\|\phi^{-\frac{1}{2}} \Lambda^{-2}_\delta
\phi^\ell_m W^\ell_m f_m\|_{L^2( [0,T]\times \mathbb
  R_+^2)}  \leq C m^{\sigma/2}\norm{(1+y)^{-\frac{\sigma}{2}}
\Lambda^{1/3} \phi^{\frac{1}{2}} \Lambda^{-2}_\delta \phi^{\ell-1}_mW^{\ell-1}_m f_m }_{L^2( [0,T]\times \mathbb
  R_+^2)}, 
  \end{eqnarray*}
 that is,  recalling $F=\Lambda_\delta^{-2} \phi^{\ell}_{m}W_{m}^\ell f_{m}$ and $f=\phi^{1/2}\Lambda_\delta^{-2} \phi^{\ell-1}_{m}W_{m}^{\ell-1} f_{m}$, 
 \begin{eqnarray*}
 	\| \phi^{-1/2} F\|_{L^2( [0,T]\times \mathbb
  R_+^2)}\leq C m^{\sigma/2} \norm{\comii y ^{-\frac{\sigma}{2}}
\Lambda^{1/3}f }_{L^2( [0,T]\times \mathbb
  R_+^2)}.
 \end{eqnarray*}
 Moreover,   using \reff{+wmi} and \reff{comfm+} we have,  observing $\phi_m^\ell\leq \phi^{1/2}\phi^{\ell-1}_m$,
  \begin{eqnarray*}
  &&	\| \partial_y^2\Lambda^{-2/3}  F\|_{L^2( [0,T]\times \mathbb
  R_+^2)}= \norm{\partial_y^2\Lambda^{-2/3}\Lambda^{-2}_\delta \phi_m^\ell 
W^\ell_m f_m}_{L^2( [0,T]\times \mathbb
  R_+^2)} \\
 &\leq&  C \norm{\partial_y^2\Lambda^{-1/3}\Lambda^{-2}_\delta \phi^\ell_m  
W^{\ell-1}_m f_m}_{L^2( [0,T]\times \mathbb
  R_+^2)} +C \| \partial_y \Lambda^{-1/3} \Lambda^{-2}_\delta \phi^{\ell}_{m}W_{m}^{\ell-1} f_{m}\|_{L^2( [0,T]\times \mathbb
  R_+^2)}\\
  &&+C \norm{\Lambda^{-1/3} \Lambda^{-2}_\delta \phi^{\ell}_{m}W_{m}^{\ell-1} f_{m}}_{L^2( [0,T]\times \mathbb
  R_+^2)}\\
 & \leq & C \norm{\partial_y^2\Lambda^{-1/3}f}_{L^2( [0,T]\times \mathbb
  R_+^2)} +C\inner{ \|  \phi^{\ell-1}_{m}W_{m}^{\ell-1} f_{m}\|_{L^2( [0,T]\times \mathbb
  R_+^2)}+\norm{\partial_y  \phi^{\ell-1}_{m}W_{m}^{\ell-1} f_{m}}_{L^2( [0,T]\times \mathbb
  R_+^2)}}.
  \end{eqnarray*}
 Then combining the above inequalities,  the first estimate in Lemma \ref{lemma261} follows.  The second one can be deduced similarly. In fact
 using   \reff{+wmi} and \reff{+comfm++}  gives
  \begin{eqnarray*}
  	&&  \norm{\partial_y^3\Lambda^{-1}F }_{L^2( [0,T]\times \mathbb
  R_+^2)}=\norm{\partial_y^3\Lambda^{-1}\Lambda^{-2}_\delta \phi_m^\ell 
W^\ell_m f_m}_{L^2( [0,T]\times \mathbb
  R_+^2)}	\\
 & \leq & C \norm{\partial_y^3\Lambda^{-2/3}\Lambda^{-2}_\delta \phi_m^\ell 
W^{\ell-1}_m f_m}_{L^2( [0,T]\times \mathbb
  R_+^2)}+C \norm{\partial_y^2\Lambda^{-2/3}\Lambda^{-2}_\delta \phi_m^\ell 
W^{\ell-1}_m f_m}_{L^2( [0,T]\times \mathbb
  R_+^2)}\\
 & & +C\norm{\partial_y \Lambda^{-2/3}\Lambda^{-2}_\delta   \phi_m^{\ell-1}
W^{\ell-1}_m  f_m}_{L^2( [0,T]\times \mathbb
  R_+^2)}+C\norm{ \Lambda^{-2/3}\Lambda^{-2}_\delta   \phi_m^{\ell-1}
W^{\ell-1}_m  f_m}_{L^2( [0,T]\times \mathbb
  R_+^2)}\\
 & \leq & ~C \norm{\partial_y^3\Lambda^{-2/3}f }_{L^2( [0,T]\times \mathbb
  R_+^2)}+C \norm{\partial_y^2\Lambda^{-1/3}f}_{L^2( [0,T]\times \mathbb
  R_+^2)}\\
 & & +C\inner{ \|  \phi^{\ell-1}_{m}W_{m}^{\ell-1} f_{m}\|_{L^2( [0,T]\times \mathbb
  R_+^2)}+\norm{\partial_y  \phi^{\ell-1}_{m}W_{m}^{\ell-1} f_{m}}_{L^2( [0,T]\times \mathbb
  R_+^2)}}.
  \end{eqnarray*}
  This is just the second estimate in Lemma \ref{lemma261}.  The proof is thus complete. 
\end{proof}

\section{Estimates of the nonlinear terms}\label{section5}

In this section we  estimate the nonlinear terms $\mathcal Z_{m,\ell,\delta}$ defined in \reff{2.18a},  and prove the Proposition \ref{nonli}. 
Recall
\begin{eqnarray*}
\mathcal{Z}_{m,\ell,\delta}
&=&-\sum_{j=1}^m{m\choose j}\Lambda_\delta^{-2} \phi^{\ell}_{m} W^\ell_m(\p^j u) f_{m+1-j}
-\sum_{j=1}^{m-1}{m\choose j}\Lambda_\delta^{-2} \phi^{\ell}_{m} W^\ell_m(\p^j v)\partial_y f_{m-j}\\
&&-\Lambda_\delta^{-2} \phi^{\ell}_{m}W^\ell_m\left[\partial_y\inner{\frac{\partial_y\omega}{\omega}}\right]\sum_{j=1}^{m-1}{m\choose j}
(\p^j v)(\p^{m-j}u)
-2\Lambda_\delta^{-2} \phi^{\ell}_{m}W^\ell_m\left[\partial_y\inner{\frac{\partial_y\omega}{\omega}}\right]f_m\\
&&+\Lambda_\delta^{-2} \left(\partial_t\phi^{\ell}_{m}\right)W^\ell_{m}f_{m}+\com{u \partial_x+v\partial_y-\partial_y^2, \Lambda_\delta^{-2}\phi^\ell_m W_m^\ell  } f_m\\
&=&\mathcal J_{m,\ell,\delta} +\Lambda_\delta^{-2} \left(\partial_t\phi^{\ell}_{m}\right)W^\ell_{m}f_{m}+\com{u \partial_x+v\partial_y-\partial_y^2, \Lambda_\delta^{-2}\phi^\ell_m W_m^\ell  } f_m,
\end{eqnarray*}  
where 
\begin{eqnarray*}
	\mathcal J_{m,\ell,\delta}&=&-\sum_{j=1}^m{m\choose j}\Lambda_\delta^{-2} \phi^{\ell}_{m} W^\ell_m(\p^j u) f_{m+1-j}
-\sum_{j=1}^{m-1}{m\choose j}\Lambda_\delta^{-2} \phi^{\ell}_{m} W^\ell_m(\p^j v)\partial_y f_{m-j}\\
&&-\Lambda_\delta^{-2} \phi^{\ell}_{m}W^\ell_m\left[\partial_y\inner{\frac{\partial_y\omega}{\omega}}\right]\sum_{j=1}^{m-1}{m\choose j}
(\p^j v)(\p^{m-j}u)
-2\Lambda_\delta^{-2} \phi^{\ell}_{m}W^\ell_m\left[\partial_y\inner{\frac{\partial_y\omega}{\omega}}\right]f_m.
\end{eqnarray*}

We  remark it is suffices to prove  the estimates \reff{2.20} and \reff{2.21} in Proposition \ref{nonli},  since the esimate \reff{+2.20++} can be treated exactly similar as \reff{2.20}.  
Next we will proceed to prove \reff{2.20} and \reff{2.21} through the following Proposition \ref{+prp5.1+} and Proposition \ref{prop5.1}.   Proposition \ref{prop5.1} is devoted to treating the term $\mathcal J_{m,\ell,\delta}$ in the definition of $\mathcal Z_{m,\ell,\delta}$,  while the the other two terms are estimated in Proposition \ref{+prp5.1+}. 
 
 To simplify the notations, we will use $C$ to denote different constants depending only on   $\sigma,$ $c$, and the constants $C_0, C_*$ in Theorem \ref{mainthm}, but independent of $m$ and $\delta$.

\begin{proposition}\label{+prp5.1+}
We  have,  denoting  $F=\Lambda_\delta^{-2}  \phi_m^\ell W^\ell_{m}f_{m}$ and $\tilde f=\phi^{1/2}\Lambda_\delta^{-2} \phi^{\ell-1}_{m}W^{\ell-1}_m  f_m,$
\begin{eqnarray*}
	&& \norm{\phi^{1/2} \Lambda_\delta^{-2} \left(\partial_t\phi^{\ell}_{m}\right)W^\ell_{m}f_{m}}_{L^2\inner{[0,T]\times\mathbb R_+^2}}+ \norm{  \phi^{1/2} [u \partial_x+v\partial_y-\partial_y^2,~  \Lambda_\delta^{-2} \phi^{\ell}_{m}W^\ell_m] f_m }_{L^2\inner{[0,T]\times\mathbb R_+^2}} \\
&\leq &	m\, C \norm{\phi^{-1/2} F}_{L^2\inner{[0,T]\times\mathbb R_+^2}}+ C \norm{  \partial_y F }_{L^2\inner{[0,T]\times\mathbb R_+^2}}
\end{eqnarray*}	
and
\begin{eqnarray*}
	&&\norm{\Lambda^{-\frac{2}{3}}\partial_y \phi^{1/2} \Lambda_\delta^{-2} \left(\partial_t\phi^{\ell-1}_{m}\right)W^{\ell-1}_{m}f_{m}}_{L^2\inner{[0,T]\times\mathbb R_+^2}}\\
	&&\qquad\qquad+ \norm{  \Lambda^{-\frac{2}{3}}\partial_y \phi^{1/2} [u \partial_x+v\partial_y-\partial_y^2,  ~\Lambda_\delta^{-2} \phi^{\ell-1}_{m}W^{\ell-1}_m] f_m }_{L^2\inner{[0,T]\times\mathbb R_+^2}}\\
	&\leq & C\norm{  \comii y^{-\sigma} \Lambda^{ \frac{1}{3}} \tilde f}_{L^2([0,T]\times\RR^2_+)}+C\norm{  \partial_y^2 \Lambda^{-\frac{2}{3}} \tilde f}_{L^2([0,T]\times\RR^2_+)}\\
		&&+ m C \norm{   \Lambda^{-\frac{2}{3}} \phi^{-1/2}  \partial_y   \phi^{\ell-1}_{m}W^{\ell-1}_m      f_m}_{L^2([0,T]\times\RR^2_+)}+C\norm{    \Lambda^{-\frac{2}{3}}  \phi^{\ell-1}_{m}W^{\ell-1}_m       f_m}_{L^2([0,T]\times\RR^2_+)}.
\end{eqnarray*}	
\end{proposition}

\begin{proof} It is sufficient to prove the second estimate in Proposition \ref{+prp5.1+}, 
	since the treatment of the first one is similar and easier and we omit it here for brevity.   Observe    
 $$
 \left|\partial_t\phi^{\ell-1}_{m}\right|\le 3m \phi^{\ell-2}_{m}\leq 3m \phi^{\ell-1}_{m}\phi^{-1},
 $$
 and thus 
 \begin{equation}
\label{15042101}
	\norm{\Lambda^{-\frac{2}{3}}\partial_y \phi^{1/2} \Lambda_\delta^{-2} \left(\partial_t\phi^{\ell-1}_{m}\right)W^{\ell-1}_{m}f_{m}}_{L^2\inner{[0,T]\times\mathbb R_+^2}} \leq 3m	\norm{\Lambda^{-\frac{2}{3}} \phi^{-1/2} \partial_y    \phi^{\ell-1}_{m} W^{\ell-1}_{m}f_{m}}_{L^2\inner{[0,T]\times\mathbb R_+^2}}  .
\end{equation}	
 We   write,  using \reff{15042210},
\begin{eqnarray*}
	&&\norm{  \Lambda^{-\frac{2}{3}}\partial_y \phi^{1/2} [u \partial_x+v\partial_y-\partial_y^2,  ~\Lambda_\delta^{-2} \phi^{\ell-1}_{m}W^{\ell-1}_m] f_m }_{L^2\inner{[0,T]\times\mathbb R_+^2}} \\
		&\leq & \norm{  \com{u\p+v\partial_y-\partial_y^2,~ \Lambda^{-\frac{2}{3}}\partial_y \phi^{1/2}\Lambda_\delta^{-2} \phi^{\ell-1}_{m}W^{\ell-1}_m}    f_m }_{L^2([0,T]\times\RR^2_+)}\\
	&&+\norm{  \com{u\p+v\partial_y-\partial_y^2,~ \Lambda^{-\frac{2}{3}}\partial_y \phi^{1/2}}    \Lambda_\delta^{-2} \phi^{\ell-1}_{m}W^{\ell-1}_m  f_m }_{L^2([0,T]\times\RR^2_+)}\\
	&\stackrel{\rm def}{=}&   Q_{5,1}+ Q_{5,2}.
\end{eqnarray*}
We first estimate $Q_{5,1}.$  Observe   
\begin{eqnarray*}
	&& \norm{  \com{u\p,~ \Lambda^{-\frac{2}{3}}\partial_y \phi^{1/2}\Lambda_\delta^{-2} \phi^{\ell-1}_{m}W^{\ell-1}_m}    f_m }_{L^2([0,T]\times\RR^2_+)}\\
		&\leq&   \norm{   \com{ u\p,~ \Lambda^{-\frac{2}{3}} \Lambda_\delta^{-2} \phi^{\ell-1}_{m}W^{\ell-1}_m \partial_y  } \phi^{1/2}  f_m}_{L^2([0,T]\times\RR^2_+)}\\
		&&+\norm{  \Big[u\p,~ \Lambda^{-\frac{2}{3}} \Lambda_\delta^{-2} \phi^{\ell-1}_{m} \com{\partial_y,~ W^{\ell-1}_m}~\Big]  \phi^{1/2}  f_m}_{L^2([0,T]\times\RR^2_+)}.
	\end{eqnarray*}
	On the other hand,  we compute,  using  Lemma  \ref{lemma3.1}  and \reff{comfm},
	\begin{eqnarray*}
		&&  \norm{   \com{ u\p,~ \Lambda^{-\frac{2}{3}} \Lambda_\delta^{-2} \phi^{\ell-1}_{m}W^{\ell-1}_m \partial_y  } \phi^{1/2}  f_m}_{L^2([0,T]\times\RR^2_+)}\\
		&\leq &  \norm{   \com{ u\p,~ \Lambda^{-\frac{2}{3}} \Lambda_\delta^{-2} \phi^{\ell-1}_{m}W^{\ell-1}_m   } \partial_y \phi^{1/2}  f_m}_{L^2([0,T]\times\RR^2_+)}\\
		&&+ \norm{  \Lambda^{-\frac{2}{3}} \Lambda_\delta^{-2} \phi^{\ell-1}_{m}W^{\ell-1}_m  \com{ u\p,~ \partial_y  } \phi^{1/2}  f_m}_{L^2([0,T]\times\RR^2_+)}\\
		&\leq & C \norm{    \Lambda^{-\frac{2}{3}} \Lambda_\delta^{-2} \phi^{\ell-1}_{m}W^{\ell-1}_m   \partial_y  \phi^{1/2}  f_m}_{L^2([0,T]\times\RR^2_+)}+ \norm{   \Lambda^{-\frac{2}{3}} \Lambda_\delta^{-2} \phi^{\ell-1}_{m}W^{\ell-1}_m  (\partial_y u)\partial_x  \phi^{1/2} f_m}_{L^2([0,T]\times\RR^2_+)}\\
		&\leq & C \norm{    \partial_y  \Lambda^{-\frac{2}{3}} \Lambda_\delta^{-2} \phi^{\ell-1}_{m}W^{\ell-1}_m    \phi^{1/2}  f_m}_{L^2([0,T]\times\RR^2_+)}+C\norm{    \Lambda^{-\frac{2}{3}} \Lambda_\delta^{-2} \phi^{\ell-1}_{m}W^{\ell-1}_m     \phi^{1/2}  f_m}_{L^2([0,T]\times\RR^2_+)}\\
		&&+ \norm{  (\partial_y u) \Lambda^{-\frac{2}{3}} \Lambda_\delta^{-2} \phi^{\ell-1}_{m}W^{\ell-1}_m  \partial_x  \phi^{1/2} f_m}_{L^2([0,T]\times\RR^2_+)}\\
		&&+ \norm{ \com{  \Lambda^{-\frac{2}{3}} \Lambda_\delta^{-2} \phi^{\ell-1}_{m}W^{\ell-1}_m,~  (\partial_y u)}\partial_x  \phi^{1/2} f_m}_{L^2([0,T]\times\RR^2_+)}\\
				&\leq & C \norm{    \partial_y  \Lambda^{-\frac{2}{3}}  \phi^{\ell-1}_{m}W^{\ell-1}_m      f_m}_{L^2([0,T]\times\RR^2_+)}+C\norm{    \Lambda^{-\frac{2}{3}}  \phi^{\ell-1}_{m}W^{\ell-1}_m       f_m}_{L^2([0,T]\times\RR^2_+)}\\
		&&+ C\norm{  \comii y^{-\sigma} \Lambda^{ \frac{1}{3}} \phi^{1/2}\Lambda_\delta^{-2} \phi^{\ell-1}_{m}W^{\ell-1}_m  f_m}_{L^2([0,T]\times\RR^2_+)}.
			\end{eqnarray*}
Similarly we also have, using again Lemma \ref{lemma3.1}, 
	\begin{eqnarray*}
	&&\norm{  \Big[u\p,~ \Lambda^{-\frac{2}{3}} \Lambda_\delta^{-2} \phi^{\ell-1}_{m} \com{\partial_y,~ W^{\ell-1}_m}~\Big]  \phi^{1/2}  f_m}_{L^2([0,T]\times\RR^2_+)}\\
	 &\leq& C\norm{    \Lambda^{-\frac{2}{3}}  \phi^{\ell-1}_{m}W^{\ell-1}_m       f_m}_{L^2([0,T]\times\RR^2_+)}.
	\end{eqnarray*}
	As a result, combining these inequalities, we have
	\begin{eqnarray*}
		&&\norm{  \com{u\p,~ \Lambda^{-\frac{2}{3}}\partial_y \phi^{1/2}\Lambda_\delta^{-2} \phi^{\ell-1}_{m}W^{\ell-1}_m}    f_m }_{L^2([0,T]\times\RR^2_+)}\\
	&\leq& C \norm{    \partial_y  \Lambda^{-\frac{2}{3}}   \phi^{\ell-1}_{m}W^{\ell-1}_m      f_m}_{L^2([0,T]\times\RR^2_+)}+C\norm{    \Lambda^{-\frac{2}{3}}  \phi^{\ell-1}_{m}W^{\ell-1}_m       f_m}_{L^2([0,T]\times\RR^2_+)}\\
		&&+ C\norm{  \comii y^{-\sigma} \Lambda^{ \frac{1}{3}} \phi^{1/2}\Lambda_\delta^{-2} \phi^{\ell-1}_{m}W^{\ell-1}_m  f_m}_{L^2([0,T]\times\RR^2_+)}. 
	\end{eqnarray*}
	Similarly,  repeating the above arguments with $u\partial_x$ replaced by $v\partial_y$ and $\partial_y^2$ respectively,  one has
	\begin{eqnarray*}
		&&\norm{  \com{v\partial_y,~ \Lambda^{-\frac{2}{3}}\partial_y \phi^{1/2}\Lambda_\delta^{-2} \phi^{\ell-1}_{m}W^{\ell-1}_m}    f_m }_{L^2([0,T]\times\RR^2_+)}\\
	&\leq& C \norm{  \partial_y^2 \Lambda^{-\frac{2}{3}} \phi^{1/2} \Lambda_\delta^{-2} \phi^{\ell-1}_{m}W^{\ell-1}_m    f_m}_{L^2([0,T]\times\RR^2_+)}+C \norm{  \Lambda^{-\frac{2}{3}}  \partial_y   \phi^{\ell-1}_{m}W^{\ell-1}_m    f_m}_{L^2([0,T]\times\RR^2_+)}\\
		&&+C\norm{   \Lambda^{-\frac{2}{3}}  \phi^{\ell-1}_{m}W^{\ell-1}_m   f_m}_{L^2([0,T]\times\RR^2_+)},
	\end{eqnarray*}
 and
	\begin{eqnarray*}
		&&\norm{  \com{ \partial_y^2,~ \Lambda^{-\frac{2}{3}}\partial_y \phi^{1/2}\Lambda_\delta^{-2} \phi^{\ell-1}_{m}W^{\ell-1}_m}    f_m }_{L^2([0,T]\times\RR^2_+)}\\
	&\leq& C \norm{  \partial_y^2 \Lambda^{-\frac{2}{3}} \phi^{1/2} \Lambda_\delta^{-2} \phi^{\ell-1}_{m}W^{\ell-1}_m    f_m}_{L^2([0,T]\times\RR^2_+)}+C \norm{  \Lambda^{-\frac{2}{3}}  \partial_y  \phi^{\ell-1}_{m}W^{\ell-1}_m    f_m}_{L^2([0,T]\times\RR^2_+)}\\
		&&+C\norm{   \Lambda^{-\frac{2}{3}}  \phi^{\ell-1}_{m}W^{\ell-1}_m   f_m}_{L^2([0,T]\times\RR^2_+)}.
	\end{eqnarray*}
	As a result, we conclude, combining these inequalities, 
	\begin{eqnarray*}
		Q_{5,1} &=&\norm{  \com{u\p+v\partial_y-\partial_y^2,~ \Lambda^{-\frac{2}{3}}\partial_y \phi^{1/2}\Lambda_\delta^{-2} \phi^{\ell-1}_{m}W^{\ell-1}_m}    f_m }_{L^2([0,T]\times\RR^2_+)}\\
	&\leq&   C\norm{  \comii y^{-\sigma} \Lambda^{ \frac{1}{3}} \phi^{1/2}\Lambda_\delta^{-2} \phi^{\ell-1}_{m}W^{\ell-1}_m  f_m}_{L^2([0,T]\times\RR^2_+)}+C\norm{  \partial_y^2 \Lambda^{-\frac{2}{3}} \phi^{1/2} \Lambda_\delta^{-2} \phi^{\ell-1}_{m}W^{\ell-1}_m    f_m}_{L^2([0,T]\times\RR^2_+)}\\
		&&+ C \norm{    \partial_y  \Lambda^{-\frac{2}{3}}   \phi^{\ell-1}_{m}W^{\ell-1}_m      f_m}_{L^2([0,T]\times\RR^2_+)}+C\norm{    \Lambda^{-\frac{2}{3}}   \phi^{\ell-1}_{m}W^{\ell-1}_m       f_m}_{L^2([0,T]\times\RR^2_+)}. 
	\end{eqnarray*}
	The term  $Q_{5,2}$ can be treated similarly and easily,   and we have
	\begin{eqnarray*}
		Q_{5,2} &=&\norm{  \com{u\p+v\partial_y-\partial_y^2,~ \Lambda^{-\frac{2}{3}}\partial_y \phi^{1/2}}    \Lambda_\delta^{-2} \phi^{\ell-1}_{m}W^{\ell-1}_m  f_m }_{L^2([0,T]\times\RR^2_+)}\\
		&\leq&  C\norm{  \comii y^{-\sigma} \Lambda^{ \frac{1}{3}} \phi^{1/2}\Lambda_\delta^{-2} \phi^{\ell-1}_{m}W^{\ell-1}_m  f_m}_{L^2([0,T]\times\RR^2_+)}+C\norm{  \partial_y^2 \Lambda^{-\frac{2}{3}} \phi^{1/2} \Lambda_\delta^{-2} \phi^{\ell-1}_{m}W^{\ell-1}_m    f_m}_{L^2([0,T]\times\RR^2_+)}\\
		&&+ C \norm{    \partial_y  \Lambda^{-\frac{2}{3}}   \phi^{\ell-1}_{m}W^{\ell-1}_m      f_m}_{L^2([0,T]\times\RR^2_+)}+C\norm{    \Lambda^{-\frac{2}{3}}  \phi^{\ell-1}_{m}W^{\ell-1}_m       f_m}_{L^2([0,T]\times\RR^2_+)}. 
	\end{eqnarray*}
	Thus
	\begin{eqnarray*}
		&&\norm{  \Lambda^{-\frac{2}{3}}\partial_y \phi^{1/2} [u \partial_x+v\partial_y-\partial_y^2,  ~\Lambda_\delta^{-2} \phi^{\ell-1}_{m}W^{\ell-1}_m] f_m }_{L^2\inner{[0,T]\times\mathbb R_+^2}} \\ 
		&\leq&   C\norm{  \comii y^{-\sigma} \Lambda^{ \frac{1}{3}} \phi^{1/2}\Lambda_\delta^{-2} \phi^{\ell-1}_{m}W^{\ell-1}_m  f_m}_{L^2([0,T]\times\RR^2_+)}+C\norm{  \partial_y^2 \Lambda^{-\frac{2}{3}} \phi^{1/2} \Lambda_\delta^{-2} \phi^{\ell-1}_{m}W^{\ell-1}_m    f_m}_{L^2([0,T]\times\RR^2_+)}\\
		&&+ C \norm{    \partial_y  \Lambda^{-\frac{2}{3}}  \phi^{\ell-1}_{m}W^{\ell-1}_m      f_m}_{L^2([0,T]\times\RR^2_+)}+C\norm{    \Lambda^{-\frac{2}{3}}  \phi^{\ell-1}_{m}W^{\ell-1}_m       f_m}_{L^2([0,T]\times\RR^2_+)}. 
			\end{eqnarray*}
  This along with \reff{15042101} gives the second estimate in Proposition \ref{+prp5.1+}.   The  proof  is thus complete. 
\end{proof}

\begin{proposition}\label{prop5.1}
Under the induction hypothesis \eqref{2.6}, \eqref{2.6+}, we have,  denoting  $F=\Lambda_\delta^{-2} \phi^{\ell}_{m} W^\ell_{m} f_{m},$
\begin{eqnarray*}
\norm{  \phi^{1/2}\mathcal J_{m,\ell, \delta}}_{L^2\inner{[0,T]\times\mathbb R_+^2}}
 \leq  m C   \norm{  F}_{L^2\inner{[0,T]\times\mathbb R_+^2}}+ C  A^{m-6}  \inner{(m-5)!}^{3(1+\sigma)},
  \end{eqnarray*}
  and
  \begin{eqnarray*}
  	&& \norm{ \Lambda^{-2/3} \partial_y \phi^{1/2} \mathcal J_{m,\ell-1, \delta}}_{L^2\inner{[0,T]\times\mathbb R_+^2}}\\
 &\leq & m C  \inner{ \norm{  \Lambda^{-2/3} \Lambda_\delta^{-2} \phi^{\ell-1}_{m} W^{\ell-1}_{m} f_{m}}_{L^2\inner{[0,T]\times\mathbb R_+^2}}+\norm{  \Lambda^{-2/3}\partial_y \Lambda_\delta^{-2} \phi^{\ell-1}_{m} W^{\ell-1}_{m} f_{m}}_{L^2\inner{[0,T]\times\mathbb R_+^2}}}\\
 &&+ C  A^{m-6}  \inner{(m-5)!}^{3(1+\sigma)},
  \end{eqnarray*}
  where the constant $C>0$ is independent on $m$ and $\delta>0$. 
\end{proposition}

We first prove the first estimate in  Proposition \ref{prop5.1}.   In view of the definition given at the beginning of this section,  we see,
\begin{eqnarray}\label{adxzml}
\begin{split}
	& \norm{    \phi^{1/2}\mathcal J_{m,\ell,\delta}}_{L^2\inner{[0,T]\times\mathbb R_+^2}}\leq \norm{    \mathcal J_{m,\ell,\delta}}_{L^2\inner{[0,T]\times\mathbb R_+^2}}\\
	\leq & \sum_{j=1}^{m}{m\choose j}\norm{  \Lambda_\delta^{-2}   \phi^{\ell}_mW^\ell_m(\p^j u)f_{m+1-j}}_{L^2\inner{[0,T]\times\mathbb R_+^2}}\\
	&+\sum_{j=1}^{m-1}{m\choose j}\norm{   \Lambda_\delta^{-2}  \phi^{\ell}_mW^\ell_m(\p^j v)\partial_y f_{m-j}}_{L^2\inner{[0,T]\times\mathbb R_+^2}}\\
&+\sum_{j=1}^{m-1}{m\choose j}
\norm{  \Lambda_\delta^{-2}   \phi^{\ell}_mW^\ell_m\left[\partial_y\inner{ \partial_y\omega/\omega}\right](\p^j v)(\p^{m-j}u)}_{L^2\inner{[0,T]\times\mathbb R_+^2}}\\
&+2\norm{  \Lambda_\delta^{-2}  \phi^{\ell}_m
W^\ell_m\left[\partial_y\inner{ \partial_y\omega/\omega}\right]f_m}_{L^2\inner{[0,T]\times\mathbb R_+^2}}.
\end{split}
\end{eqnarray}
And we  will proceed to estimate the each term on the right hand side of \eqref{adxzml},  and state as the following three Lemmas.

\begin{lemma}\label{lem1}
Under the same assumption as in Proposition \ref{nonli},  we have
\begin{eqnarray*}
&&\sum_{j=1}^{m-1}{m\choose j}\norm{  \Lambda_\delta^{-2} \phi^{\ell}_m W^\ell_m(\p^j v)\partial_y f_{m-j}}_{L^2\inner{[0,T]\times\mathbb R_+^2}}\\
&\leq & mC   \norm{  \Lambda_\delta^{-2} \phi^{\ell}_m W^\ell_m f_m}_{L^2\inner{[0,T]\times\mathbb R_+^2}}+ C  A^{m-6}  \inner{(m-5)!}^{3(1+\sigma)}.
\end{eqnarray*}
\end{lemma}

\begin{proof} 
We first split the summation as follows:
\begin{eqnarray*}
&&\sum_{j=1}^{m-1}{m\choose j}\norm{  \Lambda_\delta^{-2} \phi^{\ell}_m W^\ell_m(\p^j v)\partial_y f_{m-j}}_{L^2\inner{[0,T]\times\mathbb R_+^2}}\\
&= &   \sum_{j=m-2}^{m-1}{m\choose j}\norm{  \Lambda_\delta^{-2} \phi^{\ell}_m W^\ell_m(\p^j v)\partial_y f_{m-j}}_{L^2\inner{[0,T]\times\mathbb R_+^2}}\\
&&+ \sum_{j=1}^{m-3}{m\choose j}\norm{  \Lambda_\delta^{-2} \phi^{\ell}_m W^\ell_m(\p^j v)\partial_y f_{m-j}}_{L^2\inner{[0,T]\times\mathbb R_+^2}}.
\end{eqnarray*}
Moreover as for the  last term on the right hand side,  we use \reff{+wmi} to compute,   \begin{eqnarray*}
 &&\norm{   \Lambda_\delta^{-2} \phi^{\ell}_m W^\ell_m(\p^{j} v)\partial_y f_{m-j}}_{L^2\inner{[0,T]\times\mathbb R_+^2}}\\
 &\leq & \norm{ \phi^{\ell}_m W^0_{m} \Lambda^{\ell /3}(\p^j v)\partial_y f_{m-j}}_{L^2\inner{[0,T]\times\mathbb R_+^2}}\\
 &\leq &\norm{ \phi^{\ell}_m W^0_{m}  (\p^j v)\partial_y f_{m-j}}_{L^2\inner{[0,T]\times\mathbb R_+^2}}+\norm{ \phi^{\ell}_m W^0_{m} \partial_x (\p^j v)\partial_y f_{m-j}}_{L^2\inner{[0,T]\times\mathbb R_+^2}}\\
 &\leq &\norm{ \phi^{\ell}_m W^0_{m}  (\p^j v)\partial_y f_{m-j}}_{L^2\inner{[0,T]\times\mathbb R_+^2}}+\norm{ \phi^{\ell}_m W^0_{m}  (\p^{j+1} v)\partial_y f_{m-j}}_{L^2\inner{[0,T]\times\mathbb R_+^2}}\\
 &&+\norm{ \phi^{\ell}_m W^0_{m}   (\p^j v)(\partial_y \partial_x f_{m-j})}_{L^2\inner{[0,T]\times\mathbb R_+^2}}.
 \end{eqnarray*}
 Thus we have
\begin{eqnarray}\label{km1}
 \begin{split}
&\sum_{j=1}^{m-1}{m\choose j}\norm{   \Lambda_\delta^{-2} \phi^{\ell}_m W^\ell_m(\p^j v)\partial_y f_{m-j}}_{L^2\inner{[0,T]\times\mathbb R_+^2}}\\
\leq&    \sum_{j=m-2}^{m-1}{m\choose j}\norm{  \Lambda_\delta^{-2} \phi^{\ell}_m W^\ell_m(\p^j v)\partial_y f_{m-j}}_{L^2\inner{[0,T]\times\mathbb R_+^2}}\\
&~+\sum_{j=1}^{m-3}{m\choose j}\norm{ \phi^{\ell}_m W^0_{m}  (\p^j v)\partial_y f_{m-j}}_{L^2\inner{[0,T]\times\mathbb R_+^2}}\\
&~+\sum_{j=1}^{m-3}{m\choose j}\norm{ \phi^{\ell}_m W^0_{m}  (\p^{j+1} v)\partial_y f_{m-j}}_{L^2\inner{[0,T]\times\mathbb R_+^2}}\\
&~+\sum_{j=1}^{m-3}{m\choose j}\norm{ \phi^{\ell}_m W^0_{m}   (\p^j v)(\partial_y \partial_x f_{m-j})}_{L^2\inner{[0,T]\times\mathbb R_+^2}}.
\end{split}
\end{eqnarray}
Next we estimate step by step the terms on the right side of \reff{km1}. 

\medskip
{\it  (a)}~ We treat in this step the first  term  on the right hand side of \reff{km1},  and prove that
\begin{eqnarray}\label{1m1}
	  &&\sum_{j=m-2}^{m-1}{m\choose j}\norm{  \Lambda_\delta^{-2} \phi^{\ell}_m W^\ell_m(\p^j v)\partial_y f_{m-j}}_{L^2\inner{[0,T]\times\mathbb R_+^2}} \nonumber\\
	  &\leq&  m\, C \norm{   \Lambda_\delta^{-2} \phi^{\ell}_m  W^\ell_{m}f_{m}}_{L^{2}\inner{[0,T]\times\mathbb R_+^2}}+C A^{m-6}\inner{(m-5)!}^{3(1+\sigma)}.
\end{eqnarray}
To do so,  direct computation 
gives
\begin{eqnarray*}
&&\sum_{j=m-2}^{m-1}{m\choose j} \norm{   \Lambda_\delta^{-2} \phi^{\ell}_m W^\ell_m(\p^{j} v)\partial_y f_{m-j}}_{L^2\inner{[0,T]\times\mathbb R_+^2}}\\
&\leq&\sum_{j=m-2}^{m-1}{m\choose j} \norm{   \Lambda_\delta^{-2} \Lambda^{\ell /3} e^{2cy}\inner{1+cy}^{-1}\phi^{\ell}_m (\p^{j} v)\partial_y f_{m-j}}_{L^2\inner{[0,T]\times\mathbb R_+^2}}\\
&\leq&\sum_{j=m-2}^{m-1}{m\choose j}  \norm{ e^{2cy} (\partial_y f_{m-j} )   \Lambda_\delta^{-2} \Lambda^{\ell /3} \inner{1+cy}^{-1} \phi^{\ell}_m\p^{j} v}_{L^2\inner{[0,T]\times\mathbb R_+^2}}\\
&&+  \sum_{j=m-2}^{m-1}{m\choose j}\norm{   \com{e^{2cy}(\partial_y f_{m-j}),~    \Lambda_\delta^{-2} \Lambda^{\ell/3}} \inner{1+cy}^{-1}\phi^{\ell}_m \p^{j} v}_{L^2\inner{[0,T]\times\mathbb R_+^2}}. 
\end{eqnarray*}
On the other hand,   by \reff{15042225},  
\begin{eqnarray*}
	&&\sum_{j=m-2}^{m-1}{m\choose j} \norm{ e^{2cy} (\partial_y f_{m-j} )   \Lambda_\delta^{-2} \Lambda^{\ell /3} \inner{1+cy}^{-1} \phi^{\ell}_m\p^{j} v}_{L^2\inner{[0,T]\times\mathbb R_+^2}}\\
	&\leq & C \sum_{j=m-2}^{m-1}{m\choose j}  \norm{(1+cy)^{-1}}_{L^2\inner{ \mathbb R_+;~   L^\infty \inner{[0,T]\times\mathbb R_x}}}\norm{   \Lambda_\delta^{-2} \phi^{\ell}_m  \Lambda^{\ell/3}\p^{j} v}_{L^{\infty}\inner{\mathbb R_+; ~L^2\inner{[0,T]\times\mathbb R_x}}}\\
	&\leq &C  \sum_{j=m-2}^{m-1}{m\choose j} \norm{   \Lambda_\delta^{-2} \phi^{\ell}_{m}  \Lambda^{\ell/3}\p^{j} v}_{L^{\infty}\inner{\mathbb R_+; ~L^2\inner{[0,T]\times\mathbb R_x}}}.\end{eqnarray*}
Similarly,  we have,   by virtue of  Lemma \ref{lemma3.1}, 
\begin{eqnarray*}
	 &&\sum_{j=m-2}^{m-1}{m\choose j} \norm{   \com{e^{2cy}(\partial_y f_{m-j}),~    \Lambda_\delta^{-2} \Lambda^{\ell/3}} \inner{1+cy}^{-1}\phi^{\ell}_m \p^{j} v}_{L^2\inner{[0,T]\times\mathbb R_+^2}} \\
	 &\leq & C \sum_{j=m-2}^{m-1}{m\choose j} \norm{    \inner{1+cy}^{-1}\phi^{\ell}_{m} \p^{j} v}_{L^2\inner{[0,T]\times\mathbb R_+^2}}  \\
&\leq & C  \sum_{j=m-2}^{m-1}{m\choose j} \norm{   \phi^{\ell}_{m} \Lambda^{\ell/3}\p^{j} v}_{L^{\infty}\inner{\mathbb R_+; ~L^2\inner{[0,T]\times\mathbb R_x}}}.
\end{eqnarray*}
Thus combining these inequalities, we obtain
\begin{eqnarray*}
&&\sum_{j=m-2}^{m-1}{m\choose j}  \norm{   \Lambda_\delta^{-2} \phi^{\ell}_m W^\ell_m(\p^{j} v)\partial_y f_{m-j}}_{L^2\inner{[0,T]\times\mathbb R_+^2}}\\
&\leq& 
C  \sum_{j=m-2}^{m-1}{m\choose j} \norm{   \Lambda_\delta^{-2} \phi^{\ell}_{m}  \Lambda^{\ell/3}\p^{j} v}_{L^{\infty}\inner{\mathbb R_+; ~L^2\inner{[0,T]\times\mathbb R_x}}}\\
&\leq& 
C  m \norm{   \Lambda_\delta^{-2} \phi^{\ell}_{m}  \Lambda^{\ell/3}\p^{m-1} v}_{L^{\infty}\inner{\mathbb R_+; ~L^2\inner{[0,T]\times\mathbb R_x}}}+C  m^2 \norm{   \Lambda_\delta^{-2} \phi^{\ell}_{m}  \Lambda^{\ell/3}\p^{m-2} v}_{L^{\infty}\inner{\mathbb R_+; ~L^2\inner{[0,T]\times\mathbb R_x}}}.
\end{eqnarray*}
Moreover,  observe 
\begin{eqnarray*}
	&&\norm{   \Lambda_\delta^{-2} \phi^{\ell}_{m}  \Lambda^{\ell/3}\p^{m-2} v}_{L^{\infty}\inner{\mathbb R_+; ~L^2\inner{[0,T]\times\mathbb R_x}}}^2 \\
	&\leq & \norm{   \Lambda_\delta^{-2} \phi^{\ell}_{m}  \Lambda^{\ell/3}\p^{m-1} v}_{L^{\infty}\inner{\mathbb R_+; ~L^2\inner{[0,T]\times\mathbb R_x}}}\norm{   \Lambda_\delta^{-2} \phi^{\ell}_{m}  \Lambda^{\ell/3}\p^{m-3} v}_{L^{\infty}\inner{\mathbb R_+; ~L^2\inner{[0,T]\times\mathbb R_x}}},
\end{eqnarray*}
and thus
\begin{eqnarray*}
	&&m^2 \norm{   \Lambda_\delta^{-2} \phi^{\ell}_{m}  \Lambda^{\ell/3}\p^{m-2} v}_{L^{\infty}\inner{\mathbb R_+; ~L^2\inner{[0,T]\times\mathbb R_x}}}\\
	&\leq &m \norm{   \Lambda_\delta^{-2} \phi^{\ell}_{m}  \Lambda^{\ell/3}\p^{m-1} v}_{L^{\infty}\inner{\mathbb R_+; ~L^2\inner{[0,T]\times\mathbb R_x}}}+m^3\norm{   \Lambda_\delta^{-2} \phi^{\ell}_{m}  \Lambda^{\ell/3}\p^{m-3} v}_{L^{\infty}\inner{\mathbb R_+; ~L^2\inner{[0,T]\times\mathbb R_x}}}\\
	&\leq & m \norm{   \Lambda_\delta^{-2} \phi^{\ell}_{m}  \Lambda^{\ell/3}\p^{m-1} v}_{L^{\infty}\inner{\mathbb R_+; ~L^2\inner{[0,T]\times\mathbb R_x}}}+m^3\norm{   \Lambda_\delta^{-2} \phi^{\ell}_{m}  \p^{m-3} v}_{L^{\infty}\inner{\mathbb R_+; ~L^2\inner{[0,T]\times\mathbb R_x}}}\\
	&&+m^3\norm{   \Lambda_\delta^{-2} \phi^{\ell}_{m}  \p^{m-2} v}_{L^{\infty}\inner{\mathbb R_+; ~L^2\inner{[0,T]\times\mathbb R_x}}}\\
	&\leq & m \norm{   \Lambda_\delta^{-2} \phi^{\ell}_{m}  \Lambda^{\ell/3}\p^{m-1} v}_{L^{\infty}\inner{\mathbb R_+; ~L^2\inner{[0,T]\times\mathbb R_x}}}+m^3\norm{    \phi^{0}_{m-2}  \p^{m-3} v}_{L^{\infty}\inner{\mathbb R_+; ~L^2\inner{[0,T]\times\mathbb R_x}}}\\
	&&+m^3\norm{     \phi^{0}_{m-1}  \p^{m-2} v}_{L^{\infty}\inner{\mathbb R_+; ~L^2\inner{[0,T]\times\mathbb R_x}}}.
\end{eqnarray*}
Then we have, combining the above inequalities, 
\begin{eqnarray*}
	&&\sum_{j=m-2}^{m-1}{m\choose j} \norm{ e^{2cy} (\partial_y f_{m-j} )   \Lambda_\delta^{-2} \Lambda^{\ell /3} \inner{1+cy}^{-1} \phi^{\ell}_m\p^{j} v}_{L^2\inner{[0,T]\times\mathbb R_+^2}}\\
	&\leq& 
C  m \norm{   \Lambda_\delta^{-2} \phi^{\ell}_{m}  \Lambda^{\ell/3}\p^{m-1} v}_{L^{\infty}\inner{\mathbb R_+; ~L^2\inner{[0,T]\times\mathbb R_x}}}+C m^3\norm{     \phi^{0}_{m-2}  \p^{m-3} v}_{L^{\infty}\inner{\mathbb R_+; ~L^2\inner{[0,T]\times\mathbb R_x}}}\\
	&&+Cm^3\norm{   \phi^{0}_{m-1}  \p^{m-2} v}_{L^{\infty}\inner{\mathbb R_+; ~L^2\inner{[0,T]\times\mathbb R_x}}}\\
	&\leq& 
C  m \norm{   \Lambda_\delta^{-2} \phi^{\ell}_{m}  W^{\ell}_m f_m }_{L^{\infty}\inner{\mathbb R_+; ~L^2\inner{[0,T]\times\mathbb R_x}}}+C m^3\norm{   \phi^{0}_{m-2}W^{0}_{m-2}  f_{m-2} }_{L^{\infty}\inner{\mathbb R_+; ~L^2\inner{[0,T]\times\mathbb R_x}}}\\
	&&+Cm^3\norm{   \phi^{0}_{m-1} W^{0}_{m-1} f_{m-1}}_{L^{\infty}\inner{\mathbb R_+; ~L^2\inner{[0,T]\times\mathbb R_x}}},
\end{eqnarray*}
the last inequality following from \reff{equ2}. This,  along with the estimate
\begin{eqnarray*}
	&&m^3\norm{   \phi^{0}_{m-2}W^{0}_{m-2}  f_{m-2} }_{L^{\infty}\inner{\mathbb R_+; ~L^2\inner{[0,T]\times\mathbb R_x}}}+ m^3\norm{   \phi^{0}_{m-1} W^{0}_{m-1} f_{m-1}}_{L^{\infty}\inner{\mathbb R_+; ~L^2\inner{[0,T]\times\mathbb R_x}}}\\
	&\leq& C A^{m-6}\inner{(m-5)!}^{3(1+\sigma)} 
\end{eqnarray*}
due to the inductive  assumption \reff{2.6}, 
gives the desired estimate \reff{1m1}.

\medskip
{\it  (b)}
We will estimate in this step  the second and the third terms on the right hand side of \reff{km1}, and prove  that
\begin{eqnarray}\label{mj1}
\begin{split}
&\sum_{j=1}^{m-3}{m\choose j}\norm{ \phi^{\ell}_m W^0_{m}(\p^{j} v)\partial_y f_{m-j}}_{L^2\inner{[0,T]\times\mathbb R_+^2}}\\
&+\sum_{j=1}^{m-3}{m\choose j}\norm{ \phi^{\ell}_m W^0_{m}(\p^{j+1} v)\partial_y f_{m-j}}_{L^2\inner{[0,T]\times\mathbb R_+^2}}\\
\leq &~C A^{m-6  }\inner{(m-5)! }^{3(1+\sigma)}.
\end{split}
\end{eqnarray}
For this purpose we write,  denoting by $[m/2]$ the largest integer less than or equal to $m/2$, 
\begin{eqnarray}\label{s1s2}
\begin{split}
& \sum_{j=1}^{m-3}{m\choose j}\norm{ \phi^{\ell}_m W^0_{m}(\p^{j+1} v)\partial_y f_{m-j}}_{L^2\inner{[0,T]\times\mathbb R_+^2}}\\
\leq &~\sum_{j=1}^{\com{m/2}}{m\choose j}\norm{ \phi^{\ell}_m W^0_{m}(\p^{j+1} v)(\partial_y   f_{m-j})}_{L^2\inner{[0,T]\times\mathbb R_+^2}}\\
&~+\sum_{j=\com{m/2}+1}^{m-3}{m\choose j}\norm{ \phi^{\ell}_m W^0_{m}(\p^{j+1} v)\partial_y f_{m-j}}_{L^2\inner{[0,T]\times\mathbb R_+^2}}\\
=&~S_1+S_2.
\end{split}
\end{eqnarray}
We first treat $S_1$.  Using the inequality
\begin{eqnarray*}
\phi^{\ell}_m\leq \phi^{0}_m \leq \phi^{0}_{j+3}\phi^{0}_{m-j}, \quad  W^0_{m}\leq W^0_{m-j}~{\rm for}~j\geq 1,
\end{eqnarray*}
gives
\begin{eqnarray}\label{s1sum}
\begin{split}
S_1&=\sum_{j=1}^{\com{m/2}}{m\choose j}\norm{\phi^{\ell}_m W^0_{m}(\p^{j+1} v)\partial_y f_{m-j}}_{L^2\inner{[0,T]\times\mathbb R_+^2}}\\
&\leq  \sum_{j=1}^{\com{m/2}}{m\choose j}\norm{\phi^{0}_{j+3} \p^{j+1} v }_{L^\infty\inner{[0,T]\times\mathbb R_+^2}}\norm{\phi^{0}_{m-j}W^0_{m-j} \partial_y f_{m-j}}_{L^2\inner{[0,T]\times\mathbb R_+^2}}.
\end{split}
\end{eqnarray} 
By Sobolev inequality, we have 
\begin{eqnarray*}
&&\norm{\phi^{0}_{j+3} \p^{j+1} v }_{L^\infty\inner{[0,T]\times\mathbb R_+^2}}\\
&\leq &C \norm{\phi^{0}_{j+3} \p^{j+1} v }_{L^\infty\inner{[0,T]\times\mathbb R_+;~L^2(\mathbb R_x)}}+C \norm{\phi^{0}_{j+3} \p^{j+2} v }_{L^\infty\inner{[0,T]\times\mathbb R_+;~L^2(\mathbb R_x)}}\\
&\leq &C \norm{\phi^{0}_{j+2} \p^{j+1} v }_{L^\infty\inner{[0,T]\times\mathbb R_+;~L^2(\mathbb R_x)}}+C\norm{\phi^{0}_{j+3} \p^{j+2} v }_{L^\infty\inner{[0,T]\times\mathbb R_+;~L^2(\mathbb R_x)}}\\
&\leq & C\norm{\phi^{0}_{j+2}W^0_{j+2} f_{j+2}   }_{L^\infty\inner{[0,T];~L^2(\mathbb R_+^2)}}+C\norm{\phi^{0}_{j+3}W^0_{j+3} f_{j+3}}_{L^\infty\inner{[0,T];~L^2(\mathbb R_+^2)}},
\end{eqnarray*}
the secomd inequality using \reff{phi} and the last inequlaity following from \reff{equ2}.    As a result,    we  use the hypothesis of induction  \eqref{2.6}  and the initial hypothesis of induction \reff{15042225}  to conclude that   if $4\leq j\leq [m/2]$ then
\begin{eqnarray*}
	\norm{\phi^{0}_{j+3} \p^{j+1} v }_{L^\infty\inner{[0,T]\times\mathbb R_+^2}} \leq C  \inner{A^{j-3  }\inner{(j-3)!}^{3(1+\sigma)}+A^{j-2  }\inner{(j-2)!}^{3(1+\sigma)} }\leq CA^{j-2  }\inner{(j-2)!}^{3(1+\sigma)},
\end{eqnarray*}
and  if $1\leq j\leq 3$
\begin{eqnarray*}
	\norm{\phi^{0}_{j+3} \p^{j+1} v }_{L^\infty\inner{[0,T]\times\mathbb R_+^2}} \leq C.
\end{eqnarray*}
Moreover,  using \reff{comfm} and also the inductive assumption \eqref{2.6},  we calculate,  for any  $1\leq j\leq [m/2],$ 
\begin{eqnarray*}
	&&\norm{\phi^{0}_{m-j}W^0_{m-j} \partial_y f_{m-j}}_{L^2\inner{[0,T]\times\mathbb R_+^2}}\\
	&\leq& \norm{ \partial_y \phi^{0}_{m-j}W^0_{m-j} f_{m-j}}_{L^2\inner{[0,T]\times\mathbb R_+^2}}+\norm{\phi^{0}_{m-j} \com{\partial_y,~W^0_{m-j} }f_{m-j}  }_{L^2\inner{[0,T]\times\mathbb R_+^2}} \\
	&\leq& \norm{ \partial_y \phi^{0}_{m-j}W^0_{m-j} f_{m-j}}_{L^2\inner{[0,T]\times\mathbb R_+^2}}+ C \norm{\phi^{0}_{m-j}  W^0_{m-j} f_{m-j}  }_{L^\infty \inner{[0,T],~L^2(\mathbb R_+^2)}} \\
	&\leq & CA^{m-j-5  }\inner{(m-j-5)! }^{3(1+\sigma)}. 
\end{eqnarray*}
Putting these inequalities into \reff{s1sum} gives
\begin{eqnarray}\label{s1}
\begin{split}
S_1 &\leq C \sum_{j=4}^{\com{m/2}} \frac{m!}{j!(m-j)!}   A^{j-2  }\inner{(j-2)!}^{3(1+\sigma)}  \bigg(A^{m-j-5  }\inner{(m-j-5)! }^{3(1+\sigma)}\bigg)\\
&\quad +  C \sum_{j=1}^{3} \frac{m!}{j!(m-j)!}    \bigg(A^{m-j-5  }\inner{(m-j-5)! }^{3(1+\sigma)}\bigg)\\
&\leq  C \sum_{j=4}^{\com{m/2}} \frac{m!}{j^2 (m-j)^5}   A^{m-7} \inner{(j-2)!}^{3(1+\sigma)-1}   \inner{(m-j-5)! }^{3(1+\sigma)-1
} \\
&\quad +  C     A^{m-6  }\inner{(m-5)! }^{3(1+\sigma)}\\
&\leq  C \sum_{j=4}^{\com{m/2}} \frac{m! }{j^2 (m/2)^5}   A^{m-7} \inner{(m-7)!}^{3(1+\sigma)-1}+  C     A^{m-6  }\inner{(m-5)! }^{3(1+\sigma)}\\
&\leq  C  (m-5)!    A^{m-7} \inner{(m-7)!}^{3(1+\sigma)-1} +  C     A^{m-6  }\inner{(m-5)! }^{3(1+\sigma)}\\
&\leq C    A^{m-6  }\inner{(m-5)! }^{3(1+\sigma)}. 
\end{split}
\end{eqnarray} 
We now treat $S_2$.  Using the inequality
\begin{eqnarray*}
\phi^{\ell}_m\leq \phi^{0}_m \leq \phi^{0}_{j+2}\phi^{0}_{m-j+1}, \quad  W^0_{m}\leq W^0_{m-j+1}~{\rm for}~j\geq 1,
\end{eqnarray*}
and thus
\begin{eqnarray}\label{s2+}
\begin{split}
S_2&=\sum_{j=\com{m/2}+1}^{m-3}{m\choose j}\norm{ \phi^{\ell}_m W^0_{m}(\p^{j+1} v)\partial_y f_{m-j}}_{L^2\inner{[0,T]\times\mathbb R_+^2}}\\
&\leq   \sum_{j=\com{{m\over2}}+1}^{m-3} {m\choose j}\norm{\phi^{0}_{j+2} \p^{j+1} v }_{L^\infty\inner{[0,T]\times\mathbb R_+;~L^2(\mathbb R_x)}}\\
&\qquad\qquad \qquad\qquad \times \norm{\phi^{0}_{m-j+1}W^0_{m-j+1} \partial_y f_{m-j}}_{L^2\inner{[0,T]\times\mathbb R_+;~L^\infty(\mathbb R_x)}}\\
&\leq   \sum_{j=\com{m/2}+1}^{m-3} {m\choose j}\norm{\phi^{0}_{j+2} W^0_{j+2}f_{j+2} }_{L^\infty\inner{[0,T];~L^2(\mathbb R_+^2)}}\\
&\qquad\qquad \qquad \qquad \times\norm{\phi^{0}_{m-j+1}W^0_{m-j+1} \partial_y f_{m-j}}_{L^2\inner{[0,T]\times\mathbb R_+;~L^\infty(\mathbb R_x)}},
\end{split}
\end{eqnarray}
the last inequality using \reff{equ2}.   As for the last factor in the above inequality, 
  we use Sobolev inequality,  \reff{phi} and \reff{wmi} to compute  
\begin{eqnarray*}
&&\norm{\phi^{0}_{m-j+1}W^0_{m-j+1} \partial_y f_{m-j}}_{L^2\inner{[0,T]\times\mathbb R_+;~L^\infty(\mathbb R_x)}}\\
&\leq &C\norm{\phi^{0}_{m-j+1}W^0_{m-j+1} \partial_y f_{m-j}}_{L^2\inner{[0,T]\times\mathbb R_+^2}}+C\norm{\phi^{0}_{m-j+1}W^0_{m-j+1} \partial_y \partial_x f_{m-j}}_{L^2\inner{[0,T]\times\mathbb R_+^2}}\\
&\leq &C \norm{\phi^{0}_{m-j}W^0_{m-j} \partial_y f_{m-j}}_{L^2\inner{[0,T]\times\mathbb R_+^2}}+C\norm{\phi^{0}_{m-j+1}W^0_{m-j+1} \partial_y \partial_x f_{m-j}}_{L^2\inner{[0,T]\times\mathbb R_+^2}}.
\end{eqnarray*}
On the other hand,   in view of the definition  of $f_m,$ we have
\begin{eqnarray*}
	 &&\norm{\phi^{0}_{m-j+1}W^0_{m-j+1} \partial_y \partial_x f_{m-j}}_{L^2\inner{[0,T]\times\mathbb R_+^2}}\\
	 &\leq &\norm{\phi^{0}_{m-j+1}W^0_{m-j+1} \partial_y   f_{m-j+1}}_{L^2\inner{[0,T]\times\mathbb R_+^2}}+\norm{\phi^{0}_{m-j+1}W^0_{m-j+1} \partial_y (\partial_x^{m-j }u)\partial_x\inner{(\partial_y\omega)/\omega}}_{L^2\inner{[0,T]\times\mathbb R_+^2}}\\
	 &\leq &\norm{\phi^{0}_{m-j+1}W^0_{m-j+1} \partial_y   f_{m-j+1}}_{L^2\inner{[0,T]\times\mathbb R_+^2}}+\norm{\phi^{0}_{m-j+1}W^0_{m-j+1}   (\partial_x^{m-j }\omega)\partial_x\inner{(\partial_y\omega)/\omega}}_{L^2\inner{[0,T]\times\mathbb R_+^2}}\\
	 &&+\norm{\phi^{0}_{m-j+1}W^0_{m-j+1}   (\partial_x^{m-j }u)\partial_x\partial_y \inner{(\partial_y\omega)/\omega}}_{L^2\inner{[0,T]\times\mathbb R_+^2}}\\
	 &\leq &C\norm{\phi^{0}_{m-j+1}W^0_{m-j+1} \partial_y   f_{m-j+1}}_{L^2\inner{[0,T]\times\mathbb R_+^2}}+C\norm{\comii y^{-1} \phi^{0}_{m-j}W^0_{m-j}   \partial_x^{m-j }\omega }_{L^2\inner{[0,T]\times\mathbb R_+^2}}\\
	 &&+C\norm{\comii y^{-1}\phi^{0}_{m-j}W^0_{m-j}   \partial_x^{m-j }u }_{L^2\inner{[0,T]\times\mathbb R_+^2}},
\end{eqnarray*}
the last inequality using  \reff{phi} and \reff{wmi}.   Combining these inequalities, we conclude 
\begin{eqnarray*}
	&&\norm{\phi^{0}_{m-j+1}W^0_{m-j+1} \partial_y f_{m-j}}_{L^2\inner{[0,T]\times\mathbb R_+;~L^\infty(\mathbb R_x)}}\\
	&\leq &C\norm{\phi^{0}_{m-j}W^0_{m-j} \partial_y   f_{m-j}}_{L^2\inner{[0,T]\times\mathbb R_+^2}}+C\norm{\phi^{0}_{m-j+1}W^0_{m-j+1} \partial_y   f_{m-j+1}}_{L^2\inner{[0,T]\times\mathbb R_+^2}}\\
	 &&+C\norm{\comii y^{-1} \phi^{0}_{m-j}W^0_{m-j}   \partial_x^{m-j }\omega }_{L^2\inner{[0,T]\times\mathbb R_+^2}}+C\norm{\comii y^{-1}\phi^{0}_{m-j}W^0_{m-j}   \partial_x^{m-j }u }_{L^2\inner{[0,T]\times\mathbb R_+^2}}\\
	&\leq &C\norm{\partial_y \phi^{0}_{m-j}W^0_{m-j}    f_{m-j}}_{L^2\inner{[0,T]\times\mathbb R_+^2}}+C\norm{ \partial_y  \phi^{0}_{m-j+1}W^0_{m-j+1} f_{m-j+1}}_{L^2\inner{[0,T]\times\mathbb R_+^2}}\\
	 &&+C\norm{  \phi^{0}_{m-j}W^0_{m-j}  f_{m-j }}_{L^2\inner{[0,T]\times\mathbb R_+^2}}+C\norm{  \phi^{0}_{m-j+1}W^0_{m-j+1}  f_{m-j +1}}_{L^2\inner{[0,T]\times\mathbb R_+^2}},
\end{eqnarray*}
where the last inequality follows from  \reff{equ1+} and \reff{comfm}.    This, along with the inductive assumptions  \eqref{2.6},  yields, if  $\com{m/2}+1\leq j\leq m-4$ then
   \begin{eqnarray*}
	&&\norm{\phi^{0}_{m-j+1}W^0_{m-j+1} \partial_y f_{m-j}}_{L^2\inner{[0,T]\times\mathbb R_+;~L^\infty(\mathbb R_x)}}\\
	&\leq &CA^{m-j-5  }\inner{(m-j-5)! }^{3(1+\sigma)}+CA^{m-j-4  }\inner{(m-j-4)! }^{3(1+\sigma)}\\
	&\leq & CA^{m-j-4  }\inner{(m-j-4)! }^{3(1+\sigma)}, 
\end{eqnarray*}
and  if $  j=  m-3$ then 
  \begin{eqnarray*}
	\norm{\phi^{0}_{m-j+1}W^0_{m-j+1} \partial_y f_{m-j}}_{L^2\inner{[0,T]\times\mathbb R_+;~L^\infty(\mathbb R_x)}}\leq C  
\end{eqnarray*}
due to the initial hypothesis of induction \reff{15042225}.  On the other hand,    the inductive assumptions  \eqref{2.6} yields, for any $\com{m/2}+1\leq j\leq m-3,$
\begin{eqnarray*}
	\norm{\phi^{0}_{j+2} W^0_{j+2}f_{j+2} }_{L^\infty\inner{[0,T];~L^2(\mathbb R_+^2)}}\leq A^{j-3  }\inner{(j-3)! }^{3(1+\sigma)}.
\end{eqnarray*}
Putting these estimates  into \reff{s2+},  we have
\begin{eqnarray*}
S_2 
&\leq & C\sum_{j=\com{m/2}+1}^{m-4}  \frac{m!}{j!(m-j)!}   A^{j-3  }\big((j-3)!\big)^{3(1+\sigma)}  \bigg(A^{m-j-4  }\inner{(m-j-4)! }^{3(1+\sigma)}\bigg) \\
&&+C\sum_{j=m-3}^{m-3}  \frac{m!}{j!(m-j)!}   A^{j-3  }\big((j-3)!\big)^{3(1+\sigma)} \\
&\leq &C \sum_{j=\com{m/2}+1}^{m-4} \frac{m!}{j^3 (m-j)^4}   A^{m-7} \big((j-3)!\big)^{3(1+\sigma)-1}   \inner{(m-j-4)! }^{3(1+\sigma)-1
}\\
&&+C    A^{m-6  }\big((m-5)!\big)^{3(1+\sigma)} \\
&\leq & C \sum_{j=\com{m/2}+1}^{m-4} \frac{m! }{ (m/2)^3(m-j)^4}   A^{m-7} \inner{(m-7)!}^{3(1+\sigma)-1}+C    A^{m-6  }\big((m-5)!\big)^{3(1+\sigma)}\\
&\leq & C  (m-3)!    A^{m-7} \inner{(m-7)!}^{3(1+\sigma)-1}+C    A^{m-6  }\big((m-5)!\big)^{3(1+\sigma)} \\
&\leq &C    A^{m-6  }\inner{(m-5)! }^{3(1+\sigma)}.
\end{eqnarray*}
This  along with \reff{s1} and \reff{s1s2}  yields 
\begin{eqnarray*}
\sum_{j=1}^{m-3}{m\choose j}\norm{ \phi^{\ell}_m W^0_{m}(\p^{j+1} v)\partial_y f_{m-j}}_{L^2\inner{[0,T]\times\mathbb R_+^2}}\leq C    A^{m-6  }\inner{(m-5)! }^{3(1+\sigma)}. 
\end{eqnarray*}
Similarly, we have
\begin{eqnarray*}
\sum_{j=1}^{m-3}{m\choose j}\norm{ \phi^{\ell}_m W^0_{m}(\p^{j} v)\partial_y f_{m-j}}_{L^2\inner{[0,T]\times\mathbb R_+^2}}\leq C    A^{m-6  }\inner{(m-5)! }^{3(1+\sigma)}. 
\end{eqnarray*}
Then  the desired estimate \reff{mj1} follows.

\medskip
{\it (c)} It remains to prove that 
\begin{eqnarray}\label{fiest}
	\sum_{j=1}^{m-3}{m\choose j}\norm{ \phi^{\ell}_m W^0_{m}   (\p^j v)(\partial_y \partial_x f_{m-j})}_{L^2\inner{[0,T]\times\mathbb R_+^2}}\leq C    A^{m-6  }\inner{(m-5)! }^{3(1+\sigma)}.
\end{eqnarray}
The proof is quite similar as in the previous step.  To do so we first 
write
\begin{eqnarray*}
	&& \sum_{j=1}^{m-3}{m\choose j}\norm{ \phi^{\ell}_m W^0_{m}   (\p^j v)(\partial_y \partial_x f_{m-j})}_{L^2\inner{[0,T]\times\mathbb R_+^2}}\\
&= &~\sum_{j=1}^{\com{m/2}}{m\choose j}\norm{ \phi^{\ell}_m W^0_{m}(\p^{j} v)(\partial_y\partial_x   f_{m-j})}_{L^2\inner{[0,T]\times\mathbb R_+^2}}\\
&&+\sum_{j=\com{m/2}+1}^{m-3}{m\choose j}\norm{ \phi^{\ell}_m W^0_{m}(\p^{j} v)(\partial_y\partial_x f_{m-j})}_{L^2\inner{[0,T]\times\mathbb R_+^2}}\\
&=&~\tilde S_1+\tilde S_2.
\end{eqnarray*}
For the term $\tilde S_1$,  we use  
	\begin{eqnarray*}
\phi^{\ell}_m\leq \phi^{0}_m \leq \phi^{0}_{j+2}\phi^{0}_{m-j+1}, \quad  W^0_{m}\leq W^0_{m-j+1}~{\rm for}~j\geq 2,
\end{eqnarray*}
to obtain
\begin{eqnarray*}
	\tilde S_1 &\leq &\sum_{j=1}^{\com{m/2}}{m\choose j}\norm{\phi^{0}_{j+2} \p^{j} v }_{L^\infty\inner{[0,T]\times\mathbb R_+^2}}\norm{\phi^{0}_{m-j+1}W^0_{m-j+1} \partial_y \partial_xf_{m-j}}_{L^2\inner{[0,T]\times\mathbb R_+^2}}.
\end{eqnarray*}
Then repeating the arguments used to estimate  $S_1$ and $S_2$ in the previous step,  we 
can deduce that
\begin{eqnarray*}
	\tilde S_1\leq   C  A^{m-6} \inner{(m-6)!}^{3(1+\sigma)}. 
\end{eqnarray*}
As for $\tilde S_2$,  using the inequality
\begin{eqnarray*}
\phi^{\ell}_m\leq \phi^{0}_m \leq \phi^{0}_{j+1}\phi^{0}_{m-j+2}, \quad  W^0_{m}\leq W^0_{m-j+2}~{\rm for}~j\geq 2,
\end{eqnarray*}
gives
\begin{eqnarray*}
\tilde S_2 
\leq   \sum_{j=\com{m/2}+1}^{m-3} {m\choose j}\norm{\phi^{0}_{j+1} \p^{j} v }_{L^\infty\inner{[0,T]\times\mathbb R_+;~L^2(\mathbb R_x)}}\norm{\phi^{0}_{m-j+2}W^0_{m-j+2} \partial_y\partial_x f_{m-j}}_{L^2\inner{[0,T]\times\mathbb R_+;~L^\infty(\mathbb R_x)}}.\end{eqnarray*}
 Then repeating the arguments used to estimate $S_2$ in the previous step, we have
\begin{eqnarray*}
	\tilde S_2\leq   C    A^{m-6  }\inner{(m-5)! }^{3(1+\sigma)}.
\end{eqnarray*}
This along with the estimate on $\tilde S_1$ yields  \reff{fiest}.   Finally,  combining \reff{km1}, \reff{1m1}, \reff{mj1} and \reff{fiest} gives the desired estimate in Lemma \ref{lem1}, and thus the proof is   complete.  
\end{proof}

\begin{lemma}\label{lem2}
Under the same assumption as in Proposition \ref{nonli},  we have
\begin{eqnarray*}
&&\sum_{j=1}^{m}{m\choose j}\norm{  \Lambda_\delta^{-2}   \phi^{\ell}_mW^\ell_m(\p^j u)f_{m+1-j}}_{L^2\inner{[0,T]\times\mathbb R_+^2}}\\&&\qquad+\sum_{j=1}^{m-1}{m\choose j}
\norm{  \Lambda_\delta^{-2}   \phi^{\ell}_mW^\ell_m\left[\partial_y\inner{ \partial_y\omega/\omega}\right](\p^j v)(\p^{m-j}u)}_{L^2\inner{[0,T]\times\mathbb R_+^2}}\\
&\leq &m\,C    \norm{  \Lambda_\delta^{-2} \phi^{\ell}_m W^\ell_m f_m}_{L^2\inner{[0,T]\times\mathbb R_+^2}}+ C  A^{m-6}  \inner{(m-5)!}^{3(1+\sigma)}.
\end{eqnarray*}
\end{lemma}
 The proof of this Lemma is quite similar as in Lemma \ref{lem1}, so we omit it. 

\begin{lemma}\label{lem3}
Under the same assumption as in Proposition \ref{nonli},  we have
	\begin{eqnarray*}
 2\norm{  \Lambda_\delta^{-2}  \phi^{\ell}_m
W^\ell_m\left[\partial_y\inner{ (\partial_y\omega)/\omega}\right]f_m}_{L^2\inner{[0,T]\times\mathbb R_+^2}} \leq  C\norm{  \Lambda_\delta^{-2}  \phi^{\ell}_m
W^\ell_m f_m}_{L^2\inner{[0,T]\times\mathbb R_+^2}}.
	\end{eqnarray*}

\end{lemma}

\begin{proof}
	This is a just direct verification.  Indeed,   Lemma \ref{lemma3.1} gives
	\begin{eqnarray*}
	&&	\norm{  \Lambda_\delta^{-2}  \phi^{\ell}_m
W^\ell_m\left[\partial_y\inner{ (\partial_y\omega)/\omega}\right]f_m}_{L^2\inner{[0,T]\times\mathbb R_+^2}}\\
&\leq &\norm{\left[\partial_y\inner{ (\partial_y\omega)/\omega}\right]   \Lambda_\delta^{-2}  \phi^{\ell}_m
W^\ell_mf_m}_{L^2\inner{[0,T]\times\mathbb R_+^2}}+\norm{\com{ \partial_y\inner{ (\partial_y\omega)/\omega},~  \Lambda_\delta^{-2}  
W^\ell_m}\phi^{\ell}_m f_m}_{L^2\inner{[0,T]\times\mathbb R_+^2}}\\
&\leq &C\norm{   \Lambda_\delta^{-2}  \phi^{\ell}_m
W^\ell_mf_m}_{L^2\inner{[0,T]\times\mathbb R_+^2}}.
	\end{eqnarray*}
Then the desired estimate follows and thus the proof of Lemma \ref{lem3} is complete. 
\end{proof}

\begin{proof}[Proof of Proposition \ref{prop5.1}]
	 In view of \reff{adxzml},  we combine  the estimates in  Lemma \ref{lem1}-Lemma \ref{lem3},  to get the first estimate in Proposition \ref{prop5.1}.   The second one can be treated quite similarly and the main difference is that we will use here additionally the inductive estimates on the terms of the following form 
\begin{eqnarray*}
	\norm{\partial_y^2\Lambda^{-2/3}\phi^{0}_j  W^{0}_j f_j}_{L^2([0,T]\times\RR^2_+)},\quad 6\leq j\leq m,
\end{eqnarray*}
 while in the proof  of   Lemma \ref{lem1},   we only used the  estimates on the following two forms
\begin{eqnarray*}
	\norm{ \phi^{0}_j  W^{0}_j f_j}_{L^\infty([0,T];~L^2(\RR^2_+))},\quad \norm{\partial_y \phi^{0}_j  W^{0}_j f_j}_{L^2([0,T]\times\RR^2_+)},\quad  6 \leq j\leq m.
\end{eqnarray*} 
So we omit the treatment of the second estimate for brevity,  and  thus the proof of Proposition \ref{prop5.1} is complete. 
\end{proof}

\begin{proof}[{\bf Completeness of  the proof of Proposition \ref{nonli}}]
The estimates \reff{2.20}   follows from the combination of Propostion \ref{+prp5.1+} and the first estimate in Proposition \ref{prop5.1},  while the  estimate \eqref{2.21} in Proposition \ref{nonli} follows from  Propostion \ref{+prp5.1+} and the second estimate in Proposition \ref{prop5.1}.   The treatment of   \reff{+2.20++} is exactly the same as   \reff{2.20}.   The proof of Proposition \ref{nonli} is thus complete. 
\end{proof}


\section{Appendix}\label{apen}
Here we deduce the equation fulfilled by $f_m$ (cf.  \cite{masmoudi-1}).  
 Recall that 
\begin{equation*}
f_m=\partial_x^m \omega-\frac{\partial_y \omega}{\omega}\partial_x^m u,\quad m\geq 1,
\end{equation*}
where $u$ is a smooth solution to Prandtl equation \reff{prandtl1} and $\omega=\partial_y u$.  We will verify that   
\begin{eqnarray}\label{eqnfms}
\partial_t f_m+u\p f_m+v\partial_y f_m-\partial_y^2 f_m=\mathcal{Z}_m, 
\end{eqnarray}
where
\begin{eqnarray*}
\mathcal{Z}_m&=&-\sum_{j=1}^m{m\choose j}(\p^j u) f_{m+1-j}
-\sum_{j=1}^{m-1}{m\choose j}(\p^j v)(\partial_y f_{m-j})\\
&&-\left[\partial_y\inner{\frac{\partial_y\omega}{\omega}}\right]\sum_{j=1}^{m-1}{m\choose j}
(\p^j v)(\p^{m-j}u)
-2\left[\partial_y\inner{\frac{\partial_y\omega}{\omega}}\right]f_m.
\end{eqnarray*}
To do so,   we firstly notice that  
\begin{eqnarray}\label{ueqy}
	 u_t + u u_x + vu_y 
- u_{yy}=0, 
\end{eqnarray}
and 
\begin{eqnarray*} 	
 \omega_t + u \omega _x + v\omega _y - \omega_{yy}=0. 
\end{eqnarray*}
Thus by Leibniz's formula, 
$\p^m u, \partial_x^m \omega$  satisfy, respectively,  the following equation
\begin{eqnarray}\label{eqel2}
\begin{split}
&\partial_t\p^m u+u\p\p^m u+v\partial_y\p^m u-\partial_y^2\p^m u\\
=&-\sum_{j=1}^{m}{m\choose j}(\p^j u)(\p^{m-j+1}u)-\sum_{j=1}^{m}(\p^j v)(\partial_y\p^{m-j}u)\\
=&-\sum_{j=1}^{m}{m\choose j}(\p^j u)(\p^{m-j+1}u)-\sum_{j=1}^{m-1}(\p^j v)(\partial_y\p^{m-j}u)
-(\p^m v)(\partial_y u)
\end{split}
\end{eqnarray}
and
\begin{eqnarray}\label{eqel1}
\begin{split}
&\partial_t\p^m\omega+u\p\p^m\omega+v\partial_y\p^m\omega-\partial_y^2\p^m\omega\\
=&-\sum_{j=1}^{m}{m\choose j}(\p^j u)(\p^{m-j+1}\omega)-\sum_{j=1}^{m}(\p^j v)(\partial_y\p^{m-j}\omega)\\
=&-\sum_{j=1}^{m}{m\choose j}(\p^j u)(\p^{m-j+1}\omega)-\sum_{j=1}^{m-1}(\p^j v)(\partial_y\p^{m-j}\omega)
-(\p^m v)(\partial_y\omega).
\end{split}
\end{eqnarray}
In order to  eliminate the last terms on the right sides of the above two equations,  we  observe $\partial_y u=\omega>0$ and thus  multiply \reff{eqel2} by  $-\frac{\partial_y\omega}{\omega}$,  and then add the resulting equation to \reff{eqel1};  this gives  
\begin{eqnarray*}
\partial_t f_m+u\p f_m+v\partial_y f_m-\partial_y^2 f_m=\mathcal{Z}_m
\end{eqnarray*}
where
\begin{eqnarray*}
\mathcal{Z}_m&=&-\sum_{j=1}^m{m\choose j}(\p^j u) f_{m+1-j}
-\sum_{j=1}^{m-1}{m\choose j}(\p^j v)(\partial_y f_{m-1})\\
&&-\bigg[\partial_y\inner{\frac{\partial_y\omega}{\omega}}\bigg]\sum_{j=1}^{m-1}{m\choose j}
(\p^j v)(\p^{m-j}u)+(\p^m u)f_1\\
&&+\inner{\partial_t\inner{\frac{\partial_y\omega}{\omega} }
+u\p\inner{\frac{\partial_y\omega}{\omega} }
+v\partial_y\inner{\frac{\partial_y\omega}{\omega} }
-\partial^2_y\inner{\frac{\partial_y\omega}{\omega} } }\p^m u\\
&&-2\left[\partial_y\inner{\frac{\partial_y\omega}{\omega} }\right]\partial_y\p^m u.
\end{eqnarray*}
On the other hand we notice that
\begin{eqnarray*}
&&\partial_t\inner{\frac{\partial_y\omega}{\omega} }
+u\p\inner{\frac{\partial_y\omega}{\omega} }
+v\partial_y\inner{\frac{\partial_y\omega}{\omega} }
-\partial^2_y\inner{\frac{\partial_y\omega}{\omega} }\\
&=&\frac{1}{\omega}\inner{\partial_t\partial_y\omega
+u\p\partial_y\omega+v\partial_y\partial_y\omega
-\partial^2_y\partial_y\omega }\\
&&-\frac{\partial_y\omega}{\omega^2}
\inner{\partial_t\omega+u\p\omega+v\partial_y\omega-
\partial_y^2\omega }
+2\frac{\partial_y\omega}{\omega}\partial_y\inner{\frac{\partial_y\omega}{\omega} }\\
&=&-\p\omega+\frac{(\p u)(\partial_y\omega)}{\omega}+
2\frac{\partial_y\omega}{\omega}\partial_y\inner{\frac{\partial_y\omega}{\omega} }
\end{eqnarray*}
Therefore we have
\begin{eqnarray*}
\mathcal{Z}_m&=&-\sum_{j=1}^m{m\choose j}(\p^j u) f_{m+1-j}
-\sum_{j=1}^{m-1}{m\choose j}(\p^j v)(\partial_y f_{m-1})\\
&&-\left[\partial_y\inner{\frac{\partial_y\omega}{\omega}}\right]\sum_{j=1}^{m-1}{m\choose j}
(\p^j v)(\p^{m-j}u)+(\p^m u)f_1\\
&&+\inner{\p\omega-\frac{(\p u)(\partial_y\omega)}{\omega} }\p^m u
+2\frac{\partial_y\omega}{\omega}\partial_y\inner{\frac{\partial_y\omega}{\omega} }
\p^m u-2\left[\partial_y\inner{\frac{\partial_y\omega}{\omega} }\right]\partial_y\p^m u\\
&=&-\sum_{j=1}^m{m\choose j}(\p^j u) f_{m+1-j}
-\sum_{j=1}^{m-1}{m\choose j}(\p^j v)(\partial_y f_{m-1})\\
&&-\left[\partial_y\inner{\frac{\partial_y\omega}{\omega}}\right]\sum_{j=1}^{m-1}{m\choose j}
(\p^j v)(\p^{m-j}u)+
\left[\partial_y\Big(\frac{\partial_y\omega}{\omega}\Big)^2 \right]
\p^m u\\
&&-2\left[\partial_y\inner{\frac{\partial_y\omega}{\omega} }\right]\p^m\omega\\
&=&-\sum_{j=1}^m{m\choose j}(\p^j u) f_{m+1-j}
-\sum_{j=1}^{m-1}{m\choose j}(\p^j v)(\partial_y f_{m-1})\\
&&-\left[\partial_y\inner{\frac{\partial_y\omega}{\omega}}\right]\sum_{j=1}^{m-1}{m\choose j}
(\p^j v)(\p^{m-j}u)
-2\left[\partial_y\inner{\frac{\partial_y\omega}{\omega} }\right]f_m.
\end{eqnarray*}

\bigskip
Next we will give the boundary value of $\partial_y f_m$ and $\partial_t f_m-\partial_y^2 f_m$.   In view of \reff{ueqy},  we infer,  recalling $u|_{y=0}=v|_{y=0}=0,$ 
\begin{eqnarray*}
	\partial_y\omega\big|_{y=0}= \partial_y^2 u\big|_{y=0}=0. 
\end{eqnarray*} 
As a result,   observing 
\begin{eqnarray*}
\partial_y f_m=\partial_y\p^m\omega-\left[\partial_y\left(\frac{\partial_y\omega}{\omega}\right)\right]\p^m u-\left(\frac{\partial_y\omega}{\omega}\right)\partial_y\p^m u, 
\end{eqnarray*}
we have
\begin{eqnarray}\label{15042201}
\partial_y f_m|_{y=0}=0.
\end{eqnarray}
Direct verification shows
\begin{eqnarray*}
\mathcal{Z}_m|_{y=0}=-2 \left[\partial_y\inner{\frac{\partial_y\omega}{\omega} }\right]f_m \Big|_{y=0},
\end{eqnarray*}
and thus 
\begin{eqnarray}\label{15042202}
\left(\partial_t f_m-\partial_y^2 f_m\right)|_{y=0}=\mathcal{Z}_m|_{y=0}=-2 \left[\partial_y\inner{\frac{\partial_y\omega}{\omega} }\right]f_m \Big|_{y=0},
\end{eqnarray}
due to the equation fulfilled by $f_m.$

\section*{Acknowledgments}
W. X. Li was supported  by NSF of China(No. 11422106) and C.-J. Xu was partially supported by `` the Fundamental Research
Funds for the Central Universities'' and
the NSF of China (No. 11171261).

\end{document}